\documentclass{article}

\usepackage[english]{babel}
\usepackage{amsmath}
\usepackage{amssymb}
\usepackage{amsthm}
\usepackage[shortlabels]{enumitem}
\setlist[enumerate,1]{label={(\alph*)}}
\usepackage{mathrsfs}
\usepackage{mathtools}

\newtheorem{theorem}{Theorem}[section]
\newtheorem*{theorem*}{Theorem}
\newtheorem{corollary}{Corollary}[theorem]
\newtheorem{lemma}[theorem]{Lemma}
\newtheorem{proposition}{Proposition}[section]

\theoremstyle{remark}
\newtheorem{remark}{Remark}

\theoremstyle{definition}
\newtheorem{definition}[theorem]{Definition}

\theoremstyle{remark}
\newtheorem*{example}{Example}
\theoremstyle{definition}

\theoremstyle{remark}

\theoremstyle{remark}
\newtheorem{conjecture}{Conjecture}

\usepackage{comment}

\usepackage{verbatim}
\usepackage{graphicx}
\usepackage{stmaryrd}

\usepackage{xy}
\usepackage{amsthm}

\xyoption{all}

\usepackage{graphicx}
\usepackage{tikz}
\usetikzlibrary{cd}
\usepackage{commath}
\usepackage{wasysym} %For frowny face: \frownie{} (also, \smiley{} and \blacksmiley{} )

\usepackage[margin=2.54cm]{geometry}
\usepackage[utf8x]{inputenc}

\usepackage{url}

 \DeclareMathOperator{\Span}{Span}
 \DeclareMathOperator{\Ker}{Ker}

 \DeclareMathOperator{\Gr}{Gr}

 \DeclareMathOperator{\codim}{codim}

\newcommand{\bC}{\mathbb{C}}

\newcommand{\bP}{\mathbb{P}}
\newcommand{\bQ}{\mathbb{Q}}
\newcommand{\bR}{\mathbb{R}}
\newcommand{\bZ}{\mathbb{Z}}

\newcommand{\cA}{\mathcal{A}}

\newcommand{\cC}{\mathcal{C}}
\newcommand{\cD}{\mathcal{D}}
\newcommand{\cF}{\mathcal{F}}
\newcommand{\cG}{\mathcal{G}}
\newcommand{\cH}{\mathcal{H}}
\newcommand{\cI}{\mathcal{I}}

\newcommand{\cP}{\mathcal{P}}

\newcommand{\cT}{\mathcal{T}}

\newcommand{\fC}{\mathfrak{C}}

\newcommand{\fs}{\mathfrak{s}}
\newcommand{\ft}{\mathfrak{t}}
\newcommand{\fu}{\mathfrak{u}}

\newcommand{\sfT}{\mathsf{T}}
\newcommand{\sfx}{\mathsf{x}}

\newcommand{\tto}{\longrightarrow}

\newcommand{\mapstto}{\longmapsto}
\newcommand{\mapsffrom}{\longmapsfrom}

\newcommand{\wtilde}{\widetilde}

\title{Compactifying the Parameter Space for the Quantum Multiplication for Hypertoric Varieties.}
\author{Jeremy Peters}
\date{\today}

\begin{document}

\maketitle

\begin{abstract}
    In this paper, we will be studying the parameter space for the quantum multiplication for hypertoric varieties.  The operation of quantum multiplication for hypertoric varieties has an explicit formulation which is given by McBreen and Shenfeld.  In particular, this multiplication depends on a parameter which lives in the complement of a toric arrangement.  Following a paper of deConcini and Gaiffi, I will define a compactification of this parameter space and show how the quantum multiplication can be extended to this compactification.  
\end{abstract}

\tableofcontents

\section{Introduction}

In this paper, we will be working with hypertoric varieties.  A {\it Hypertoric variety} is an algebraic symplectic variety which is constructed as the Hamiltonian reduction of $T^*\mathbb{C}^n{/\!\!/\!\!/}_{0,\chi}T^k$ by $T^k$, where $T^k$ acts on $T^*\mathbb{C}^n$ by an inclusion $T^k\hookrightarrow T^n$; where $T^n$ acts on $T^*\mathbb{C}^n$ coordinate wise with a weight--$1$ action on the base and a weight--$-1$ action on the fibre; and where $\chi\in X^*(T^k)$ and $0\in \mathfrak{t}^k$.  The precise definition is given in Definition \ref{D2.1.1}.  For an overview of hypertoric varieties, see~\cite{P07}. 

Hypertoric varieties are special examples of conical symplectic resolutions.  A {\it symplectic resolution} is a morphism $\pi\colon X\longrightarrow X_0$, where $X$ is a smooth symplectic algebraic variety over $\mathbb{C}$ with symplectic form $\omega$, $X_0$ is its affinization, which is normal and Poisson, and $\pi$ is the canonical map, which is birational, projective and Poisson.  A symplectic resolution is said to be {\it conical} if it comes with additional $\mathbb{C}^\mathsf{x}$--actions on $X$ and on $X_0$, which commute with $\pi$, and are such that $\mathbb{C}^\mathsf{x}$ contracts $X_0$ to a point and scales the symplectic form $\omega$ on $X$ with positive weight.  Symplectic resolutions were originally studied by Beauville~\cite{B00}, Namikawa~\cite{N01a,N01b,N01c,N03,N04,N05}, and Kaledin~\cite{K06}.  In~\cite{BLPW16,BPW16}, it was conjectured that these symplectic resolutions should come in dual pairs such that certain properties are interchanged.  This phenomenon is known as {\it symplectic duality}.  Aside from hypertoric varieties, other examples of conical symplectic resolutions include cotangent bundles to flag varieties, quiver varieties and slices in the affine Grassmannian.  In this paper, we will consider symplectic resolutions that come equipped with an additional Hamiltonian group action by a torus $T^d$, which commute with the respective conical $\bC^\sfx$--actions on $X$ and $X_0$, and are such that $\pi$ is $T^d$--equivariant and such that the fixed point set $X^{T^d}$ is finite.  See~\cite{K22} for an overview. 

One invariant of interest is the quantum equivariant cohomology of a conical symplectic resolution, the study of which was originated in~\cite{BMO11}.  If $X$ is a symplectic resolution equipped with a Hamiltonian $T^d$--action, and a conical $\mathbb{C}^\mathsf{x}$--action, we obtain an action of $G = T^d\times\mathbb{C}^\mathsf{x}$ on $X$ and can study its $G$--equivariant cohomology $H_G^\bullet(X,\bC)$.  We define the $G$--{\it equivariant quantum cohomology} $QH_G^\bullet(X,\bC)$ of a symplectic resolution $X$ to be the commutative, associative deformation of $H_G^\bullet(X,\bC)$ by a formal power series in effective curve classes, defined as follows:
\begin{align*}
    \langle \gamma_1\star_q\gamma_2, \gamma_3\rangle = \sum_{\beta\in H_2(X,\mathbb{Z})_{\mathrm{eff}}}\langle \gamma_1,\gamma_2,\gamma_3\rangle_{0,3,\beta}\;q^\beta.
\end{align*}
Here, $\gamma_1,\gamma_2,\gamma_3\in H_G^\bullet(X,\bC)$; $\langle -,-\rangle$ is the Poincar\'{e} pairing on $H_G^\bullet(X,\bC)$; and $\langle -,-,-\rangle_{0,3,\beta}$ is the genus-$0$, degree-$3$ Gromov-Witten invariants valued in $H_G^\bullet(\mathrm{pt},\bC)$.  Also, $\beta$ ranges over the cone of effective curve classes in $H_2(X,\mathbb{Z})_{\mathrm{eff}}\subset H_2(X,\mathbb{Z})\simeq X^*(T^k)$; $q$ can be viewed as a coordinate on $H^2(X,\mathbb{C}^\mathsf{x}) = T^k$, with $q^\beta$ a formal symbol (with the relation $q^\beta q^{\beta'} = q^{\beta+\beta'}$); and $\omega$ is the symplectic form on $X$.  In what follows, we will assume our (equivariant) cohomology rings to have complex coefficients, unless otherwise specified.

Given $u\in H_G^2(X)$, the quantum multiplication operator $u\star_q(-)$ defines an endomorphism operator of $H_G^\bullet(X)$, which is linear with respect to the $H_G^\bullet(\mathrm{pt})$--module structure on $H_G^\bullet(X)$ induced by the natural map $H_G^\bullet(\mathrm{pt})\longrightarrow H_G^\bullet(X)$.  Let $E :=  \mathrm{End}_{H_G^\bullet(\mathrm{pt})}H_G^\bullet(X)$.

It is conjectured by Okounkov in Section 2.3.4 of~\cite{O17} that the quantum multiplication takes the following form:
\begin{align*}
    u\star_q(-) = u\cup(-) + \sum_{\beta\in \Phi^+}\frac{q^\beta}{1-q^\beta}\langle \beta,u\rangle\;\hbar L_\beta(-)\in E,
\end{align*}
where $\Phi^+\subset X^*(T^k)\simeq (\ft_\bZ^k)^*$ is a finite set of {\it positive K\"{a}hler roots} and where $L_\beta(-)$ is a {\it Steinberg operator}, constructed from each $\beta\in\Phi^+$ in the following manner.  Beginning with a symplectic resolution $X\rightarrow X_0$ with a $G$--action, described as above, the {\it Steinberg variety} is the product $X\times_{X_0} X\subset X\times X$.  The irreducible components of $X\times_{X_0} X$ define $G$--equivariant Borel-Moore homology classes $[Z_i]\in H_{2\mathrm{dim}_\bR(X)}^{\mathrm{BM},G}(X\times_{X_0} X)$, which are called {\it Steinberg correspondences}.  These correspondences define endomorphisms
\begin{align*}
    H_G^\bullet(X)\xrightarrow[]{p_1^*}H_G^\bullet(Z_i)\xrightarrow[]{p_{2,*}}H_G^\bullet(X)
\end{align*}
via pullback and pushforward along the projection maps 
\begin{align*}
    p_{1,2}\colon Z_i\hookrightarrow X\times_{X_0}X\longrightarrow X.
\end{align*}
The Steinberg operators $L_\beta(-)\in E$ are then given by these endomorphisms. 

In the case of this conjecture, the parameter $q$ can be specialized to a parameter in $T^k$ away from the {\it discriminantal arrangement}:
\begin{align*}
    T^\mathrm{reg}:= T^k\setminus \bigcup_{\beta\in \Phi^+}\left\{t\in T^k\mid \beta(t) = 1\right\}.
\end{align*} 

Details on the computation of the quantum multiplication are given by~\cite{MO12} for Nakajima quiver varieties and by~\cite{BMO11} for the Springer Resolution.  Following up on these examples, the quantum multiplication for slices in the Affine Grassmannian for simply-laced groups was computed in~\cite{D24,D20} and for hypertoric varieties, the computation was carried out in~\cite{MBS13}.  So in the foregoing examples, Okounkov's conjecture has been verified.  

For $X$ a conical symplectic resolution, it is of interest to study the map 
\begin{align*}
    Q\colon T^\mathrm{reg}&\longrightarrow \left\{\mathrm{Subspaces\;spanned\;by\;commuting\;elements\;of\;}E\right\}\\
    q&\longmapsto \left\{u\star_q(-)\mid u\in H^2(X)\right\},
\end{align*}
and extend this map to a compactification of the parameter space $T^\mathrm{reg}$.  A related question is to study the algebra $\mathcal{A}(q)\subset E$ given by the action of the entire quantum cohomology algebra $QH_G^\bullet(X)_q$ on $H_G^\bullet(X)$ for a given $q\in T^\mathrm{reg}$ and extend this full action to a compactification of $T^\mathrm{reg}$, but this will not be studied here.  If $X$ is a symplectic resolution whose cohomology $H_G^\bullet(X)$ is generated by divisor classes $u\in H_G^2(X)$, then the family $Q(q)$ is generated by the quantum multiplication by divisor classes and it follows that for each fixed $q\in T^\text{reg}$, the algebra $\mathcal{A}(q)$ is generated by $Q(q)$.  Thus, the question of extending $\mathcal{A}$ to $\widetilde{X}$ reduces to the question of studying the extension of $Q$ to $\widetilde{X}$.  If $X$ is a hypertoric variety, then Theorem \ref{T2.1.1} implies that the equivariant cohomology $H_G^\bullet(X)$ is generated by divisor classes, so the foregoing remarks are true.   

In the case where $X$ is a Nakajima quiver variety corresponding to a simply-laced Lie algebra $\mathfrak{g}$ of ADE type, the equivariant cohomology $H_G^\bullet(X)$ forms a representation of the {\it Yangian algebra}, $\mathsf{Y}(\mathfrak{g})\longrightarrow E$.  The subspace $Q(q)$ is found to be generated by certain {\it trigonometric Casimir Hamiltonians} in $\mathsf{Y}(\mathfrak{g})$ and it is conjectured that the full algebra $\mathcal{A}(q)$ is given by the image of a family of Bethe subalgebras in $E$~\cite{MO12}.  This conjecture is true in the case where $X$ is a partial flag variety in type A~\cite{G13}. 

In the case where $X$ is a slice in the affine Grassmannian $\mathrm{Gr}_G$, where $G$ is of ADE type, the equivariant cohomology $H_G^\bullet(X)$ is given in terms of weight modules of $\mathfrak{g}$ via~\cite{GR13}.  The operation of quantum multiplication on the specialized equivariant cohomology space $H_{\theta,1}^\bullet(X)$--where $\theta\in\mathfrak{h}^*$ is generic--is given by the action of {\it trigonometric Gaudin Hamiltonians} in $((U\mathfrak{g})^{\otimes n})^{\mathfrak{h}}$~\cite{D24,D20,IKR25}.  The algebra $QH_{\theta,1}^\bullet(X)$ is conjecturally described by the action of the family of {\it trigonometric Gaudin subalgebras}, which is the maximal commutative subalgebra of $(U\mathfrak{g}^{\otimes n})^\mathfrak{h}$ containing the trigonometric Gaudin Hamiltonians~\cite{IKLPR24, IKR25}.  The parameter space for this family of algebras is $((\mathbb{C}^\mathsf{x})^n\setminus \Delta)/\mathbb{C}^\mathsf{x}$, which compactifies to the Deligne-Mumford space $\overline{M}_{0,n+2}$ and it was shown in~\cite{IKR25} that this family of algebras extends to a family indexed by the compactification.  

In this present paper, we will study the family  
\begin{align*}
    Q\colon T^\mathrm{reg}&\longrightarrow \mathrm{Gr}(n', E)\\
    q&\longrightarrow \left\{u\star_q(-)\mid u\in H^2_G(X)\right\}
\end{align*}
for $n' = \mathrm{dim}(H_G^2(X))$ in the case where $X$ is a hypertoric variety.  Our goal is to obtain the following extension:
\begin{center}
    % https://tikzcd.yichuanshen.de/#N4Igdg9gJgpgziAXAbVABwnAlgFyxMJZABgBpiBdUkANwEMAbAVxiRABUA9AHW5xgAeOYACcYAcwC+ISaXSZc+QigBM5KrUYs2vfkOABxEZIAUYAOSkAogEoZckBmx4CRMgEYN9Zq0QheAO5YsHgMsMAAGgD6vADKWOIAtnTSkhowUOLwRKAAZiIQiUhkIDgQSO7U3tp+AIr2eQVFiGql5YiVIAxYYL4gUHRwABYZDSD5hcXUZUgqaZJAA
\begin{tikzcd}
T^\mathrm{reg} \arrow[rr, "Q"] \arrow[d]   &  & {\mathrm{Gr}(n',E)} \\
\widetilde{X_\Sigma}, \arrow[rru, dashed] &  &                  
\end{tikzcd}
\end{center}
for some compactification $\widetilde{X_\Sigma}$ of $T^\mathrm{reg}$.  From~\cite{HP04}, we can write $u_1,\ldots,u_n,\hbar\in H_G^2(X)$ as the generating divisors, so that studying $Q(q)$ reduces to studying each of the quantum multiplication operators $u_i\star_q(-)$, for $i=1,\ldots,n$ (since $\hbar\star_q(-) = \hbar\;\mathbb{I}(-)$).  In order to extend the map $Q$, we can write out the quantum multiplication $u_i\star_q(-)$ in terms of the operators of classical multiplication $u_i\cup(-)$, and the Steinberg operators $L_\alpha(-)$, using Okounkov's conjecture, stated above.  These terms then satisfy the same algebraic relations that define the {\it modified trigonometric holonomy Lie algebra} $\widetilde{\mathfrak{u}_{\Phi^+}}$, which is a graded Lie algebra to be defined in Definition \ref{D2.2.2}.  This Lie algebra generalizes the holonomy Lie algebra, whose relations express the commutativity of the Hamiltonians of Gaudin's integrable system, or equivalently the flatness of the Knizhnik-Zamolodchikov connection.  Families of subspaces of the Holonomy Lie algebra spanned by these Hamiltonians were studied in~\cite{AFV16}.  Similarly, the conditions for flatness of the trigonometric KZ connection of~\cite{D24,D20} for ADE root systems gives rise to the commuting trigonometric Gaudin Hamiltonians of~\cite{IKLPR24,IKR25,KR25}. 

Using this modified trigonometric holonomy Lie algebra, we can define a map $\gamma\colon\widetilde{\mathfrak{u}_{\Phi^+}}\rightarrow E$ and express the quantum multiplications as $u_i\star_q(-) = \gamma(H_i^\mathrm{trig}(q))$, where 
\begin{align*}
    H_i^\mathrm{trig}(q) := u_i + \sum_{\alpha\in\Phi^+}\frac{q^\alpha}{1-q^\alpha}\alpha_it_\alpha,
\end{align*}  
for $i=1,\ldots,n$.  The goal of Section \ref{S2} is to prove the following first Main Theorem concerning this map $\gamma$.

\begin{theorem}\label{T1}
    Let $X = T^*\mathbb{C}^n{/\!\!/\!\!/}_{(0,\chi)}T^k$ be a Hypertoric variety with an action of $G = T^d\times\mathbb{C}^\mathsf{x}$, where $d = n-k$ and define the following map:
    \begin{align*}
        \gamma\colon \widetilde{\mathfrak{u}_{\Phi^+}}&\longrightarrow E\\
        t_\alpha&\longmapsto \hbar L_\alpha\\
        u_i&\longmapsto u_i\cup(-)\\
        \hbar&\longmapsto \hbar\;\mathbb{I}(-),
    \end{align*}
    where $\alpha\in\Phi^+$ and $i=1,\ldots,n$ and let $\widetilde{\mathfrak{u}_{\Phi^+}}^1\subset\widetilde{\mathfrak{u}_{\Phi^+}}$ be the subspace generated by the degree--$1$ elements.  Then $\gamma$ is well-defined on the level of Lie algebras and the restricted map
    \begin{align*}
        \gamma\colon\widetilde{\mathfrak{u}_{\Phi^+}}^1\longrightarrow E
    \end{align*}
    is injective.   
\end{theorem}

To this end, in Section \ref{S2.1}, we will describe preliminary properties to do with hypertoric varieties and their {\it circuits}--combinatorial data which correspond to their K\"{a}hler roots.  In Section \ref{S2.2}, we will demonstrate that $\gamma$ is well-defined, in that the operators in $E$ satisfy the requisite commutation relations which define $\widetilde{\mathfrak{u}_{\Phi^+}}^1$.  In Section \ref{S2.3}, we will be using the technology of {\it Stable Basis}, defined in~\cite{MO12}, to prove that $\gamma$ is injective and that the operators in $E$ are linearly independent.

Once we prove this theorem, we can then write $n' = n+1$ and we obtain the following map:
\begin{align*}
    \gamma_*\colon \mathrm{Gr}(n+1,\widetilde{\mathfrak{u}_{\Phi^+}}^1)\longrightarrow \mathrm{Gr}(n+1,E),
\end{align*}
induced by the inclusion of subspaces.  We will define a projection $\pi\colon \widetilde{\mathfrak{u}_{\Phi^+}}^1\longrightarrow \mathfrak{u}_{\Phi^+}^1$ in Definition \ref{D2.2.2}, which corresponds to taking the quotient $H_G^2(X)\to H_G^2(X)/H_G^2(\mathrm{pt}) = H^2(X,\mathbb{C}) = \mathfrak{t}_\mathbb{C}^k$.  From this, we can take the induced map 
\begin{align*}
    \pi^*\colon \mathrm{Gr}(k,\mathfrak{u}_{\Phi^+}^1)\longrightarrow \mathrm{Gr}(n+1,\widetilde{\mathfrak{u}_{\Phi^+}}^1), 
\end{align*}
defined by taking preimages.  As vector spaces, we have the decomposition $H_G^2(X)=\mathfrak{t}_\mathbb{C}^k\oplus H_G^2(\mathrm{pt})$, so we can choose the generators $u_1,\ldots,u_n,\hbar$ of $H_G^2(X)$ so that $u_1\ldots,u_k$ generate $\mathfrak{t}_\mathbb{C}^k$ and $u_{k+1},\ldots,u_{n},\hbar$ generate $H_G^2(\mathrm{pt})$.  Defining 
\begin{align*}
    Q'\colon T^\mathrm{reg}&\longrightarrow \mathrm{Gr}(k,\mathfrak{u}_{\Phi^+}^1)\\
    q&\longmapsto \mathrm{Span}\{H_i^\mathrm{trig}(q)\mid i=1,\ldots,k\}, 
\end{align*}
We can fit $\gamma_*$, $\pi^*$, $Q$ and $Q'$ into the following diagram:
\begin{center}
    % https://tikzcd.yichuanshen.de/#N4Igdg9gJgpgziAXAbVABwnAlgFyxMJZABgBpiBdUkANwEMAbAVxiRABUA9AHW5xgAeOYACcYAcwC+ISaXSZc+QigDM5KrUYs2vfkOABxEZIAUYAOQBqAIykAogEoZckBmx4CRNQCYN9ZqyIILqCwkamANY2pLwAtnQ4ABYAZiJ0EcBMkgD6wLwAColYnJaSnNZOkhowUOLwRKCpELFIZCA4EEi2mgFsAIog1Ax0AEYwDPkKHsogDDDJOM6NIs2t1B1I3tT+2kF95oOzo+OT7kpsIljiiYtDWGCBIFB0cIk1SyBNLYhb7Z2I3R2j144josXi2QAVLwAMZYEQw3hoYqQw7DMYTKbnIKXa6LKqSIA
\begin{tikzcd}[ampersand replacement = \&]
T^\mathrm{reg} \arrow[rrr, "Q"] \arrow[rrrdd, ,"Q'"', dashed] \&  \&  \& {\mathrm{Gr}(n+1,E)},                                                        \\
                                                           \&  \&  \&                                                                            \\
                                                           \&  \&  \& {\mathrm{Gr}(k,\mathfrak{u}_{\Phi^+}^1),} \arrow[uu, "\gamma_*\circ\pi^*"', dashed]
\end{tikzcd}
\end{center}
which we will verify at the end of Section \ref{S2.4}.  Our above problem of extending $Q$ therefore reduces to extending the map $Q'$:
\begin{center}
    % https://tikzcd.yichuanshen.de/#N4Igdg9gJgpgziAXAbVABwnAlgFyxMJZABgBpiBdUkANwEMAbAVxiRABUA9AHW5xgAeOYACcYAcwC+ISaXSZc+QigBM5KrUYs2vfkOABxEZIAUYAOSleAdyyw8DWMF4BbOjgAWAMxF0A1sBMkgD6ztwACh5YkpKcAIwAlDJyIBjYeAREZHEa9MysiCA2djAOTgAawbwAyljibtKSGjBQ4vBEoD4QLkhkIDgQSHGynSLdvdQDSCojIF09iGr9g4hx1AxYYAUgUHRwHi0yFJJAA
\begin{tikzcd}
T^\mathrm{reg} \arrow[rr,"Q'"] \arrow[d]        &  & \mathrm{Gr}(k,\mathfrak{u}_{\Phi^+}^1) \\
\widetilde{X_\Sigma}. \arrow[rru, dashed] &  &                                                  
\end{tikzcd}
\end{center}
The extended map $\widetilde{Q}\colon \widetilde{X_\Sigma}\rightarrow \mathrm{Gr}(n+1,E)$ will then be given by the counter-clockwise outer composition in the below diagram:
\begin{center}
    % https://tikzcd.yichuanshen.de/#N4Igdg9gJgpgziAXAbVABwnAlgFyxMJZABgBpiBdUkANwEMAbAVxiRABUA9AHW5xgAeOYACcYAcwC+ISaXSZc+QigBM5KrUYs2vfkOABxEZIAUYAOSkAogEoZckBmx4CRMgEYN9Zq0QheAO5YsHgMsMAAGgD6vADKWOIAtnSy9vLOSkRqntTe2n66gsJGphakgcEwoeG8yTgAFgBmInQA1sBMklHAvAAK9VicANSSkpzudpIaMFDi8ESgzRCJSGQgOBBI7rlaviAAiiDUDHQARjAMvQouyiAMMI04aSBLK4hrG0hqmj5IYEwMBjHM4XK4ZVx+e6PZ6vVbUT6IADMO1+iH+gOB50u10ykIeT2OWDAeygdDg9RmMJEyy+8M2SJR+XRQLuIOx4NuUIJdyJJLJFKgVJpDPW9O2PyZAJZJyxYMUELu+JkFEkQA
\begin{tikzcd}
T^\mathrm{reg} \arrow[rr, "Q"] \arrow[d] \arrow[rrd, "Q'" description, dashed] &  & {\mathrm{Gr}(n+1,E)}                                             \\
{\widetilde{X_\Sigma}} \arrow[rr, dashed, "\widetilde{Q}" description]                 &  & {\mathrm{Gr}(k,\mathfrak{u}_{\Phi^+}^1)}. \arrow[u]
\end{tikzcd}
\end{center}     

In Section \ref{S3}, we will attend to the task of defining the space $\widetilde{X_\Sigma}$ to extend $Q'$.  Its construction can be summarized as follows.  Starting with the subset $\Phi^+\subset(\mathfrak{t}^k_\mathbb{Z})^*$, we study the hypersurface arrangement $\mathcal{H}_{T^k} = \{H_\alpha\subset T^k\mid \alpha\in\Phi^+\}$ and from this arrangement, obtain a fan $\Sigma$.  This fan determines a toric variety $X_\Sigma$, defined in Section \ref{S3.1} with embedding given by $T^\mathrm{reg}\hookrightarrow X_\Sigma$ as a dense open torus.  We will assume throughout the paper that the hypertoric variety and the set $\Phi^+$ are such that the resulting toric variety $X_\Sigma$ is smooth.  The hypersurface arrangement $\mathcal{H}_{T^k}$ then embeds into a hypersurface arrangement $\mathcal{H}$ in $X_\Sigma$, and we can form the poset of its connected components.  By blowing up along the poset, we will obtain a space $\widetilde{X_\Sigma}$ first constructed and studied by De Concini and Gaiffi~\cite{DCG18}.  We then prove the following second Main Theorem.  

\begin{theorem}\label{T2}
    For $X$ a hypertoric variety, with a $G = T^d\times\mathbb{C}^\mathsf{x}$--action, define
    \begin{align*}
        Q\colon T^\mathrm{reg}&\longrightarrow \mathrm{Gr}(n+1,E)\\
        q&\longmapsto \left\{u\star_q(-)\mid u\in H_G^2(X)\right\},
    \end{align*}
    where $E = \mathrm{End}_{H_G^\bullet(\mathrm{pt})}H_G^\bullet(X)$ and let $\widetilde{X_\Sigma}$ to be the compactification of $T^\mathrm{reg}$, as constructed in~\cite{DCG18}.  Then $Q$ fits into the following diagram:
    \begin{center}
    % https://tikzcd.yichuanshen.de/#N4Igdg9gJgpgziAXAbVABwnAlgFyxMJZABgBpiBdUkANwEMAbAVxiRAAoAVAPQGsBKbgB0hOGAA8cwAE4wA5gF8QC0uky58hFGQBMVWoxZsRAdyyw8DWMAAaAfREBlLHIC2dJSrXY8BIgGZyfXpmVkQQETFJYABxaQV2XgBqAEZSAFF+ZVUQDB9NANI9ahCjcMiJKTiE5LSRdxwACwAzaTpeYCYFByEABUasbhSshX0YKDl4IlBWiFckMhAcCCQdEsMwkABFbJnpOYXqZaQUrxBZ+cRApZXEHTOLk6Pb64YsME2oOjhG8d3z-aXRbHK7UN4fNhfH5-UYKIA
\begin{tikzcd}
T^\mathrm{reg} \arrow[rrr, "Q"] \arrow[dd] \arrow[rrrdd, dashed] &  &  & {\mathrm{Gr}(n+1,E)}                              \\
                                                                   &  &  &                                                 \\
\widetilde{X_\Sigma} \arrow[rrr, dashed]                           &  &  & {\mathrm{Gr}(k,\mathfrak{u}_{\Phi^+}^1)}. \arrow[uu]
\end{tikzcd}
    \end{center}   
\end{theorem} 
We prove this theorem by decomposing $\widetilde{X_\Sigma}$ into open sets and extending along each open set.  To define these charts, our starting point will be to study the compactification of the complement of an arrangement of subtori, as given in~\cite{M11}.  In Section \ref{S3.2}, we will follow this paper, which itself is based on~\cite{DCP95}, and define the open sets of this toric compactification using the combinatorics of nested sets.  In Section \ref{S3.3}, we will then extend the map $Q$ to this compactification.  In Section \ref{S3.4}, we will use the description of the charts given in the previous sections to define open charts on the compactified toric variety $\widetilde{X_\Sigma}$.  We can modify the results in Section \ref{S3.3} to extend $Q$ to the boundary divisors of $X_\Sigma$.  

In~\cite{I19,IR19}, a parametrized family of Bethe subalgebras in the Yangian $\mathsf{Y}(\mathfrak{g})$ was studied and characterized as the image of a family of algebras in the trigonometric holonomy Lie algebra $\mathfrak{t}^\mathrm{trig}_\Phi$.  This Lie algebra was introduced in~\cite{TL11}, where its presentation was worked out in detail.  This Lie algebra $\mathfrak{t}_\Phi^\mathrm{trig}$ generalizes the rational trigonometric holonomy Lie algebra of Kohno~\cite{K89} and is constructed from the rank--$k$ root system $\Phi$ in the torus $T^k = \mathrm{Hom}(\Lambda,\mathbb{C}^\mathsf{x})$, where $\Lambda$ is the root lattice of $\Phi$.  The family of algebras in $\mathfrak{t}_\Phi^\mathrm{trig}$ coming from~\cite{I19,IR19} was recently studied in Dec 2025 by Ilin and Rybnikov in~\cite{IR25}, following the work of~\cite{TL11} and also of~\cite{AFV16}.  In~\cite{IR25}, the authors studied the family 
\begin{align*}
    \psi\colon T^\mathrm{reg}\longrightarrow \mathrm{Gr}(k,\mathrm{dim}(\ft_\Phi^\mathrm{trig})^1),
\end{align*}
of Bethe subspaces in $\ft^\mathrm{trig}_\Phi$.  They obtained a regular extension to the family 
\begin{align*}
    \overline{\psi}\colon \widetilde{X_\Sigma}\longrightarrow \mathrm{Gr}(k,\mathrm{dim}(\ft_\Phi^\mathrm{trig})^1),
\end{align*}
where $\widetilde{X_\Sigma}$ is the De Concini-Gaiffi compactification of $T^\mathrm{reg}$ as above, and showed that this extended map is injective in the case where $\Phi$ is a root system of Types $A$, $B$, $C$, and $D$~\cite{BDS49}.  In \cite{IR25} the authors also constructed and characterized the ``Bethe bundle,'' $\overline{\psi}^*\cT$ which is defined as the pullback of the tautological bundle $\cT$ over $\mathrm{Gr}(k,\mathrm{dim}(\ft_\Phi^\mathrm{trig})^1)$.  This bundle was found to be isomorphic to the tangent bundle $T\widetilde{X_\Sigma}(-\mathrm{log}D)$, logarithmic at the boundary divisor $D = \widetilde{X_\Sigma}\setminus T^k$.  Similar results will be studied by the author in the future in the hypertoric setting, presently explored in this paper.

\subsection*{Acknowledgements}
The author would like to thank Joel Kamnitzer for suggesting and advising this Ph.D. project, as well as Michael McBreen and Valerio Toledano Laredo for help with the proof of Theorem \ref{T2.2.1}.  The author would also like to thank and Leonid Rybnikov and Alexei Ilin for useful discussions about the open sets that form the open compactification in the paper.  The author was funded by NSERC-CGS-D for the duration of this project and thanks the University of Toronto, McGill University and the Fields Institute for hospitality during the writing of this paper.    

\section{Algebraic Structure of Quantum Multiplication}\label{S2}

\subsection{Hypertoric Varieties and Circuits}\label{S2.1}
In this section, we define a hypertoric variety, and analyze the structure of the quantum multiplication.  

Throughout this paper, fix integers $k,n,d>0$ and a short exact sequence of complex tori
\begin{align*}
    1\tto T^k\tto T^n\tto T^d\tto 1,
\end{align*}
with the corresponding exact sequence of complex Lie algebras
\begin{align*}
    0\tto \ft^k\xrightarrow[]{\;\;[\iota]\;\;} \ft^n\xrightarrow[]{\;\;[a]\;\;} \ft^d\tto 0,
\end{align*} 
where $[a]$ and $[\iota]$ are represented by matrices with integer entries.  Throughout this paper, take $\ft_\bZ^k$, $\ft_\bZ^n$ and $\ft_\bZ^d$ to denote the integer lattices of the Lie algebras $\ft^k$, $\ft^n$ and $\ft^d$, respectively.  Set $\ft_\bR^k = \ft_\bZ^k\otimes_\bZ \bR$ and similarly define $\ft_\bR^n$ and $\ft_\bR^d$.  Then, using this notation, the above exact sequence can also be expressed as an exact sequence of lattices
\begin{align*}
    0\tto \ft_\bZ^k\xrightarrow[]{\;\;[\iota]\;\;} \ft_\bZ^n\xrightarrow[]{\;\;[a]\;\;} \ft_\bZ^d\tto 0.
\end{align*}

\begin{definition}\label{D2.1.1}  
    Start with the natural action of $T^n$ on $T^*\mathbb{C}^n$ as follows:  for every $(t_i)_{i=1}^n\in T^n$ and $(z_i,w_i)_{i=1}^n\in T^*\mathbb{C}^n$ define
    \begin{align*}
        (t_i)_{i=1}^n\cdot (\;(z_i)_{i=1}^n,\;(w_i)_{i=1}^n\;) &= (\;(t_iz_i)_{i=1}^n,\;(t_i^{-1}w_i)_{i=1}^n\;),
    \end{align*}
    where the $z_i$ are coordinates on the base and $w_i$ are coordinates on the fibre of $T^*\mathbb{C}^n$.  This action is Hamiltonian with (complex) moment map 
    \begin{align*}
        \mu_n\colon T^*\mathbb{C}^n&\longrightarrow (\mathfrak{t}^n)^*\\
        (z_i,w_i)_{i=1}^n&\longmapsto z_iw_i. 
    \end{align*}
    The inclusion $T^k\hookrightarrow T^n$ induces an action $T^k$ on $T^*\mathbb{C}^n$ with corresponding moment map 
    \begin{align*}
        \mu_k = \iota^*\circ \mu_n\colon T^*\mathbb{C}^n\longrightarrow (\mathfrak{t}^k)^*.
    \end{align*}
    Given a generic character $\chi\in X^*(T^k)$, we construct the following quotient:
    \begin{align*}
        X &:= T^*\mathbb{C}^n{/\!\!/\!\!/}_{(0,\chi)} T^k\\
        &= \mu_k^{-1}(0){/\!\!/}_\chi T^k\\
        &= \mathrm{Proj}\left(\bigoplus_{m\in \mathbb{N}} \mathbb{C}[\mu_k^{-1}(0)]^{T^k,m\chi}\right),
    \end{align*}
    where the last line is the projective GIT quotient.  The space $X$ admits a natural action by the quotient $T^d = T^n/T^k$ and a conical action of $\mathbb{C}^\mathsf{x}$, which preserves the base and scales the cotangent fibres by weight $1$.  These actions by $T^d$ and by $\mathbb{C}^\mathsf{x}$ both induce a Hamiltonian $T^d\times\mathbb{C}^\mathsf{x}$--action on $X$.  We define a {\it hypertoric variety} to be any symplectic resolution $X\tto X_0$ that is constructed by the above quotient, with the above $G:= T^d\times\mathbb{C}^\mathsf{x}$--action.  The K\"{a}hler torus for $X$ is given by $T^k$.
\end{definition}

Starting with the map 
\begin{align*}
    [a]\colon\ft_\bZ^n\tto \ft_\bZ^d,
\end{align*}
let $a_1,\ldots,a_n\in\ft_\bZ^d$ be the images of the canonical generators of $\ft_\bZ^n$.  The coordinate entries of each of the $a_i$ then form the columns of $[a]$, regarded as a $d\times n$--matrix, so we can write 
\begin{align*}
    [a] = \begin{pmatrix}
        \vline & & \vline \\
        a_1 & \cdots & a_n \\
        \vline & & \vline 
    \end{pmatrix}.
\end{align*}

After choosing a lift $\widehat{\chi}\in (\ft_\mathbb{Z}^n)^*$ of the defining GIT parameter $\chi\in X^*(T^k)$, the vectors $a_1,\ldots,a_n\in\ft_\bZ^d$ determine a hyperplane arrangement $\cA = \{H_i\mid i=1,\ldots,n\}$ via 
\begin{align*}
    H_i = \{x\in(\ft^d_\bR)^*\mid x\cdot a_i + \widehat{\chi}_i = 0\}.
\end{align*}

Thus, the relevant combinatorial data necessary for the construction of a hypertoric variety is the hyperplane arrangement $\cA$. 

\begin{definition}\label{D2.1.2}
    We say a hyperplane arrangement is {\it unimodular} if the corresponding matrix $a = (a_{ij})$ is unimodular; i.e. every set of $d$ linearly independent vectors $(a_{i_1},\ldots,a_{i_d})$ spans $\ft^d_\bR$ over $\bZ$.  We say that a hyperplane arrangement is {\it simple} if every collection of $m$ hyperplanes with nonempty intersection intersects in codimension $m$.  A hyperplane arrangement is {\it smooth} if it is simple and unimodular.    
\end{definition}

\begin{theorem}(Theorems 3.2, 3.3 of~\cite{BD00})
    A hypertoric variety is {\it smooth} iff the corresponding hyperplane arrangement is smooth.
\end{theorem}

Henceforth, we will assume that our parameters $a_1,\ldots,a_n$ and $\chi$ are chosen such that the hyperplane arrangement, and hence the resulting hypertoric variety, is smooth.  We will now introduce the combinatorial data necessary for our study of hypertoric varieties in the rest of this paper.  

\begin{definition}\label{D2.1.3}
    Given an arrangement of hyperplanes $\cA = \{H_i\mid i=1,\dots,n\}$, define a {\it circuit} to be a nonempty subset $S\subset [n] = \{1,\ldots,n\}$ corresponding to a collection of hyperplanes $\{H_i\mid i\in S\}$ for which $\bigcap_{i\in S}H_i = \emptyset$, but is minimal with this property: $\bigcap_{i\in S'}H_i\neq \emptyset$ for all subsets $S'\subsetneq S$.

    Denote 
    \begin{align*}
        \mathscr{S}:= \left\{S\subset [n]\;\biggr\rvert\; S\mathrm{\; circuit\;in\;}\cA\right\}
    \end{align*}
    to be the set of all circuits.  
\end{definition}

We now characterize the circuits of the corresponding hyperplane arrangement in terms of its column vectors.

\begin{lemma}\label{L2.1.1}
    If $\cA$ is a simple hyperplane arrangement, then we have the following bijection 
    \begin{align*}
        \mathscr{S}&\longleftrightarrow \left\{\begin{pmatrix}a_{i_1}\;\ldots\;a_{i_m}\end{pmatrix}\;(d\times m)-\mathrm{submatrix\;of\;}[a] \;\biggr\rvert\; \substack{m\geq 1,\\\mathrm{rk}(a_{i_1}\;\ldots\;\widehat{a_{i_j}}\;\ldots\;a_{i_m}) = m-1\\\;\forall j=1,\ldots,m\\a_{i_1},\ldots,a_{i_m}\mathrm{\;linearly\; dependent}}\right\}\\
        \{i_1,\ldots,i_m\}&\mapstto \begin{pmatrix} a_{i_1}\;\ldots\;a_{i_m}\end{pmatrix}.
    \end{align*}
\end{lemma}

\begin{proof}
Expanding out $[a] = \begin{pmatrix} a_1\;\ldots\;a_n\end{pmatrix}$, we claim that given a list of column vectors $\{a_{i_1},\ldots,a_{i_m}\}$, the corresponding subarrangement $\{H_{i_j}\mid j=1,\ldots,m\}$ in $(\ft_\bR^d)^*$ has a nonempty intersection iff the vectors  $\{a_{i_j}\}$ are linearly independent (i.e. the corresponding $d\times m$-submatrix has rank $m$).  On the one hand, if the arrangement $\{H_{i_j}\}$ has nonempty intersection, then because the arrangement is simple, we must have
    \begin{align*}
        \codim_{(\ft_\bR^d)^*}\bigcap_{j=1}^m H_{i_j} = m.
    \end{align*}
Since $a_{i_1},\ldots,a_{i_m}$ are the normal vectors to the $H_{i_j}$, it follows that the normal space to the intersection $\bigcap_{j=1}^m H_{i_j}$ is given by $\mathrm{Span}_\mathbb{R}\{a_{i_j}\mid j=1,\ldots,m\}$.  Since the hyperplanes $H_{i_j}$ are linear subspaces, we can write the following: 
    \begin{align*}
        m = \codim_{(\ft_\bR^d)^*}\bigcap_{j=1}^m H_{i_j} = \dim\Span_\bR\{a_{i_j}\mid j=1,\ldots,m\},
    \end{align*}
    which implies that the $\{a_{i_j}\mid j=1,\ldots,m\}$ are linearly independent.  Conversely, if the $\{a_{i_j}\}$ are linearly independent, then by row-reducing the corresponding matrix $\begin{pmatrix}a_{i_1}\;\ldots\;a_{i_m}\end{pmatrix}$, one can take the intersection $\bigcap_{j=1}^m H_{i_j}\subset (\ft_\bR^d)^*$, which is a point in $(\ft_\bR^d)^*$.  
    This implies the lemma.  
\end{proof}

We now introduce a finite set of vectors in $\mathfrak{t}_\mathbb{R}^k\setminus \{0\}$, which will be useful in characterizing these circuits.       

\begin{definition}\label{D2.1.4}
    Let $S = \{i_j\mid j=1,\ldots,m\}\subset [n]$ be a nonempty subset and define
    \begin{align*}
        \Phi^{+}_S &:= \left\{\beta = (\beta_i)_{i=1}^n\in \Ker([a])_\bR\setminus\{0\}= \iota(\ft_\bR^k)\setminus\{0\}\;\;\biggr\rvert\;\;\beta_{i_1} = 1;\;\;\beta_{i_j} = \pm 1,\; i_{j-1}<i_j,\; j=2,\ldots,m;\;\; \beta_i = 0,\; i\not\in S\right\},
    \end{align*}
    and $\Phi^-_S = -\Phi^+_S$.
    We note that $\Phi_S^{\pm}\cap \Phi_{S'}^\pm = \emptyset$ whenever $S\neq S'$.  

    Regarding $([n]\setminus\emptyset,\subset)$ as a poset with respect to the inclusion of subsets, we have 
    \begin{align*}
        \mathscr{S} = \{S\subset [n]\setminus\emptyset\mid \Phi_S^+\neq\emptyset,\mathrm{\;S\;is\;minimal\;with\;this\;property}\}.
    \end{align*}
    Define
    \begin{align*}
        \Phi^{\pm} = \bigsqcup_{S\in\mathscr{S}}\Phi_S^{\pm}\subset \Ker([a])_\bR\setminus \{0\}\simeq \ft_\bR^k\setminus \{0\},
    \end{align*}
    and set $\Phi:= \Phi^+\sqcup\Phi^-$.  We will define a {\it circuit vector} to be any element of $\Phi$.  
\end{definition}

\begin{remark}
    We observe that $\Phi^{\pm}$ spans $\iota((\ft^k)_\bZ)$ and that if $c\beta\in \Phi_S^{\pm}$, then either $c = \pm 1$.  Moreover, $(\alpha,\beta)\in\bZ$ for all $\alpha,\beta\in\iota(\ft_k)_\bR$.  Thus, $\Phi$ satisfies similar properties as usual root systems.  There is no analogue to the reflection property for usual root systems, however.
\end{remark} 

\begin{definition}
    If $\beta\in \Phi$, Define the {\it support} of $\beta\in\ft_\bR^n$ to be 
    \begin{align*}
        \mathrm{Supp}(\beta) = \{i\in [n]\mid \beta_i\neq 0\}
    \end{align*}
    By the minimality condition of $\Phi$ in Definition \ref{D2.1.4}, $\mathrm{Supp}(\beta)$ is a circuit.     
\end{definition}
 
\begin{lemma}\label{L2.1.2}
    Let $\cA$ be a smooth hyperplane arrangement in $(\ft_\bR^d)^*$.  Then there are inverse bijections
    \begin{align*}
        \mathscr{S} &\longleftrightarrow \Phi^+\\
        S&\mapstto \beta_S\\
        \mathrm{Supp}(\beta)&\mapsffrom \beta.
    \end{align*}
\end{lemma}

\begin{proof}
    Suppose $S = \{i_1,\ldots,i_m\}\subset [n]$ is a circuit with corresponding $d\times m$ submatrix $\begin{pmatrix}a_{i_1}\;\ldots\;a_{i_m}\end{pmatrix}$, according to Lemma \ref{L2.1.1}.  Because $\cA$--and hence $[a]$--is unimodular, we can row-reduce it to a matrix of the following form:
    \begin{align*}
        \begin{pmatrix}
             1 & & & 0 & \pm 1\\
            & \ddots & & \\
            & & \ddots & \\
            0 & & &  1 & \pm 1\\
            \hline \\
            & & 0_{(d-m+1)\times m} & &
        \end{pmatrix}.
    \end{align*}
    The kernel of this submatrix is $1$--dimensional with uniquely defined generating element $(\beta_{S,i})_{i=1}^n$ for which $\beta_{S,i_1} = 1$, $\beta_{S,i_2},\ldots,\beta_{S,i_n} = \pm 1$ and all other $\beta_{S,i} = 0$.  Because the rank of this submatrix is $m-1$, $S$ is minimal and hence $S\in\mathscr{S}$, in the notation of Definition \ref{D2.1.4}.  Therefore, $\beta_S\in\Phi^+$ is the corresponding root in $\Phi^+$. 

    Conversely, if $\beta = (\beta_i)_{i=1}^n\in\Phi^+$, then $S = \{i\in[n]\mid \beta_i\neq 0\}$ is the unique set for which $\beta\in\Phi_S^+$.  To prove that $S$ is a circuit, we note that by the minimality of $S$, $\beta$ corresponds uniquely to the relation $\sum_{i\in S}(\pm 1)a_i = 0$ among the column vectors $a_{i_j}$ for which all subcollections $\{a_{i_j}\mid i_j\in S'\subsetneq S\}$ have no relations, because of the minimality of $S$.  It follows that the column vectors $\{a_i\mid i\in S\}$ form a $d\times m$-submatrix of $a$ for which every $d\times(m-1)$--submatrix is of rank $m-1$.  By Lemma \ref{L2.1.1}, $S$ is a circuit.  
\end{proof}

\begin{remark}
    For the matrix $[a]\colon \mathfrak{t}_{\mathbb{Z}}^n\longrightarrow \mathfrak{t}_{\mathbb{Z}}^d$, one can change coordinates on $\mathfrak{t}_{\mathbb{Z}}^n$ and on $\mathfrak{t}_{\mathbb{Z}}^d$ so that $[a]$ is brought to row-reduced echelon form:
    \begin{align*}
        [a] &= \begin{pmatrix}
        \begin{array}{c|c}
          I_{d\times d} & A_{d\times k}
        \end{array}
        \end{pmatrix}\colon \mathfrak{t}_\mathbb{Z}^n\longrightarrow\mathfrak{t}_\mathbb{Z}^d.        
    \end{align*}
    The kernel is then given by 
    \begin{align*}
        [\iota] = \begin{pmatrix}
            \begin{array}{c}
                -A_{d\times k}\\
                \hline
                I_{k\times k}
            \end{array}            \end{pmatrix}\colon\mathfrak{t}_\mathbb{Z}^k\longrightarrow\mathfrak{t}_{\mathbb{Z}}^n.
    \end{align*}
    Since the columns of $[\iota]$ form a basis for $\iota(\mathfrak{t}_\mathbb{Z}^k)$, it follows by the minimal support condition that $\Phi$ contains the image of a basis $\{\varepsilon_1,\ldots,\varepsilon_k\}\subset T^k$.    
\end{remark}

\begin{remark}
    Given a circuit set $S\subset [n]$, along with the corresponding vector $\beta_S\in \Phi^+$, one has a splitting $S = S^+\sqcup S^-$, where 
    \begin{align*}
        S^+ &= \{i\in S\mid \beta_S = 1\}\\
        S^- &= \{i\in S\mid \beta_S = -1\}.
    \end{align*}
\end{remark}

Using this splitting, one can now write down a generators-and-relations description of $H_{T^d\times\mathbb{C}^\mathsf{x}}^\bullet(X)$.  
\begin{theorem}\label{T2.1.1}\cite{HP04}
    The $T^d\times\mathbb{C}^\mathsf{x}$--equivariant cohomology of $X$ has the following presentation:
    \begin{align*}
        H_{T^d\times\mathbb{C}^\mathsf{x}}^\bullet(X)\simeq \frac{\mathbb{C}[u_1,\ldots u_n,\hbar]}{\langle \prod_{i\in S^+}u_i\cdot\prod_{i\in S^-}(\hbar - u_i)\mid S\mathrm{\;circuit}\rangle}.
    \end{align*}
    Here, $S = S^+\sqcup S^-$ is the splitting of the circuit set $S$ based on the previous remark; the $u_i$ are the divisor classes for the hyperplanes $H_i$; and $\hbar$ is the divisor class for the weight of the $\mathbb{C}^\mathsf{x}$--action.   

    The $H_{T^d\times\mathbb{C}^\mathsf{x}}^\bullet(\mathrm{pt})$--module structure is induced by the dual map $[a]^*\colon (\mathfrak{t}_\mathbb{R}^d)^*\longrightarrow (\mathfrak{t}_\mathbb{R}^n)^*$ on divisors.  
\end{theorem}

We now specify the quantum multiplication operation on $H_{T^d\times\mathbb{C}^\mathsf{x}}^\bullet(X)$.

\begin{theorem}\cite{MBS13}
    For divisors $u\in H_{T^d\times\mathbb{C}^\mathsf{x}}^\bullet(X)$, the quantum multiplication is given by 
    \begin{align*}
        u\star_q(-) = u\cup(-) + \hbar\sum_{\beta\in \Phi^+}\frac{q^\beta}{1-q^\beta}(u,\beta)L_\beta(-),
    \end{align*}
    where $L_\beta(-) := L_{\beta_S}(-)$ is a Steinberg operator, defined in~\cite{MBS13}, and $\Phi^+$, regarded as a subset of $\mathfrak{t}_\mathbb{Z}^k$ is the set of K\"{a}hler roots.   
\end{theorem}

Henceforth, we will take $G = T^d\times\bC^\sfx$.  We now prove a lemma about circuits which will be useful for later.

\begin{lemma}\label{L2.1.3}
    Let $\cF\subset \Phi^+$ be a rank--$2$ flat--i.e., $\cF\subset\Phi^+$ is a subset whose elements span a $2$--dimensional subspace of $\ft_\bR^k$ and is maximal with this property.  Then either $\cF = \{\alpha,\beta\}$, or $\cF = \{\alpha,\beta,\alpha+\beta\}$.    
\end{lemma}

\begin{example}
    We note that in the Type $A_n$ root system, where 
    \begin{align*}
        \Phi^+ = \{\varepsilon^{ij} := (0,\ldots,1,\ldots,-1,\ldots,0) \mid 1\leq i<j\leq n+1\}\subset (\bR^{n+1})^*,
    \end{align*}
    the rank--$2$ flats are precisely 
    \begin{align*}
        \cF = \{\varepsilon^{ij},\varepsilon^{kl}\},\; i,j,k,l\mathrm{\;distinct},\\
        \cF = \{\varepsilon^{ij}, \varepsilon^{jk}, \varepsilon^{ik} = \varepsilon^{ij} + \varepsilon^{jk}\},\; i<j<k.
    \end{align*}  
\end{example}

\begin{proof}
    Let $\alpha,\beta\in\cF$.  By the remark following Definition \ref{D2.1.4}, we know that the only scalar multiples of $\alpha,\beta$ that are in $\Phi$ must be $\pm\alpha$ and $\pm\beta$.  Because $\Phi = \Phi^+\sqcup \Phi^-$ and because $\Phi^- = -\Phi^+$, we know that if $\alpha,\beta\in\Phi^+$, then $-\alpha\not\in\Phi^+$ and similarly, $-\beta\not\in\Phi^-$.  Thus, $\alpha,\beta$ cannot be collinear and hence, are linearly independent.  Thus, we find that $\cF\subset \{\lambda\alpha + \mu\beta\mid \lambda,\mu\in\bR\}$.  Since all coefficients are to be integers, we immediately deduce that $\cF\subset \{\lambda\alpha+\mu\beta\mid \lambda,\mu\in\bQ\}$.  Fixing such an element $\lambda\alpha + \mu\beta\in\Phi^+$, we determine further the values of $\lambda,\mu$.  Take $\alpha\in\Phi^+_R$ and $\beta\in\Phi^+_S$, for unique circuits $R$ and $S$, as in the proof of Lemma \ref{L2.1.2}.  Since $R,S\in\mathscr{S}$, $R$ and $S$ are minimal, and we cannot have $R\subset S$ or $S\subset R$.  Taking $i\in R\setminus S$ and $j\in S\setminus R$, we obtain the components
    \begin{align*}
        (\lambda \alpha + \mu \beta)_i &= \lambda\alpha_i\\
        (\lambda \alpha + \mu \beta)_j &= \mu \beta_j.
    \end{align*}
    Since the components of all elements in $\Phi^+$ must have entries either $\pm 1$ or $0$, it follows that $\lambda,\mu\in\{0,\pm 1\}$.  Thus, $\cF\subset \{\alpha,\beta,\alpha+\beta,\alpha-\beta,-\alpha+\beta,-\alpha-\beta\}$.  But because $\alpha,\beta\in\Phi^+$, we cannot have $-\alpha-\beta\in\Phi^+$.  Similarly, we cannot have both $\alpha-\beta,-\alpha+\beta\in\cF$.  Without loss of generality, take $-\alpha+\beta\not\in\cF$, so that $\cF\subset \{\alpha,\beta,\alpha+\beta,\alpha-\beta\}$.  We now consider two cases.  If $R\cap S = \emptyset$, then $\alpha+\beta\in\Phi^+_{R\sqcup S}$, and since $R\sqcup S$ is not minimal, we would find that $\alpha\pm\beta\not\in\Phi$, so in this case, $\cF = \{\alpha,\beta\}$.  Now suppose $R\cap S\neq \emptyset$.  Then in choosing $i\in R\cap S$, and noting that $\alpha_i = \pm 1$ and $\beta_i = \pm 1$, we have the following four possibilities for the $i$--th component of $\alpha\pm \beta$:
    \begin{align*}        
        \begin{pmatrix}
            \alpha_i + \beta_i\\
            \alpha_i - \beta_i 
        \end{pmatrix}
        &= \begin{pmatrix}
            1 &  1\\
            1 & -1\\
        \end{pmatrix}\begin{pmatrix}
            \alpha_i \\
            \beta_i
        \end{pmatrix}
        = \begin{cases}
            \begin{pmatrix}
                2 \\
                0
            \end{pmatrix}, &(\alpha_i, \beta_i) = (1,1)\\\\
            \begin{pmatrix}
                0 \\
                2
            \end{pmatrix}, &(\alpha_i, \beta_i) = (1,-1)\\\\
            \begin{pmatrix}
                -2 \\
                0
            \end{pmatrix}, &(\alpha_i, \beta_i) = (-1,-1)\\\\
            \begin{pmatrix}
                0 \\
                -2
            \end{pmatrix}, &(\alpha_i, \beta_i) = (-1,1)
        \end{cases}.
    \end{align*}
    Thus, the remaining possibilities are either $\cF \subset \{\alpha, \beta\}$, $\cF \subset \{\alpha, \beta, \alpha + \beta\}$, $\cF \subset \{\alpha, \beta, \alpha - \beta\}$.  In the third case, we can write $\alpha = (\alpha - \beta) + \beta$ and upon renaming variables, we complete the proof.     
\end{proof}

\subsection{Defining the Algebraic Structure of the Quantum Multiplication}\label{S2.2}

In this section, we consider the problem of extending the domain of this map to the compactification of $T^{\mathrm{reg}}$, defined by De Concini and Gaiffi~\cite{DCG18}.  To do so, we will factor the codomain of this map, using the Holonomy Lie algebra, originally defined in~\cite{K89}.  The following definition generalizes the Lie algebra of~\cite{K89} and was originally given by Toledano Laredo in~\cite{TL11}, where the presentation was worked out in detail.      

\begin{definition} \label{D2.2.1} (c.f. Equations (1) and (2) of~\cite{AFV16}).
    Define $\ft := \ft_\mathbb{C}^k = H^2(X,\bC)$ and recall the finite subset $\Phi^+\subset \ft^*$ as in Definition \ref{D2.1.4}.  We define the {\it holonomy Lie algebra} $\fs_{\Phi^+}$ to be the Lie algebra generated by $\left\{t_\alpha\mid \alpha\in\Phi^+\right\}$, along with the following commutation relations:
    \begin{align*}
        \left[t_\alpha,\sum_{\beta\in W\cap \Phi^+}t_\beta\right] = 0
    \end{align*}
    where $W\subset \ft^*$ is of dimension $2$ and $\alpha\in W\cap \Phi^+$.  
    
    Define the {\it trigonometric holonomy Lie algebra} $\fu_{\Phi^+}$ to be the Lie algebra generated by $\fs_{\Phi^+}$ and $\ft$, viewed as an Abelian Lie algebra, subject to the following commutation relations:
    \begin{align*}
        \left[t_\alpha,\delta(v)\right] = 0,
    \end{align*}
    where $\alpha\in\Phi^+$, $v\in\Ker(\alpha)$ and 
    \begin{align*}
        \delta(v) &= v - \frac{1}{2}\sum_{\beta\in\Phi^+}\beta(v)t_\beta.
    \end{align*}
    
    Define a grading on $\mathfrak{s}_{\Phi^+}$ so that its generating set $\{t_\alpha\mid \alpha\in\Phi^+\}$ is in degree $1$ and define a grading on $\mathfrak{u}_{\Phi^+}$ so that its generating set $\{t_\alpha\mid \alpha\in\Phi^+\}\cup\mathfrak{t}$ is in degree $1$.  Let $\mathfrak{u}_{\Phi^+}^1$ be the first graded piece.  
\end{definition} 

\begin{remark}
    We observe that in the second inhomogeneous commutation relation for the trigonometric holonomy Lie algebra, it is necessary that the quantity $\alpha(v)$ be well-defined, where $\alpha\in\Phi^+$ and $v\in\ft$.  That is, we require $\Phi^+\subset\ft^*$.  In what follows, we will need to replace the Lie algebra $\ft$, appearing in the definition of $\mathfrak{u}_{\Phi^+}$, with $H_G^2(X)$ so that we can describe the action of the classical multiplication operators $u_i\cup(-)$ for $u_i\in H_G^2(X)$.  But in order to do this, we will need to replace $\alpha\in\Phi^+\subset\ft^*$ with elements $\widetilde{\alpha}\in\widetilde{\Phi^+}\subset H_G^2(X)^*$, so that we can again have a well-defined quantity $\widetilde{\alpha}(u_i)$.  
    
    To this end, we will define an extension of $\fu_{\Phi^+}$, using the so-called ``cohomology exact sequence" in Section 4, Equation (6) of \cite{KMBP21}.  This exact sequence will identify $\ft\simeq H_G^2(X)/H_G^2(\mathrm{pt})$ and thus, we can identify the roots $\alpha\in\Phi^+$ with their lifts $\widetilde{\alpha}\in\widetilde{\Phi^+}\subset H_G^2(X)^*$ which vanish on $H_G^2(\mathrm{pt})$, as desired.   
\end{remark}

From the presentation given in Theorem \ref{T2.1.1}, identify $H_G^2(X)\simeq \ft^n\oplus \bC\hbar$ and $H_G^2(\mathrm{pt})\simeq \ft^d\oplus \bC\hbar$.  

\begin{definition}\label{D2.2.2}
    Define the {\it modified trigonometric holonomy Lie algebra} $\wtilde{\fu_{\Phi^+}}$ to be the Lie algebra generated by $\fs_{\Phi^+}$ and $\mathfrak{t}^n\oplus\mathbb{C}\hbar$, with commutator relations for $\fs_{\Phi^+}$ given as above; with the abelian Lie algebra structure on $\mathfrak{t}^n\oplus\mathbb{C}\hbar$; and with ``cross-term'' commutator relations given by  
    \begin{align*}
        \left[t_\alpha, \delta(v)\right] = 0,
    \end{align*}
    for $\alpha\in\Phi^+$, $v\in \Ker(\wtilde{\alpha})$, and  
    \begin{align*}
        \delta(v) = v - \frac{1}{2}\sum_{\wtilde{\beta}\in\wtilde{\Phi^+}}\wtilde{\beta}(v)t_\beta.
    \end{align*}

    Define a grading on $\widetilde{\mathfrak{u}_{\Phi^+}}$ so that its generating set $\{t_\alpha\mid\alpha\in\Phi^+\}\cup (\mathfrak{t}^n\oplus \mathbb{C}\hbar)$ is in degree $1$.  Let $\widetilde{\mathfrak{u}_{\Phi^+}}^1$ be the first graded piece.   

    Define the projection map 
    \begin{align*}
        \pi\colon \widetilde{\mathfrak{u}_{\Phi^+}}^1\longrightarrow \mathfrak{u}_{\Phi^+}^1,
    \end{align*}
    which restricts to the identity map on $\mathfrak{s}_{\Phi^+}$ and projects $\mathfrak{t}^n\oplus \mathbb{C}\hbar\longrightarrow \mathfrak{t}$.  
\end{definition}

\begin{remark}
    In identifying vectors $\beta\in\ft^n_\bR$ with $(\beta,-)\in(\ft^n_\bR)^*$, regard $\widetilde{\Phi^+}$ as a subset of $(\ft^n_\bR)^*$ and write down its elements as $(\beta,-)\in \widetilde{\Phi^+}$.  Recalling the surjective map $[a]\colon \mathfrak{t}_\mathbb{R}^n\longrightarrow \mathfrak{t}_\mathbb{R}^d$ from Section \ref{S2.1}, observe that for all $\alpha\in\mathfrak{t}_\mathbb{R}^d$, we have $(\beta, ([a]^\mathsf{T}\alpha)) = (([a]\cdot \beta),\alpha) = 0$.  Thus, we may regard $\widetilde{\Phi^+}\subset(\mathfrak{t}_\mathbb{R}^n)^*$--the subset of elements $(\beta,-)$ which vanish on $\ft_\mathbb{R}^d$--with the subset $\Phi^+\subset(\ft^n_\bR)^*/(\ft^d_\bR)^*\simeq (\ft^k_\bR)^*$.  Moreover, we may identify elements $\alpha\in(\ft^k_\bR)^*$ with $\alpha\otimes_\bR 1\in \ft^k = (\ft^k_\bR)\otimes_\bR \bC$, to regard $\Phi^+\subset \ft^*$.
\end{remark}

\begin{definition}\label{D2.2.4}
    Set 
    \begin{align*}
        \gamma\colon\wtilde{\fu_{\Phi^+}}^1&\tto E\\
        t_\alpha&\mapstto \hbar L_\alpha(-),\qquad \qquad\;\alpha\in{\Phi^+}\\
        u_i&\mapstto u_i\cup(-),\qquad i=1,\ldots,n\\
        \hbar&\mapstto \hbar\;\mathbb{I}(-).
    \end{align*}
\end{definition}

\begin{theorem}\label{T2.2.1}
    The map $\gamma$ extends to a well-defined morphism of Lie algebras.  
\end{theorem}

\begin{remark}
    To each divisor $u\in \ft^n\oplus\bC\hbar$, one can assign the operator $\nabla_u = \partial_u+u\star(-)$ so that $\partial_uq^\alpha = (\alpha,u)q^\alpha$ for all $q^\alpha\in T^\mathrm{reg}$ and where $u\star(-)\in E$.  This assignment forms a connection on the trivial bundle over $T^\mathrm{reg}$ with fibre $H_G^\bullet(X)$.  This connection, known as the {\it quantum connection}, is the GKZ connection when $X$ is a hypertoric variety~\cite{MBS13}.  The curvature of this connection is 
    \begin{align*}
        [\nabla_{u_i},\nabla_{u_j}] = [u_i\star(-),u_j\star(-)],
    \end{align*}
    so that the flatness of this connection is equivalent to imposing the commutativity of the quantum multiplication operators $u_i\star(-)$ and $u_j\star(-)$ in $E$.  We will reduce the statement of Theorem \ref{T2.2.1} to proving the commutativity of the quantum multiplication operators $u_i\star(-)$ and $u_j\star(-)$.  
\end{remark}

\begin{remark}
    If $X$ were a resolution of a slice in the affine Grassmannian, then the quantum connection would be a trigonometric KZ connection~\cite{D20,D24}, and if $X$ were a Nakajima quiver variety, then the quantum connection would be the trigonometric Casimir connection~\cite{MO12}.  In either of these cases, the condition of the quantum connection being flat is equivalent to the Steinberg operators, together with the classical multiplication operators, satisfying the commutation relations of Definition \ref{D2.2.1}~\cite{TL02,TL11}.   
\end{remark}

\begin{remark}
    The proof of Theorem \ref{T2.2.1} involves demonstrating that the classical multiplication operators and the Steinberg operators satisfy the commutation relations of Definition \ref{D2.2.2}, which is based on Definition \ref{D2.2.1}.  The trigonometric holonomy Lie algebra of Definition \ref{D2.2.1} is in turn based on the holonomy Lie algebra of Toledano Laredo in~\cite{TL11}.  In Section 2 of~\cite{TL11}, the presentation of the holonomy Lie algebra is worked out in detail.  This was done by writing down the trigonometric Casimir connection and proving that the flatness of this connection is equivalent to the trigonometric holonomy Lie algebra relations.  
\end{remark}

\begin{remark}
    The following proof of Theorem \ref{T2.2.1} is analogous to the proof given in Section 2, paragraphs 2.7-2.19 of~\cite{TL11}.  It is also based on ideas of Yaping Yang, when describing the holonomy of connections defined on hyperplane complements in a set of notes written in Fall, 2013.  

    The main difference between the following proof given here and the proof in paragraphs 2.7-2.19 of~\cite{TL11} is the fact that the trigonometric Casimir connection in~\cite{TL11} uses the classical Lie algebra root system, whereas in this present hypertoric setting, the GKZ quantum connection uses the circuit set $\Phi^+$ defined in Definition \ref{D2.1.4}.  In the following example, we construct a hypertoric variety, where the circuit set $\Phi^+$ cannot come from a classical Lie algebra root system.  Thus, for hypertoric varieties studied here, the K\"{a}hler root system--as well as the rank-$2$ K\"{a}hler root subsystems--must be different than the classical Lie algebra root systems studied in~\cite{TL11}.  
\end{remark}

\begin{example}\label{E1}
    We construct a hypertoric variety whose K\"{a}hler root system $\Phi^+$ cannot be a crystallographic Lie algebra root system.  Start with the arrangement 
    \begin{align*}
        \cA = \{H_1,H_2,H_3,H_4,H_5,H_6\},
    \end{align*}
    of hyperplanes in $(\ft_\bR^3)^*$, where 
    \begin{align*}
        H_1 &= \{x\in(\ft_\bR^3)^*\mid x_1 = 0\}\\
        H_2 &= \{x\in(\ft_\bR^3)^*\mid x_2 = 0\}\\
        H_3 &= \{x\in(\ft_\bR^3)^*\mid x_3 = 0\}
    \end{align*}
    \begin{align*}
        H_4 &= \{x\in(\ft_\bR^3)^*\mid x_2 - x_3 = 1\}\\
        H_5 &= \{x\in(\ft_\bR^3)^*\mid x_3 - x_1 = 1\}\\
        H_6 &= \{x\in(\ft_\bR^3)^*\mid x_1 - x_2 = 1\}.
    \end{align*}
    Then the corresponding matrix is given by 
    \begin{align*}
        [a] = \begin{pmatrix}
            1 & 0 & 0 & 0 & -1 & 1 \\
            0 & 1 & 0 & 1 & 0 & -1 \\
            0 & 0 & 1 & -1 & 1 & 0
        \end{pmatrix}\colon\ft_\bR^6\tto\ft_\bR^3,
    \end{align*}
    with kernel 
    \begin{align*}
        [b] = \begin{pmatrix}
            0 & 1 & -1 \\
            -1 & 0 & 1 \\
            1 & -1 & 0 \\
            1 & 0 & 0 \\
            0 & 1 & 0 \\
            0 & 0 & 1
        \end{pmatrix}\colon\ft_\bR^3\tto\ft_\bR^6.
    \end{align*}
    The resulting K\"{a}hler roots are given by the circuit set 
    \begin{align*}
        \Phi^+ = \{\varepsilon_1,\varepsilon_2,\varepsilon_3,\varepsilon_1+\varepsilon_2,\varepsilon_1+\varepsilon_3,\varepsilon_2+\varepsilon_3,\varepsilon_1+\varepsilon_2+\varepsilon_3\}.
    \end{align*}
    In particular, we find that $|\Phi^+| = 7$ and that $|\Phi| = 2|\Phi^+| = 14$, which implies that $\Phi = \Phi^+\sqcup\Phi^-$ cannot come from a crystallographic root system for a semisimple Lie algebra.  
\end{example}

\begin{comment}
\begin{remark}
    If $X$ were a resolution of a slice in the affine Grassmannian, then the quantum connection would be the trigonometric KZ connection~\cite{D24,D20}, and if $X$ were a Nakajima quiver variety, then the quantum connection would be the trigonometric Casimir connection~\cite{MO12}.  In either case, the condition of the quantum connection being flat is equivalent to the Steinberg operators satisfying the commutation relations in Definition \ref{D2.2.1}. 
   ~\cite{L02}.
\end{remark}

\begin{remark}
    To each divisor $u\in \mathfrak{t}^n\oplus \mathbb{C}\hbar$, one can assign the operator $\nabla_u = \partial_u + u\star(-)$ so that  $\partial_u q^\alpha = (\alpha,u)q^\alpha$ for all $q^\alpha\in T^\mathrm{reg}$ and where $u\star(-)\in E$.  This assignment forms a connection on the trivial bundle over $T^\mathrm{reg}$ with fibre $H_G^\bullet(X)$.  This connection, known as the {\it quantum connection}, is the GKZ connection when $X$ is a hypertoric variety~\cite{MBS13}.  The curvature of this connection is
    \begin{align*}
        [\nabla_{u_i},\nabla_{u_j}] = [u_i\star(-),u_j\star(-)],
    \end{align*}
    so that the flatness of this connection is equivalent to imposing the commutativity of the quantum multiplication operators $u_i\star(-)$ and $u_j\star(-)$ in $E$.  
\end{remark}
\end{comment}

\begin{proof}[Proof of Theorem \ref{T2.2.1}]
    We claim the stronger statement that the trigonometric holonomy relations hold in $E$, 
    \begin{align*}
        [L_\alpha, \sum_{\beta\in\cF}L_\beta] &= 0,\qquad \cF\mathrm{\;rk\; }2\mathrm{\;flat},\; \alpha\in\cF,\\
        [L_\alpha, \delta(u_i)] &= 0,\qquad \delta(u_i) = u_i - \frac{1}{2}\sum_{\beta\in\Phi^+}(\beta,v)\hbar L_\beta,\;\;\;\;\alpha(u_i) = 0,
    \end{align*}
    iff the quantum multiplication operators $u_i\star(-)$, $i=1,\ldots,n$ are commutative,
    \begin{align*}
        [u_i\star(-),u_j\star(-)] = 0, \qquad i,j=1,\ldots,n.
    \end{align*}

    To prove this claim, we compute 
    \begin{align*}
        [u_i\star(-),u_j\star(-)] &= \left[u_i\cup(-) + \hbar\sum_{\alpha\in\Phi^+} \frac{q^\alpha}{1-q^\alpha}(\alpha,u_i) L_\alpha,\;\;\;
        u_j\cup(-) + \hbar\cdot \sum_{\beta\in\Phi^+}\frac{q^\beta}{1-q^\beta}(\beta,u_j)L_\beta\right]\\
        &= \hbar\sum_{\alpha\in\Phi^+}\frac{q^\alpha}{1-q^\alpha}[L_\alpha,\; \alpha_iu_j - \alpha_ju_i] + \hbar^2 \sum_{\alpha,\beta\in\Phi^+} \frac{q^\alpha}{1-q^\alpha}\frac{q^\beta}{1-q^\beta}\frac{1}{2}(\alpha_i\beta_j - \alpha_j\beta_i) [L_\alpha,\; L_\beta],
    \end{align*}
    where $\alpha_i=(\alpha,u_i)$, $\beta_j = (\beta,u_j)$.  Using the identity, $\frac{q^\beta}{1-q^\beta} = -1 + \frac{1}{1-q^\beta}$, we obtain
    \begin{align*}
        [u_i\star(-),u_j\star(-)] &= \hbar\sum_{\alpha\in\Phi^+}\frac{q^\alpha}{1-q^\alpha}\left[L_\alpha,\;\delta(\alpha_iu_j - \alpha_ju_i)\right] + \frac{1}{2}\hbar^2\sum_{\alpha,\beta\in\Phi^+}\frac{q^\alpha}{1-q^\alpha}\frac{1}{1-q^\beta}\det\begin{pmatrix} \alpha_i & \beta_i\\
        \alpha_j & \beta_j\end{pmatrix}[L_\alpha,\;L_\beta],
    \end{align*}
    where 
    \begin{align*}
        \delta(\alpha_i u_j - \alpha_j u_i) = (\alpha_i u_j - \alpha_j u_i) - \frac{1}{2}\hbar\sum_{\beta\in\Phi^+}(\beta,\alpha_iu_j - \alpha_ju_i)L_\beta.
    \end{align*}
    We split the second sum into sums over all rank--$2$ flats $\sum_{\alpha,\beta\in\Phi^+}(\ldots) = \sum_{\cF\subset \Phi^+}\sum_{\alpha,\beta\in\cF}(\ldots)$, and using Lemma \ref{L2.1.3}, we can split the sum over flats into sums of flats of size $2$ and $3$ to get the following:
    \begin{align*}
        [u_i\star(-),u_j\star(-)]
        = \hbar\sum_{\alpha\in\Phi^+}\frac{q^\alpha}{1-q^\alpha}[L_\alpha,\delta(\alpha_iu_j - \alpha_ju_i)] &+         
        \hbar^2\sum_{ \substack{\cF\mathrm{\;rk\; }2\mathrm{\;flat},\\ \alpha\in\cF,\\\cF = \{\alpha,\beta\}}}\frac{1}{2}\det\begin{pmatrix}\alpha_i & \beta_i\\
        \alpha_j & \beta_j\end{pmatrix}\frac{q^\alpha + q^\beta}{(1-q^\alpha)(1-q^\beta)}[L_\alpha, L_\beta]        
        \\+ \hbar^2\sum_{ \substack{\cF\mathrm{\;rk\; }2\mathrm{\;flat},\\ \alpha\in\cF,\\\cF = \{\alpha,\beta,\alpha+\beta\}}}\frac{1}{2}\det\begin{pmatrix}\alpha_i & \beta_i \\
        \alpha_j & \beta_j\end{pmatrix}
        \begin{split}
        &\left( \frac{q^\alpha + q^\beta}{(1-q^\alpha)(1-q^\beta)}[L_\alpha,L_\beta] + \frac{q^\alpha + q^\alpha q^\beta}{(1-q^\alpha)(1-q^\alpha q^\beta)}[L_\alpha,L_{\alpha+\beta}] \right. \\
        &\quad \left. {}+ \frac{q^\alpha q^\beta + q^\beta}{(1-q^\alpha q^\beta)(1-q^\beta)}[L_{\alpha+\beta},L_\beta]  \vphantom{\frac{q^\alpha + q^\beta}{(1-q^\alpha)(1-q^\beta)[L_\alpha,L_\beta]}}\right)
        \end{split}
    \end{align*} 

Using the identity 
\begin{align*}
    \frac{1}{(1-q^\alpha)(1-q^\beta)} = -\frac{1}{1-q^\alpha q^\beta} + \frac{1}{(1-q^\alpha)(1-q^\alpha q^\beta)} + \frac{1}{(1-q^\alpha q^\beta)(1-q^\beta)},
\end{align*}
we rewrite the term in parentheses as follows:
\begin{align*}
    &\qquad \left(-\frac{1}{1-q^\alpha q^\beta} + \frac{1}{(1-q^\alpha)(1-q^\alpha q^\beta)} + \frac{1}{(1-q^\alpha q^\beta)(1-q^\beta)}\right)(q^\alpha + q^\beta)[L_\alpha, L_\beta]\\
    &\qquad+ \frac{q^\alpha + q^\alpha q^\beta}{(1-q^\alpha)(1-q^\alpha q^\beta)}[L_\alpha,L_{\alpha+\beta}] + \frac{q^\alpha q^\beta + q^\beta}{(1-q^\alpha q^\beta)(1-q^\beta)}[L_{\alpha+\beta},L_\beta]\\\\
    %&= -\frac{(q^\alpha + q^\beta)[L_\alpha,L_\beta]}{(1-q^\alpha q^\beta)} + \frac{q^\alpha[L_\alpha, L_\beta + L_{\alpha + \beta}] + q^\beta[L_\alpha,L_\beta] + q^\alpha q^\beta[L_\alpha, L_{\alpha+\beta}]}{(1-q^\alpha)(1-q^\alpha q^\beta)}\\
    %&\qquad +\frac{q^\alpha q^\beta[L_{\alpha+\beta},L_\beta] + q^\alpha[L_\alpha, L_\beta] + q^\beta[L_\alpha + L_{\alpha+\beta},L_\beta]}{(1-q^\alpha q^\beta)(1-q^\beta)}\\\\
    %&= \frac{q^\alpha + q^\beta}{(1-q^\alpha)(1-q^\alpha q^\beta)}[L_\alpha, L_\beta + L_{\alpha + \beta}] + \frac{q^\alpha + q^\beta}{(1-q^\alpha q^\beta)(1-q^\beta)}[L_\alpha + L_{\alpha + \beta}, L_\beta]\\
    %&\qquad- \frac{(q^\alpha + q^\beta)[L_\alpha, L_\beta] + q^\beta[L_\alpha, L_{\alpha+\beta}] + q^\alpha[L_{\alpha+\beta},L_\beta]}{1-q^\alpha q^\beta}\\\\
    &= \frac{q^\alpha + q^\beta}{(1-q^\alpha)(1-q^\alpha q^\beta)}[L_\alpha, L_\beta + L_{\alpha + \beta}] + \frac{q^\alpha + q^\beta}{(1-q^\alpha q^\beta)(1-q^\beta)}[L_\alpha + L_{\alpha + \beta}, L_\beta]\\
    &\qquad- \frac{q^\alpha[L_\alpha + L_{\alpha+\beta}, L_\beta] + q^\beta[L_\alpha, L_\beta + L_{\alpha+\beta}]}{1-q^\alpha q^\beta}\\\\
    &= \frac{q^\alpha + q^\beta - (1-q^\alpha)q^\beta}{(1-q^\alpha)(1-q^\alpha q^\beta)}[L_\alpha, L_{\beta} + L_{\alpha+\beta}]
    + \frac{q^\alpha + q^\beta - (1-q^\beta)q^\alpha}{(1-q^\alpha q^\beta)(1-q^\beta)}[L_\alpha + L_{\alpha + \beta}, L_\beta]\\\\
    &= \frac{q^\alpha}{1-q^\alpha}\frac{1+q^\beta}{1-q^\alpha q^\beta}[L_\alpha, L_\beta + L_{\alpha + \beta}] + \frac{q^\beta}{1-q^\beta}\frac{1+q^\alpha}{1-q^\alpha q^\beta}[L_\alpha + L_{\alpha + \beta}, L_\beta], 
\end{align*}
%where we used $q^\alpha = (q^\alpha - 1) + 1$ and $q^\beta = (q^\beta - 1) + 1$ in the third line above.  
Thus, we find that 
\begin{align*}
    &[u_i\star(-),u_j\star(-)]
        = \hbar\sum_{\alpha\in\Phi^+}\frac{q^\alpha}{1-q^\alpha}[L_\alpha,\delta(\alpha_iu_j - \alpha_ju_i)] +         
        \hbar^2\sum_{ \substack{\cF\mathrm{\;rk\; }2\mathrm{\;flat},\\ \alpha\in\cF,\\\cF = \{\alpha,\beta\}}}\frac{1}{2}\det\begin{pmatrix}\alpha_i & \beta_i\\
        \alpha_j & \beta_j\end{pmatrix}\frac{q^\alpha + q^\beta}{(1-q^\alpha)(1-q^\beta)}[L_\alpha, L_\beta]        
        \\&+ \hbar^2\sum_{ \substack{\cF\mathrm{\;rk\; }2\mathrm{\;flat},\\ \alpha\in\cF,\\\cF = \{\alpha,\beta,\alpha+\beta\}}}\frac{1}{2}\det\begin{pmatrix}\alpha_i & \beta_i \\
        \alpha_j & \beta_j\end{pmatrix}\left(\frac{q^\alpha}{1-q^\alpha}\frac{1+q^\beta}{1-q^\alpha q^\beta}[L_\alpha, L_\beta + L_{\alpha + \beta}] + \frac{q^\beta}{1-q^\beta}\frac{1+q^\alpha}{1-q^\alpha q^\beta}[L_\alpha + L_{\alpha + \beta}, L_\beta]\right).
\end{align*}
 
We claim that the set
\begin{align*}
    \left\{\frac{q^\alpha}{1-q^\alpha} \right\}_{\alpha\in\Phi^+}\cup\left\{\frac{q^\alpha + q^\beta}{(1-q^\alpha)(1-q^\beta)}\right\}_{\substack{\cF\mathrm{\;rk\; }2\mathrm{\;flat},\\ \alpha\in\cF,\\\cF = \{\alpha,\beta\}}}\cup\left\{\frac{q^\alpha}{(1-q^\alpha)}\frac{1+q^\beta}{(1-q^\alpha q^\beta)},\;\frac{q^\beta}{(1-q^\beta)}\frac{1+q^\alpha}{(1-q^\alpha q^\beta)}\right\}_{\substack{\cF\mathrm{\;rk\; }2\mathrm{\;flat},\\ \alpha\in\cF,\\\cF = \{\alpha,\beta,\alpha+\beta\}}}
\end{align*}
forms a set of linearly independent generators in the ring
\begin{align*}
    E\otimes_{\mathbb{C}}\mathbb{C}[T^\mathrm{reg}] = E\otimes_\mathbb{C}\mathbb{C}[q^\alpha\mid\alpha\in\Phi^+]_{\prod_{\alpha\in\Phi^+}(1-q^\alpha)}, 
\end{align*}
viewed as a module over $E$.  That is to say that there are no linear combinations among the elements in the above set with coefficients in $E$.  To prove this claim, suppose we had the following:
\begin{align*}
    &\sum_{\alpha\in\Phi^+}M_\alpha \frac{q^\alpha}{1-q^\alpha} + \sum_{\substack{\cF\mathrm{\;rk\; }2\mathrm{\;flat},\\ \alpha\in\cF,\\\cF = \{\alpha,\beta\}}}M_{\alpha,\beta}\frac{q^\alpha + q^\beta}{(1-q^\alpha)(1-q^\beta)}\\
    &+ \sum_{\substack{\cF\mathrm{\;rk\; }2\mathrm{\;flat},\\ \alpha\in\cF,\\\cF = \{\alpha,\beta,\alpha+\beta\}}} \left(M_{\alpha,\alpha+\beta} \frac{q^\alpha}{(1-q^\alpha)}\frac{1+q^\beta}{(1-q^\alpha q^\beta)} + M_{\beta,\alpha+\beta} \frac{q^\beta}{(1-q^\beta)}\frac{1+q^\alpha}{(1-q^\alpha q^\beta)}\right) = 0,
\end{align*}
for all $M_{\alpha,\beta}, M_{\alpha,\alpha+\beta}, M_{\alpha+\beta,\beta}\in E$.  Define a partial ordering on $\Phi^+$ so that $\alpha\leq \alpha+\beta$.  Then, since there are finitely many circuits in an arrangement $\cA$ (as there are at most finitely many hyperplanes, we may invoke Lemma \ref{L2.1.3} to conclude that $\Phi^+$ is finite, so that minimal elements in $\Phi^+$ exist).  These would correspond to ``generalized simple roots.''  Choosing one such minimal $\alpha_0$, we can multiply both sides by $1-q^{\alpha_0}$, and obtain the following relation:
\begin{align*}
    &\sum_{\alpha\in\Phi^+}M_\alpha \;q^\alpha\frac{1-q^{\alpha_0}}{1-q^\alpha} + \sum_{\substack{\cF\mathrm{\;rk\; }2\mathrm{\;flat},\\ \alpha\in\cF,\\\cF = \{\alpha,\beta\}}}M_{\alpha,\beta}\frac{1-q^{\alpha_0}}{1-q^\alpha}\frac{q^\alpha + q^\beta}{1-q^\beta}\\
    &+ \sum_{\substack{\cF\mathrm{\;rk\; }2\mathrm{\;flat},\\ \alpha\in\cF,\\\cF = \{\alpha,\beta,\alpha+\beta\}}} \left(M_{\alpha,\alpha+\beta}\;q^\alpha\frac{1-q^{\alpha_0}}{(1-q^\alpha)}\frac{1+q^\beta}{(1-q^\alpha q^\beta)} + M_{\beta,\alpha+\beta}\;(1-q^{\alpha_0}) \frac{q^\beta}{(1-q^\beta)}\frac{1+q^\alpha}{(1-q^\alpha q^\beta)}\right) = 0,
\end{align*}
as an equality in 
\begin{align*}
    E\otimes_\mathbb{C}\mathbb{C}[q^\alpha\mid \alpha\in\Phi^+]_{\prod_{\alpha\in\Phi^+\setminus\{\alpha_0\}}(1-q^{\alpha})}.
\end{align*}
Taking the image of this relation under the following quotient:
\begin{align*}
    E\otimes_{\mathbb{C}}\mathbb{C}[q^\alpha\mid \alpha\in\Phi^+]_{\prod_{\alpha\in\Phi^+\setminus\{\alpha_0\}}(1-q^\alpha)}&\longrightarrow E\otimes_{\mathbb{C}}\frac{\mathbb{C}[q^\alpha\mid \alpha\in\Phi^+]_{\prod_{\alpha\in\Phi^+\setminus\{\alpha_0\}}(1-q^\alpha)}}{\langle 1 - q^{\alpha_0}\rangle},
\end{align*}
we have the following:
\begin{align*}
    &M_{\alpha_0} + \sum_{\substack{\cF\mathrm{\;rk\; }2\mathrm{\;flat},\\ \alpha\in\cF,\\\cF = \{\alpha_0,\beta\}}}M_{{\alpha_0},\beta}\frac{1 + q^\beta}{1-q^\beta}
    + \sum_{\substack{\cF\mathrm{\;rk\; }2\mathrm{\;flat},\\ \alpha\in\cF,\\\cF = \{\alpha_0,\beta,\alpha_0+\beta\}}} M_{{\alpha_0},{\alpha_0}+\beta}\;\frac{1+q^\beta}{1-q^\beta} = 0.
\end{align*}

Since no two elements of $\Phi^+$ are collinear, each $\beta\in\Phi^+$ appearing in the terms of the above sums must be distinct.  It follows that the rational functions $\frac{1+q^\beta}{1-q^\beta}$ for each $\beta$ must be linearly independent, so we have $M_{\alpha_0,\beta} = 0$ and $M_{\alpha_0,\alpha_0+\beta} = 0$ for all $\beta$.  Running through all minimal elements $\alpha$ in this manner, we find that all coefficients $M_{\alpha,\alpha+\beta} = M_{\beta,\alpha+\beta} = 0$ and similarly, $M_{\alpha,\beta} = 0$.  It thus follows that the above rational functions are linearly independent.  

 This proves that $[u_i\star(-),u_j\star(-)] = 0$ for all $i,j$ iff the following relations hold: 
 \begin{align*}
     [L_\alpha,\delta(\alpha_iu_j - \alpha_ju_i)] &= 0\qquad\forall\;i,j=1,\ldots,n\\
     [L_\alpha, L_\beta] &= 0
     \qquad\forall\;\cF\mathrm{\;rk\;}2\mathrm{\;flat},\;\cF = \{\alpha,\beta\}\\
     [L_\alpha, L_\beta + L_{\alpha + \beta}] = [L_\alpha + L_{\alpha + \beta}, L_\beta] &= 0\qquad \forall\; \cF\mathrm{\;rk\;}2\mathrm{\;flat},\;\cF = \{\alpha,\beta,\alpha+\beta\}.
 \end{align*}

We now study each equation.  In the first above equation, we observe that $(\alpha,\alpha_i u_j - \alpha_j u_i) = \alpha_i \alpha_j - \alpha_j \alpha_i = 0$ and moreover, if any arbitrary $u\in \mathfrak{t}^n\oplus\mathbb{C}\hbar$ has $(\alpha, u) = 0$, then we may choose $v\in \mathfrak{t}^n\oplus\mathbb{C}\hbar$ for which $(\alpha, v) = 1$ and write $u = (\alpha, v)u - (\alpha, u)v$, so every element in $\mathrm{Ker}(\alpha)$ can be written as $(\alpha, v)u - (\alpha, u)v$.  Expressing $u$ and $v$ in terms of a basis, we find that $\Ker(\alpha)$ is generated by $\alpha_i u_j - \alpha_j u_i$, for $i,j=1,\ldots,n$.  Thus, the inhomogeneous condition that
\begin{align*}
    [L_\alpha, \delta(u)] = 0\qquad\; \forall\; u\in\Ker(\alpha)
\end{align*}
is equivalent to stating that 
\begin{align*}
    [L_\alpha, \delta(\alpha_i u_j - \alpha_j u_i)] = 0\qquad \forall i,j=1,\ldots,n.
\end{align*} 

We now turn to the second equation. If $\alpha\in\Phi^+$ and $\cF$ is a rank--$2$ flat containing $\alpha$ and $\abs{\cF} = 2$, then the second equation is trivially equivalent to the homogeneous commutator condition on the Steinberg operators:
\begin{align*}
    [L_\alpha, L_\beta] = [L_\alpha,\sum_{\beta\in\cF} L_\beta] = 0.
\end{align*} 
As for the third equation, suppose $\cF$ is a rank--$2$ flat of the form $\cF = \{\alpha,\beta,\alpha+\beta\}$.  Then the above conditions give 
\begin{align*}
    [L_\alpha, L_\beta] + [L_\alpha, L_{\alpha+\beta}] &= 0\\
    [L_\alpha,L_\beta] + [L_{\alpha+\beta},L_\beta] &= 0.
\end{align*}
Subtracting these two equations and using the anticommutativity, we find that
\begin{align*}
    [L_{\alpha+\beta},L_\alpha + L_\beta] = 0.
\end{align*}
So that taken together, the equations
\begin{align*}
    [L_\alpha,L_\beta + L_{\alpha+\beta}] &= 0\\
    [L_\beta, L_\alpha + L_{\alpha+\beta}] &= 0\\
    [L_{\alpha+\beta},L_\alpha + L_\beta] &= 0
\end{align*}
are precisely the homogeneous commutativity equations for $\cF$.  This proves the claimed equivalence between the trigonometric holonomy commutation relations and the commutation relations among the quantum multiplication, thereby completing the proof.
\end{proof}
\subsection{Injectivity of the map $\gamma$}\label{S2.3}

In this section, we show that $\gamma$ of Definition \ref{D2.2.4} and Theorem \ref{T2.2.1} is injective.  This will amount to showing that the Steinberg operators $\{L_\alpha\mid \alpha\in{\Phi^+}\}$ for the quantum cohomology of a hypertoric variety are linearly independent.  These operators are indexed by circuits in a certain hyperplane arrangement.  We will first give a more general characterization of the circuits we have defined in Definition \ref{D2.1.4} in terms of matroidal circuits.  Then we will use this characterization to obtain an explicit matrix form of the $L_\alpha$ and then, using this, we will be able to prove that the $L_\alpha$ are linearly independent. 

\begin{definition}\label{D2.3.1}
    A collection $\cC\subset \cP([n])$ of subsets of $[n] = \{1,\ldots,n\}$ is said to be a {\it collection of matroidal circuits} if it satisfies the following axioms:
    \begin{enumerate}
        \item $\emptyset\notin \cC$.
        \item For any pair $S,S'\in \cC$, if $S'\subset S$ then $S' = S$. 
        \item For any pair $S,T\in\cC$, if there exists an element $i\in S\cap T$ then there exists a $U\in\cC$ for which $U\subset (S\cup T)\setminus\{i\}$.
    \end{enumerate}
    Define a {\it matroidal circuit} to be any element of $\cC$.
\end{definition}

\begin{definition}\label{D2.3.2}
    Let $\cA = \{H_i\mid i=1,\ldots,n\}$ be a simple hyperplane arrangement in $(\mathfrak{t}_\mathbb{R}^d)^*$.  We say $S\subset [n]$ is {\it independent} if the corresponding set of normal vectors $\{a_i\mid i\in S\}\subset (\ft_\bR^d)^*$ is a linearly independent set.  Otherwise, we say that $S\subset [n]$ is {\it dependent}.
\end{definition}

\begin{lemma}\label{L2.3.1}
    Let $\cA = \{H_i\mid i=1,\ldots,n\}$ be a simple hyperplane arrangement in $(\ft_\bR^d)^*$ and let
    \begin{align*}
    \Psi = \{\;S\subset[n]\mid S\neq \emptyset,\; S\mathrm{\;is\;a\;minimally\; dependent\;set}\}
    \end{align*}
    be the collection of all circuits, in the sense of Definition \ref{D2.1.3} and Lemma \ref{L2.1.1}.  Then $\Psi$ forms a collection of matroidal circuits.  
\end{lemma}

\begin{comment}
\begin{proof}
    We verify that the set of all circuits in $\cA$ satisfies the axioms of Definition \ref{D2.3.1}.  Axiom (a) is tautological and Axiom (b) follows by the minimality condition on $\Psi$.

    We now verify Axiom (c).  Suppose $S,T\in\Psi$ are distinct circuits.  If $S\cap T = \emptyset$, then Axiom (c) is vacuously true.  Otherwise, let $i\in S\cap T$.  If we can show that $(S\cup T)\setminus\{i\}$ is dependent, then we can extract a subset $U\subset (S\cup T)\setminus\{i\}$ for which $U\in\Psi$.  
    
    To prove dependence, note that Axiom (b) allows us to select $j\in S\setminus T$.  Since $S\in\Psi$, $S\setminus\{j\}$ is independent.  Extend $S\setminus\{j\}$ to a maximal independent set $R\subset S\cup T$.  The dependence of $S$ and $T$ implies that $R\not\supset S$ and $R\not\supset T$ so $\abs{R}\leq \abs{S\cup T}-2$.  On the other hand, we know that $\abs{(S\cup T)\setminus\{i\}} = \abs{S\cup T}-1$.  Since maximally independent sets index basis vectors of $(\ft_\bR^d)^*$ and since all bases have the same cardinality, we conclude that $(S\cup T)\setminus\{i\}$ must be dependent, as claimed.  
\end{proof}
\end{comment}

\begin{remark}
    If $\cA$ is smooth, then Lemma \ref{L2.1.2} implies that $\Psi\simeq \Phi^+$, where $\Phi^+$ is defined in Definition \ref{D2.1.3}.  Thus, Axiom (c) satisfied by $\Psi$ implies that if $\alpha,\beta\in{\Phi^+}$ and $\alpha_i = \pm \beta_i\neq 0$, then there exists a $\gamma\in{\Phi^+}$ for which $\mathrm{Supp}(\gamma)\subset \mathrm{Supp}(\alpha\mp\beta)$.  
\end{remark}

\begin{remark}
    We should note that if $\alpha,\beta\in {\Phi^+}$ are roots in a smooth hyperplane arrangement defined by $[a]\colon \ft_\bR^n\tto \ft_\bR^d$, then if $\mathrm{Supp}(\alpha)\cap \mathrm{Supp}(\beta) = \emptyset$, then $\alpha+\beta$ is not a circuit.  However, the converse need not be true.  Consider, for example the (smooth) hypertoric variety defined by the matrix 
    \begin{align*}
        [a] = \begin{pmatrix}
            1 & 0 & -1 & 0 & -1 \\
            0 & 1 & 0 & -1 & -1
        \end{pmatrix},
    \end{align*}
    The circuits $\alpha = \begin{pmatrix} 0 & 0 & 1 & 1 & -1\end{pmatrix}^\sfT$, $\beta = \begin{pmatrix} 1 & 1 & 0 & 0 & 1\end{pmatrix}^\sfT$ have overlapping support, but $\alpha+\beta = \begin{pmatrix}1 & 1 & 1 & 1 & 0 \end{pmatrix} = \begin{pmatrix}1 & 0 & 1 & 0 & 0 \end{pmatrix}^\sfT + \begin{pmatrix}0 & 1 & 0 & 1 & 0 \end{pmatrix}^\sfT$ is not minimally supported, and hence not a circuit.  
\end{remark}

We now prove a lemma that will be useful for explicitly evaluating the Steinberg operators on the stable basis.     

\begin{lemma}\label{L2.3.2}
    Suppose $X$ is a smooth hypertoric variety with $\cA = \{H_i\mid i=1,\ldots,n\}\subset (\ft_\bR^d)^*$ given by the (smooth) arrangement and let $u_1,\ldots,u_n,\hbar$ be the corresponding generating divisors of $H_G^\bullet(X)$.  For any subset $M\subset [n]$, we will denote $u_M:= \prod_{l\in M}u_l$.
    
    If $P\subset [n]$ is any independent subset and if $\alpha\in {\Phi^+}$ is the sign vector corresponding to the circuit $S = \mathrm{Supp}(\alpha)$, as in Lemma \ref{L2.1.2}, then we have 
    \begin{align*}
        L_\alpha(u_P) = u_{P\setminus S}\cdot L_\alpha(u_{P\cap S}).
    \end{align*}
\end{lemma}

To prove this lemma, we will need a couple of technical results.  The first main result is a vanishing result.     

\begin{lemma}(Lemma 5.1 of~\cite{MBS13}).
    Consider a circuit $S = \mathrm{Supp}(\alpha)$ and a subset $M\subset\mathcal{A}$ such that for all $i\in S\setminus M$, the set $M\sqcup\{i\}$ contains no circuits.  Then for any splitting $M = M^+\sqcup M^-$, we have 
    \begin{align*}
        L_\alpha\left(u_M\right) = 0.
    \end{align*}
\end{lemma}  

\begin{lemma}\label{L2.3.3}
    Suppose $S = \mathrm{Supp}(\alpha)$ is a circuit in a smooth hyperplane arrangement $\cA$ and $M\subset S$ has $\abs{M}\leq \abs{S}-2$.  Then 
    \begin{align*}
        L_\alpha(u_M) = 0.
    \end{align*}
\end{lemma}

\begin{proof}
     Since we know that $M\subset S$ and $|S\setminus M| = |S| - |M|\geq 2$, it follows that for any $i\in S\setminus M$, we have $M\sqcup\{i\}\subsetneq S$ must contain no circuits by the minimality of $S$.  Thus, we can apply Lemma 5.1 of~\cite{MBS13} and conclude.  
\end{proof}

The next lemma can then be used to obtain the following 

\begin{lemma}(Lemma 5.2 of~\cite{MBS13}
    Suppose $S = \mathrm{Supp}(\alpha)$ is a circuit, $S = S^+\sqcup S^-$ a splitting according to the remarks following Lemma \ref{L2.1.2}, and let
    \begin{align*}
        v_l &= \begin{cases}
            u_l,&l\in S^+\\
            \hbar - u_l,&l\in S^-
        \end{cases}\\\\
        w_l = \hbar - v_l &=\begin{cases}
            \hbar - u_l, &l\in S^+\\
            u_l,&l\in S^-
            \end{cases}.
    \end{align*}
    Then for $M = S\setminus i$ for $i\in S$, we have 
    \begin{align*}
        \hbar L_\alpha(v_M) = (-1)^{|S|}w_S.
    \end{align*}
\end{lemma}

We can reformulate the above lemma as follows:

\begin{lemma}\label{L2.3.4}
    Suppose $S = \mathrm{Supp}(\alpha)$ is a circuit and let $M = S\setminus i$ for some $i\in S$.  Let 
    \begin{align*}
        z_l = -\alpha_l w_l = \begin{cases}
            u_l-\hbar, &l\in S^+\\
            u_l, &l\in S^-
        \end{cases}.
    \end{align*}
    Then we have 
    \begin{align*}
        \hbar L_\alpha(u_M) = \alpha_i z_S.
    \end{align*}
\end{lemma}

\begin{proof}
    By definition of $S$ and $\alpha$, we have
    \begin{align*}
        \alpha_i = \begin{cases}
            1, &i\in S^+\\
            -1, &i\in S^-
        \end{cases}.
    \end{align*}
    Then in using the $\bC[\hbar]$--module compatibility of $L_\alpha$, together with Lemma \ref{L2.3.3}, we write the left-hand side of Lemma 5.2 of~\cite{MBS13} as follows
    \begin{align*}
        \hbar L_\alpha(v_M) &= \begin{cases}
            (-1)^{|S^-|}\hbar L_\alpha(u_M),&i\in S^+\\
            -(-1)^{|S^-|}\hbar L_\alpha(u_M),&i\in S^-
        \end{cases}\\
        &= \alpha_i\;(-1)^{|S^-|} \hbar L_\alpha(u_M)
    \end{align*}
    On the other hand, we can write $w_l = -\alpha_lz_l$ and rewrite the right-hand side of Lemma 5.2 of~\cite{MBS13} as
    \begin{align*}
        (-1)^{|S|}w_S &= \prod_{l\in S}(-w_l)\\
        &=\prod_{l\in S}\alpha_lz_l\\
        &= (-1)^{|S^-|}z_S.
    \end{align*}
    Thus, in combining the left-hand and right-hand sides, we obtain the following 
    \begin{align*}
        \alpha_i (-1)^{|S^-|}\hbar L_\alpha(u_M) &= (-1)^{|S^-|}z_S\\
        \hbar L_\alpha(u_M) &= \alpha_i z_S,
    \end{align*}
    completing the proof.  
\end{proof}

\begin{proof}[Proof of Lemma \ref{L2.3.2}]
    We prove this by induction on $\abs{P}$.  Let $\alpha\in{\Phi^+}$ and $S = \mathrm{Supp}(\alpha)$.  For the base case, if $P = \emptyset$, then $P\setminus S = \emptyset$ and $P\cap S = \emptyset$, and the lemma follows trivially.

    Now assume the inductive hypothesis that the lemma holds for all $P'$ for which $\abs{P'}<\abs{P}$.  If $P\subset S$, then we are done.  Otherwise, select $i\in P\setminus S$ and write $P = \{i\}\sqcup P_0$, so that $u_P = u_i u_{P_0}$.  By Theorem \ref{T2.2.1} and the fact that $\alpha(u_i) = 0$, the Steinberg operators $L_\alpha$ satisfy the following relation:
    \begin{align*}
        [L_\alpha,u_i] = \frac{1}{2}\hbar\left[L_\alpha,\sum_{\beta\in{\Phi^+}}\beta_i L_\beta\right],
    \end{align*}
    Evaluating both sides on $u_{P_0}$, we find that 
    \begin{align*}
        L_\alpha(u_P) - u_iL_\alpha(u_{P_0}) = \frac{1}{2}\hbar L_\alpha\left(\sum_{\beta\in{\Phi^+}}\beta_i L_\beta(u_{P_0})\right) - \frac{1}{2}\hbar\sum_{\beta\in{\Phi^+}}\beta_i L_\beta\left(L_\alpha(u_{P_0})\right). 
    \end{align*}

    If we can prove that both terms on the right-hand side vanish, then we will have $L_\alpha(u_P) = u_iL_\alpha(u_{P_0})$ and the result will follow by the inductive hypothesis:
    \begin{align*}
        L_\alpha(u_P) &= u_i\; L_\alpha(u_{P_0})\\
        &= u_i\; u_{P_0\setminus S}\; L_\alpha(u_{P_0\cap S})\\
        &= u_{P\setminus S}\; L_\alpha(u_{P\cap S}).
    \end{align*}

    We first claim that $\beta_i L_\beta(u_{P_0}) = 0$ for all $\beta\in{\Phi^+}$.  Fix a $\beta\in{\Phi^+}$ and set $T = \mathrm{Supp}(\beta)$.  We invoke the inductive hypothesis on $u_{P_0}$ to obtain $L_\beta(u_{P_0}) = u_{P_0\setminus T}\cdot L_\beta(u_{P_0\cap T})$ and using Lemma \ref{L2.3.3}, we find that $L_\beta(u_{P_0}) = L_\beta(u_{P_0\cap T}) = 0$ if $\abs{T\setminus P_0}\geq 2$.  So consider the case where $T\setminus P_0 = \{j\}$.  If $j\neq i$, then $i\not\in T = (P_0\cap T)\sqcup\{j\}$, so $\beta_i = 0$.  But if $j = i$, then we have $T = \{i\}\sqcup (P_0\cap T)\subset P$, which contradicts the assumption that $P$ contains no circuits. It follows that $\beta_i L_\beta(u_{P_0}) = 0$ for all $\beta\in{\Phi^+}$, proving the claim.    

    We next claim that 
    \begin{align*}
        L_\alpha(u_{P_0}) = \begin{cases}
            \alpha_j\; z_{P_0\cup S},& S = (S\cap P_0)\sqcup\{j\},\; j\neq i\\
            0, &\mathrm{otherwise},
        \end{cases}
    \end{align*}
    where
    \begin{align*}
        z_l = \begin{cases}
            u_l - \hbar, &l\in S^+\\
            u_l, &l\in S^-\\
            u_l, &l\in P_0\setminus S
        \end{cases}.
    \end{align*}
    In the case where $S = (S\cap P_0)\sqcup \{j\}$, and $j\neq i$, this is given by Lemma \ref{L2.3.4}, together with the inductive hypothesis on $P_0$.  Otherwise, if $i = j$, then similar to before, we find $S = (S\cap P_0)\sqcup\{i\}\subset P$, contradicting the assumption that $P$ contains no circuits.  The only remaining case to consider is where $\abs{S\setminus P_0}\geq 2$, in which case $L_\alpha(u_{P_0}) = 0$ by Lemma \ref{L2.3.3}.  This proves the second claim.    
    
    We now claim that $\beta_i L_\beta(L_\alpha(u_{P_0})) = 0$ for all $\beta\in {\Phi^+}$.  Fixing such a $\beta$, with $T = \mathrm{Supp}(\beta)$, the second above claim implies  
    \begin{align*}
        \beta_i L_\beta\left(L_\alpha(u_{P_0})\right) 
        &= \begin{cases}
            \alpha_j\beta_i L_\beta(z_{P_0\sqcup\{j\}}), &S = (S\cap P_0)\sqcup\{j\},\; j\neq i\\
            0,&\mathrm{otherwise}
        \end{cases},
    \end{align*}
    because $P_0\cup S = P_0\sqcup\{j\}$.  Using the relation $\prod_{i\in S^+}u_i\prod_{i\in S^-}(\hbar - u_i) = 0$ in $H_G^\bullet(X)$ for the given circuit $S$, we can write $z_{P_0\sqcup\{j\}}$ as a sum of squarefree monomials $u_{P_1}$, for $P_1\subset P_0\sqcup\{j\}$, with $\abs{P_1}\leq \abs{P_0}$ and where the $P_1$ contains no circuits.  Using the inductive hypothesis on each such $P_1$ and Lemma \ref{L2.3.3}, we find that $\beta_i L_\beta(z_{P_0\sqcup\{j\}}) = 0$ if $\abs{T\setminus(P_0\sqcup\{j\})}\geq 2$ or if $i\notin T$. 
    Thus, it suffices to consider the case where $T\setminus(P_0\sqcup\{j\}) = \{i\}$, and where $S\setminus P_0 = \{j\}$, with $j\neq i$.  But this implies $T = \{i\}\sqcup(T\cap (P_0\sqcup\{j\}))\subset \{i\}\sqcup P_0\sqcup\{j\} = P\cup S$.  If we can show that the $S$ is the only circuit contained in $P\cup S$, then the fact that $i\in T$ implies $T\neq S$, which is a contradiction.  This would then imply the claim that $\beta_i L_\beta\left(L_\alpha(u_{P_0})\right) = 0$, as needed.     

    To prove that $P\cup S$ contains only the circuit $S$, suppose $T'\subset P\cup S$ is any other circuit such that $T'\neq S$.  Recalling that $S\setminus P = \{j\}$, if $j\not\in T'$, then $T\subset P$, contradicting the assumption that $P$ contains no circuits.  Thus, we must have $j\in T'\cap S$.  By Lemma \ref{L2.3.1}, Axiom (c), there must be a circuit $U'$ for which $U'\subset (S\cup T')\setminus\{j\}\subset (P\cup S)\setminus\{j\}\subset P$, which also contradicts the assumption that $P$ contains no circuits.  This proves that $P\cup S$ only contains the circuit $S$, as needed.       
\end{proof}

\begin{remark}
    Lemmas \ref{L2.3.2}, \ref{L2.3.3} and \ref{L2.3.4} together can be used to express the matrix forms of the Steinberg operators $L_S$ in the stable basis, whose basis elements are indexed by the $G = T^d\times\bC^\sfx$--fixed points of $X$.  These fixed points are in bijection with the vertices of $\cA$ which are in turn bijective with full-rank $d\times d$ submatrices of $a = \begin{pmatrix}a_1\dots a_n\end{pmatrix}$.  This can then be used to show that the set $\{L_\alpha\mid \alpha\in{\Phi^+}\}$ are linearly independent and hence that $\gamma\colon \wtilde{\fu_{\Phi^+}}^1\to E$ is injective.  
\end{remark}

We now find an explicit description of the Steinberg operators $\{L_\alpha\mid \alpha\in{\Phi^+}\}$ using the basis constructed from the fixed point basis.  We have the following description of the fixed-point set of a hypertoric variety:

\begin{proposition}\label{P2.3.1}  Let $X$ be a hypertoric variety, defined in Definition \ref{D2.1.1}, whose hyperplane arrangement is specified by $\cA$.  Then the $G = T^d\times\mathbb{C}^\mathsf{x}$-fixed points of $X$ are in bijection with the vertices of $\cA$.  Equivalently, the $G$--fixed points are in bijection with the $d\times d$ full rank submatrices of $[a]$.  In particular, the fixed point set $X^G$ is finite.   
\end{proposition}

\begin{proof}
    The proof is given in the proof of Proposition 3.2 of~\cite{HH05}.
\end{proof}

\begin{definition}\label{D2.3.3}\cite{MO12}
    Let $P\in X^G$ be a fixed point which we will describe using the corresponding subset $P = \{i_1,\ldots,i_d\}\subset[n]$ and let $[a_P] = \begin{pmatrix}a_{i_1}&\cdots&a_{i_d}\end{pmatrix}$ be the corresponding submatrix.  Then the $G$--action on $X$ induces a $T^d$--action on the normal space $N_P X\simeq T_P X$ whose weights are the $\{a_{i_j}\mid j=1,\ldots,d\}$.  
    
    The union of all weights at all points $P\in X^G$ is given by $\{a_i\mid i\in [n]\}$, which generate all one-dimensional spaces in $\Sigma_{\cA_0}^\vee$.  Taking ${H'}_{i'}$ to be the hyperplanes ($d-1$--dimensional faces) in $\Sigma_{\cA_0}^\vee$, forming the arrangement $\cA_0^\vee$, we obtain the {\it chamber decomposition} of $T^d$
    \begin{align*}
        \ft_\bR^d\setminus\bigcup_{{H'}_{i'}\in \cA_0^\vee}{H'}_{i'} = \bigsqcup_j\fC_j,
    \end{align*}
    where the {\it chambers} $\fC_i$ are the $d$--dimensional open cones in $\Sigma_{\cA_0}^\vee$. 

    We say that $\tau\in\ft_\bR^d$ is {\it generic} if $\tau\in \fC_j$ for some unique $j$.  Given such $\tau$, one can decompose the normal bundle $N_{X,P} = N_{+,P}\oplus N_{-,P}$ into positive and negative weight spaces with respect to $\tau$.  Using the symplectic form on $X$, we have $N_{+,P}^\vee = N_{-,P}\oplus \bC_{\hbar,P}$, where $\bC_{\hbar,P}$ encodes the weight of the residual $\bC^\sfx$--action.  Hence, the $T^d$--equivariant Euler class of $N_{+,P}$ can be written as 
    \begin{align*}
        \mathrm{eu}_{T^d}(N_{+,P}) 
        &= (-1)^{\frac{1}{2}\mathrm{codim}Z}\mathrm{eu}_{T^d}(N_{+,P}^\vee)\\
        &= (-1)^{\frac{1}{2}\mathrm{codim}Z}\mathrm{eu}_{T^d}(N_{-,P})\in H_{T^d}^\bullet(\{P\}),
    \end{align*} 
    and the $T^d$--equivariant Euler class of $N_{X,P}$ can be written as 
    \begin{align*}
        \mathrm{eu}_{T^d}(N_{X,P}) 
        &= \mathrm{eu}_{T^d}(N_{+,P})\;\mathrm{eu}_{T^d}(N_{-,P})\\
        &= (-1)^{\frac{1}{2}\mathrm{codim}Z}\mathrm{eu}_{T^d}(N_{+,P})^2\in H_{T^d}(\{P\}).
    \end{align*}
    Combining these, we can set $\varepsilon_P = \pm \mathrm{eu}_{T^d}(N_{-,P})\in H_{T^d}^\bullet(\{P\})$ and we find 
    \begin{align*}
        \varepsilon_P^2 = (-1)^{\frac{1}{2}\mathrm{codim} Z}\;\mathrm{eu}_{T^d}(N_{X,P}). 
    \end{align*}
    We define a {\it polarization} $\varepsilon$ to be a choice of $\varepsilon_P$ for each point $P\in X^G$.  Since each $\varepsilon_P$ is uniquely determined up to a choice of sign, the data of a polarization is equivalent to a choice of sign $\pm 1$ for each $P\in X^G$. 
\end{definition}

\begin{remark}
    We can identify any $\tau\in \ft_\bZ^d$ uniquely with a one-parameter subgroup $\mathrm{exp}(\tau)\colon T^1\tto T^d$, and under this identification, $\tau$ is generic iff $X^{\mathrm{exp}(\tau)} = X^{T^d}$.  
\end{remark}       

We now define the stable basis that we will be using.

\begin{theorem}(Theorem 3.3.4 of~\cite{MO12})
    Let $X$ be a symplectic resolution with a $G$--action for which the fixed-point set $X^G$ is finite.  Then we have the decomposition $H_G^\bullet(X^G) = \bigoplus_{P\in X^G}H_G^\bullet(\{P\})$ and for a given chamber $\mathfrak{C}$ and polarization $\varepsilon$, there is a unique $H_G^\bullet(\mathrm{pt})$--linear map 
    \begin{align*}        \mathrm{Stab}_{\mathfrak{C},\varepsilon}\colon H_G^\bullet(X^G)&\tto H_G^\bullet(X)\\
        \gamma_P&\longmapsto \Gamma_P
    \end{align*}
    for all $\gamma_P\in H_G^\bullet(\{P\})$ in one of the direct factors of $H_G^\bullet(X^G)$ and for all $P\in X^G$ which satisfies the following properties:
    \begin{enumerate} 
         \item Support: $\mathrm{Supp}(\Gamma)\subset \mathrm{Attr}(\{P\})$.
         \item Normalization: $\Gamma\rvert_P = \varepsilon_P\cup \gamma_P$.
         \item Degree in $\bC[\ft^d]$: Using $H_G^\bullet(\{P\}) = \bC[\hbar]\otimes_\bC\bC[\ft^d]$, the degree in $\ft^d$ satisfies $\mathrm{deg}_{T^d}\Gamma\rvert_{Z'}\leq \frac{1}{2}\mathrm{codim }Z'$ for $Z'\leq P$,  
    \end{enumerate} 
    where $\Gamma_P = \mathrm{Stab}_{\mathfrak{C},\varepsilon}(\gamma_P)$.
\end{theorem}

\begin{remark}
    Because of the normalization axiom, and by the Atiyah-Bott Localization formula, the stable basis map $\mathrm{Stab}_{\mathfrak{C},\varepsilon}$ becomes an isomorphism after localizing $H_{T^d}^\bullet(\{P\})$ at $\varepsilon_P = \pm \mathrm{eu}_{T^d}(N_{-,P})$ for each $P$.   
\end{remark}

\begin{remark}
    We can then obtain an explicit description of the Steinberg operators $L_S$ by restricting along this map.  The result we need is the following.  
\end{remark}

\begin{theorem} \label{T2.3.1} (Theorem 4.6.1 of~\cite{MO12})
    The following diagram commutes for any $\mathfrak{C}$, $\varepsilon$ and circuit $S$:
    \begin{center}
        % https://tikzcd.yichuanshen.de/#N4Igdg9gJgpgziAXAbVABwnAlgFyxMJZABgBpiBdUkANwEMAbAVxiRAAkB9YAcQF8AegB0hAIyYMGMHAAoRAWzo4AFgGNGwALKCeAShB9S6TLnyEUAJnJVajFmy49hYiVNkKlajdv2Hj2PAIiMgBGG3pmVkQObn5ncUlpOSFFFXUGLR1fIxAMALMiKzDqCPtox3jXJI807z5fGxgoAHN4IlAAMwAnCHkkMhAcCCQQkrsokBEcGAAPHGAAZRw6UT5uGuVuugBrYABhQxF6Lpg0bAYCPgMc7t6kK0HhxABmMci2Kdn5pZW14A2trsDqQjnQTmcsBcwFc-CBbn1EKNHkhXiAGCsYAwAAomQLmEBdLDNZQ4EBvMogAAynAW106PQRAyG93JE2pCwEAEEyWiMdjcQVooTiaS+BQ+EA
\begin{tikzcd}[ampersand replacement = \&]
H_G^\bullet(X^G) \arrow[rr, "{\mathrm{Stab}_{\mathfrak{C},\varepsilon}}"] \arrow[d, "L_S^G"'] \&  \& H_G^\bullet(X) \arrow[d, "L_S"'] \\
H_G^\bullet(X^G) \arrow[rr, "{\mathrm{Stab}_{\mathfrak{C},\varepsilon}}"]                     \&  \& H_G^\bullet(X),                  
\end{tikzcd}
    \end{center}
    where $L_S$ is viewed as a subvariety in $H_{2d}^{BM}(X\times_{X_0} X)$ and $L_S^G$ is a restriction of this correspondence to a subvariety in $H_{2d}^{BM}(X^G\times X^G)$ (c.f. Lemma 3.4.2 of~\cite{MO12}).  
\end{theorem}

We now provide an explicit description of the stable basis map in the case of a hypertoric variety.

\begin{definition}
    Let $X$ be a smooth hypertoric variety, specified by the matrix
    \begin{align*}
        [a]\colon\mathfrak{t}_\mathbb{R}^n&\longrightarrow\mathfrak{t}_\mathbb{R}^d
    \end{align*}
    and let $\tau\in\ft_\bR^d$ be a generic cocharacter.  Let $P\in X^G$, with corresponding $d\times d$--submatrix $[a_{P}]$, we can write out the coordinates of $\tau$ as $\tau = \begin{pmatrix}\tau_1&\cdots &\tau_n\end{pmatrix}^\sfT$.  The genericity of $\tau$ implies that all coordinates of $\tau$, namely $(a_{P}^{-1}\cdot \tau)_j$, are nonzero, so take $\epsilon\in \{\pm 1\}^d$ to be a sign vector recording the signs of these entries.  Define 
    \begin{align*}
        v_i = \begin{cases}
            u_i, &\epsilon_i = 1\\
            \hbar - u_i, &\epsilon_i = -1
        \end{cases},
    \end{align*}
    and take 
    \begin{align*}
        v_{P} = \prod_{i\in P}v_i.
    \end{align*}
    Define $\varepsilon$ to be the polarization which assigns to each $P\in X^G$, the polynomial $v_{P}$, and let $\fC\subset\mathfrak{t}_\mathbb{R}^d$ be the unique chamber containing $\tau$.  Define $\mathrm{Stab}_{([a],\tau)} := \mathrm{Stab}_{\fC,\varepsilon}$.      
\end{definition}

\begin{theorem} (Theorem 3.3.5 of~\cite{S13})\label{T2.3.2}
For $X$ a smooth hypertoric variety and $\tau\in\mathfrak{t}_\mathbb{R}^d$ a generic cocharacter, the stable basis map is given by 
\begin{align*}
    \mathrm{Stab}_{([a],\tau)}\colon H_G^\bullet(\mathrm{pt})^{\oplus m}&\longrightarrow H_G^\bullet(X)\\
    e_{P}&\longmapsto v_{P}.
\end{align*}
\end{theorem}

\begin{lemma}\label{L2.3.5}
    The map $\mathrm{Stab}_\tau$ is injective.  
\end{lemma}

\begin{proof}
    Denoting $H_G^\bullet(\mathrm{pt})_{\mathrm{loc}}$ to be the localization of $H_G^\bullet(\mathrm{pt})$ formed by inverting all Euler classes $e(N_{P,-})$ at all points $P\in X^G$, and defining $H_G^\bullet(X)_{\mathrm{loc}} = H_G^\bullet(X)\otimes_{H_G^\bullet(\mathrm{pt})}H_G^\bullet(\mathrm{pt})_\mathrm{loc}$ we have the following commutative diagram:
    \begin{center}
    % https://tikzcd.yichuanshen.de/#N4Igdg9gJgpgziAXAbVABwnAlgFyxMJZABgBpiBdUkANwEMAbAVxiRAAkB9AcQD0AdfgCMmDBjBwAKQThgAPHMDQ4AvgEpewQRDTM4AAklgA1AEY1KkCtLpMufIRQAmclVqMWbLn0EixEyQANNSsbEAxsPAIiMlM3emZWRA4eAWFRcSkZeUVldU4tflkFYAYIAGMVFU1tXSYDIzMLUNtIhyIXOOoEz2TvNL9MoLVObJKyyqs3GCgAc3giUAAzACcIAFskMhAcCCRTbo8kkDHFAGUcOiEVAsF1uhwAC1W6AGtgAGFrQXoVmDRsGUwJZrMs1ptEAcdntEABmUEgVYbLbUXZIFw7OhYBhsR4QCCvFqI8Ho1Ew2GHRJsQTYdYwACOUxUQA
    \begin{tikzcd}
    H_G^\bullet(\mathrm{pt})^{\oplus (n+1)} \arrow[rr, "{\mathrm{Stab}_{([a],\tau)}}"] \arrow[d, hook] &  & H_G^\bullet(X) \arrow[d]  \\
    H_G^\bullet(\mathrm{pt})_{\mathrm{loc}}^{\oplus (n+1)} \arrow[rr, "\simeq"]                                      &  & H_G^\bullet(X)_\mathrm{loc}.
    \end{tikzcd}    
    \end{center}
    Here, the bottom horizontal map is the induced map on localization.  By the remark following our recollection of Theorem 3.3.4 of~\cite{MO12}, we know that this induced map is an isomorphism.  Because  $H_G^\bullet(\mathrm{pt})\simeq \bC[u_1,\ldots,u_{n+1},\hbar]$ is a domain, the left vertical map is an injection.  Thus, the result follows.   
\end{proof}

We will now give a description of the stable basis map and its restriction $L^G$ to the fixed-point set in a variety of examples, where $X\simeq T^*\mathbb{P}^n$ is a hypertoric variety with one circuit, constructed via the exact sequence
\begin{align*}
    0\longrightarrow\mathfrak{t}_\mathbb{R}^1&\longrightarrow \mathfrak{t}_\mathbb{R}^{n+1}\longrightarrow\mathfrak{t}_\mathbb{R}^n\longrightarrow 0,
\end{align*}
and where $\tau\in\mathfrak{t}_\mathbb{R}^n$ is generic.  We will then use this to prove that the Steinberg operators $\{L_\alpha\mid \alpha\in\Phi^+\}$ are linearly independent.     

\begin{lemma}\label{L2.3.6}
    Suppose $X = T^*\mathbb{P}^n$.  If $\tau = \begin{pmatrix}n & \cdots & 1\end{pmatrix}^\mathsf{T}$, then the stable basis map is given as follows 
    \begin{align*}
        \mathrm{Stab}_{([a],\tau)}\colon H_G^\bullet(\mathrm{pt})^{\oplus (n+1)}&\tto H_G^\bullet(T^*\bP^n)\\
        e_{i}&\mapstto u_1\cdot\ldots\cdot u_{i-1}\cdot (\hbar - u_{i+1})\cdot\ldots\cdot (\hbar - u_{n+1}).
    \end{align*}
    
    The restriction of the Steinberg operator $L$ to $L^G$, in the manner of Theorem 4.6.1 of~\cite{MO12}, is given by the following:
    \begin{align*}
        [L^G] = [(-1)^{1+i+j}].
    \end{align*}
\end{lemma}

\begin{proof}
    This proof is a direct application of Theorem \ref{T2.3.1} and Theorem \ref{T2.3.2}.  
    
    In order to derive the stable basis map for $X = T^*\mathbb{P}^n$, we study the signs of $[a_{P}]^{-1}\tau$.  Using this map, together with Theorem 4.6.1 of~\cite{MO12}, we can deduce the matrix entries of $[L^G]$. 
    
    We observe that for $X = T^*\bP^n$, the hyperplane arrangement $\cA$ is the arrangement of $n+1$ hyperplanes bounding an $n$--simplex, and the fixed points of the $G = T^n\times\bC^\sfx$--action on $X$ are indexed by the $n+1$--vertices of this simplex.  We first specify the sign vector $\varepsilon$  The matrix specifying this arrangement is given by 
    \begin{align*}
        [a] = \begin{pmatrix}
            1 & & & -1\\
            & \ddots & & \vdots\\
            & & 1 & -1        \end{pmatrix}\colon\mathfrak{t}^{n+1}\longrightarrow \mathfrak{t}^n,
    \end{align*}
    so the vertices of the simplex correspond uniquely to the subsets $P_i = [n+1]\setminus\{i\}$ for each $i=1,\ldots,n+1$.  For each $i=1,\ldots,n$, the corresponding submatrix is
    \begin{align*}
        [a_i]:= [a_{P_i}] &= \begin{pmatrix}
            \begin{array}{c|c|c}
                \mathbb{I}_{i-1} & 0 & \substack{\\-1\\\vdots\\-1\\}\\       
                \hline
                0 & 0 & -1\\
                \hline
                0 & \mathbb{I}_{n-i} & \substack{\\-1\\\vdots\\-1\\}
            \end{array}
        \end{pmatrix}
    \end{align*}    
    with the $0$s filling the $i$th row.
    The inverse matrix is given by 
    \begin{align*}
        [a_i]^{-1} &= \begin{pmatrix}
            \begin{array}{c|c|c}
                \mathbb{I}_{i-1}  & \substack{\\-1\\\vdots\\-1\\} & 0\\
                \hline
                0 & \substack{\\-1\\\vdots\\-1\\} & \mathbb{I}_{n-i}\\
                \hline
                 0 & -1 & 0
            \end{array}
        \end{pmatrix}
    \end{align*}
    where it is the $i$th column that has the $-1$--entries.  If $i = n+1$, then 
    \begin{align*}
        [a_i]:= [a_{P_i}] = \mathrm{id}_{n\times n},
    \end{align*}
    and $[a_i]^{-1} = [a_i]$.

    Taking $\tau = \begin{pmatrix}n & \ldots & 1\end{pmatrix}^\mathsf{T}\in \ft_\bR^n$, the sign vector $\epsilon_{P_i}$, for $i = 1,\ldots,n$, is given by $\begin{pmatrix}+ & \ldots & + & - & 
    \ldots & -\end{pmatrix}^\sfT$ with $(i-1)$ ``$+$'' signs and $(n-i+1)$ ``$-$'' signs.  If $i=n+1$ then $\epsilon_{P_i}$ is given by $\begin{pmatrix} + & \ldots & \ldots & + \end{pmatrix}^\sfT$.  In this case, Theorem 3.3.5 of~\cite{S13} gives
    \begin{align*}
        \mathrm{Stab}_{([a],\tau)}\colon H_G^\bullet(\mathrm{pt})^{\oplus (n+1)}&\tto H_G^\bullet(T^*\bP^n)\\
        e_{i}&\mapstto u_1\cdot\ldots\cdot u_{i-1}\cdot (\hbar - u_{i+1})\cdot\ldots\cdot (\hbar - u_{n+1}).
    \end{align*}
    This proves the first statement of the claim.  
    
    Since $T^*\bP^n$ only contains one circuit, there is only one Steinberg operator $L$.  A direct computation using Lemmas 5.1 and 5.2 of~\cite{MBS13} gives
    \begin{align*}
        (-1)^i\hbar L\circ\mathrm{Stab}_{([a],\tau)}(e_i)
        &= (\hbar - u_1)\cdot\ldots\cdot(\hbar - u_{n+1})\\
       &= \hbar\cdot (\hbar - u_2)\cdot\ldots\cdot (\hbar - u_{n+1}) - \hbar\cdot u_1\cdot (\hbar - u_3)\cdot\ldots\cdot (\hbar - u_{n+1}) \\
       &\qquad +\ldots + (-1)^{n+2}\hbar\cdot u_1\cdot\ldots\cdot u_n + u_1\cdot\ldots\cdot u_{n+1}.
    \end{align*}
    Because $u_1\cdot\ldots\cdot u_{n+1} = 0$ in $H_G^\bullet(T^*\bP^n)$, we thus find that 
    \begin{align*}
        (-1)^i\hbar L\circ\mathrm{Stab}_{([a],\tau)}(e_i) 
        &= \hbar\;\mathrm{Stab}_{([a],\tau)}(e_1) - \hbar\;\mathrm{Stab}_{([a],\tau)}(e_2) + \ldots + (-1)^{n+2}\hbar\;\mathrm{Stab}_{([a],\tau)}(e_{n+1})\\
        L\circ \mathrm{Stab}_{([a],\tau)}(e_i) &= \sum_{j=1}^{n+1}(-1)^{1+i+j}\;\mathrm{Stab}_{([a],\tau)}(e_j) \\
        &= \mathrm{Stab}_{([a],\tau)}\left(\sum_{j=1}^{n+1}(-1)^{1+i+j}e_j\right).
    \end{align*}
    On the other hand, by Proposition 4.6.1 of~\cite{MO12}, we find that $L\circ\mathrm{Stab}_{([a],\tau)}(e_i) = \mathrm{Stab}_{([a],\tau)}\circ L^G(e_i)$.  Thus,  
    \begin{align*}
        \mathrm{Stab}_{([a],\tau)}\left(L^G(e_i)\right) = \mathrm{Stab}_{([a],\tau)}\left(\sum_{j=1}^{n+1}(-1)^{1+i+j}e_j\right).
    \end{align*}

    By Lemma \ref{L2.3.5}, we can then write 
    \begin{align*}
         L^G(e_i) = \sum_{j=1}^{n+1}(-1)^{1+i+j}e_j
     \end{align*}
     and conclude that in matrix form, $[L^G] = [(-1)^{1+i+j}]$, proving the second statement.  
\end{proof}

\begin{lemma}\label{L2.3.7}
    Now suppose $X = T^*\mathbb{P}^n$, and $\tau\in\mathfrak{t}_\mathbb{R}^n$ is an arbitrary generic character.  Then there exists a permutation $\sigma\in S_{n+1}$ so that the stable basis map is given by
    \begin{align*}
         \mathrm{Stab}_{([a],\tau)}\colon H_G^\bullet(\mathrm{pt})^{\oplus (n+1)}&\tto H_G^\bullet(X)\\
         e_{\sigma(i)}&\mapstto u_{\sigma(1)}\cdot\ldots\cdot u_{\sigma(i-1)}\cdot (\hbar - u_{\sigma(i+1)})\cdot\ldots\cdot (\hbar - u_{\sigma(n+1)}),
     \end{align*}
     and the restriction of the Steinberg variety is 
     \begin{align*}
         [L^G] = [\sigma][(-1)^{1+i+j}][\sigma]^{-1},
     \end{align*}
     where $[\sigma]e_i = e_{\sigma(i)}$.  
\end{lemma}

\begin{proof}
    This is another direct calculation which uses Theorems \ref{T2.3.1} and \ref{T2.3.2} and adapts the proof of Lemma \ref{L2.3.6}.  
\end{proof}

\begin{example}
    Take $X = T^*\bP^4$ and let $\tau = \begin{pmatrix}\tau_1& \tau_2& \tau_3\end{pmatrix}^\sfT$ be such that $\tau_2>\tau_4=0>\tau_3>\tau_1$.  Then take $\sigma\in S_4$ to be $\sigma(1) = 2$, $\sigma(2) = 4$, $\sigma(3) = 3$ and $\sigma(4) = 1$.  From the above matrix descriptions of $[a_i]^{-1}$, we find that the sign vectors of $[a_i]^{-1}\tau$ are given as follows:
    \begin{align*}
        \mathrm{sgn}([a_2]^{-1}\tau) &= \begin{pmatrix}   - & - & - \end{pmatrix}^\sfT\\
        \mathrm{sgn}([a_4]^{-1}\tau) &= \begin{pmatrix}   - & + & - \end{pmatrix}^\sfT\\
        \mathrm{sgn}([a_3]^{-1}\tau) &= \begin{pmatrix}
          - & + & + \end{pmatrix}^\sfT\\
        \mathrm{sgn}([a_1]^{-1}\tau) &= \begin{pmatrix}
          + & + & + \end{pmatrix}^\sfT,
    \end{align*}
    and it follows that
    \begin{align*}
        \mathrm{Stab}_{([a],\tau)}\colon H_G^\bullet(\mathrm{pt})^{\oplus 4}&\tto H_G^\bullet(T^*\bP^4)\\
        e_2&\mapstto (\hbar - u_4)\cdot (\hbar - u_3)\cdot (\hbar - u_1)\\
        e_4&\mapstto u_2\cdot (\hbar - u_3)\cdot (\hbar - u_1)\\
        e_3&\mapstto u_2\cdot u_4\cdot (\hbar - u_1)\\
        e_1&\mapstto u_2\cdot u_4\cdot u_3.
    \end{align*}    
    Hence, we have
    \begin{align*}
        [\sigma] = \begin{pmatrix}
            0 & 0 & 0 & 1 \\
            1 & 0 & 0 & 0 \\
            0 & 0 & 1 & 0 \\
            0 & 1 & 0 & 0
        \end{pmatrix},
    \end{align*}
    and 
    \begin{align*}
        [L^G] =\sigma[(-1)^{1+i+j}]\sigma^{-1} = \begin{pmatrix}
            -1 & 1 & 1 & -1 \\
            1 & -1 & -1 & 1 \\
            1 & -1 & -1 & 1 \\
            -1 & 1 & 1 & -1
        \end{pmatrix}.
    \end{align*}
\end{example}

\begin{lemma}\label{L2.3.8}
    Let $X\simeq T^*\mathbb{P}^n$ be a hypertoric variety specified by the matrix $[a] = \begin{pmatrix} a_1&\cdots&a_n & a_{n+1}\end{pmatrix}$, where $a_i\in \ft_\bR^n$ are column vectors; where $a_i\in\ft_\bR^n$ with $a_{n+1} = - (a_1+\ldots + a_n)$; and where every $n\times n$ submatrix is invertible.

    Write out
    \begin{align*}
        [a] = \begin{pmatrix}\vline & & \vline\\
            a_1&\ldots& a_n\\
            \vline & & \vline 
        \end{pmatrix}
        \begin{pmatrix}
            1 &  &  & -1 \\
            & \ddots &  & \vdots \\
            &  &  1 & -1 
        \end{pmatrix},
    \end{align*}
    and take
    \begin{align*}
        [a]_0 = \begin{pmatrix}
            1 &  &  & -1 \\
            & \ddots &  & \vdots \\
            &  &  1 & -1 
        \end{pmatrix}   
    \end{align*}
    so that 
    \begin{align*}
        [a] = \left(a_i\right)\cdot[a_0].
    \end{align*}

    Then $\tau\in\mathfrak{t}_\mathbb{R}^{n+1}$ is generic iff all coordinate entries of $\left(a_i\right)^{-1}\cdot\tau$ are nonzero and unequal, and we have
    \begin{align*}
        \mathrm{Stab}_{([a],\;\tau)} = \mathrm{Stab}_{(\left(a_i\right)\cdot [a]_0,\; \tau)} = \mathrm{Stab}_{([a]_0,\;\left(a_i\right)^{-1}\cdot\tau)},
    \end{align*}
    as an equality of maps 
    \begin{align*}
        H_G^\bullet(\mathrm{pt})^{\oplus (n+1)}&\longrightarrow H_G^\bullet(X).
    \end{align*}
\end{lemma}

\begin{proof}
    Given a $T^n$--fixed point $P\subset [n+1]$, we observe that the corresponding polarization sign vectors are given by 
    \begin{align*}
        [a_{P}]^{-1}\cdot \tau
        &= [a_{P}]^{-1}\left(a_i\right)\cdot \left(a_i\right)^{-1}\tau\\
        &= \left(\left(a_i\right)^{-1}[a_{P}]\right)^{-1}\cdot\left(\left(a_i\right)^{-1}\tau\right)\\
        &= [a_0]^{-1}\cdot \left(\left(a_i\right)^{-1}\tau\right).
    \end{align*}
    Thus, the genericity of $\tau$ is equivalent to saying that the coordinate entries of $\left(a_i\right)^{-1}\tau$ are nonzero and unequal.  Moreover, the computations for the polarization and stable basis of $X$, as presented by the matrix $[a]$, are equivalent to the corresponding computations for $X$ as presented as $[a]_0$.  
\end{proof}

\begin{lemma}\label{L2.3.9}
    Let $X$ be a hypertoric variety given by the matrix 
    \begin{align*}
        [a] = \begin{pmatrix}1 &  &  & -\varepsilon_1\\ & \ddots & & \vdots \\ & & 1 & -\varepsilon_n \end{pmatrix}.
    \end{align*}
    Then the stable basis map is given by 
    \begin{align*}
        \mathrm{Stab}_{([a],\tau)}\colon H_G^\bullet(\mathrm{pt})^{\oplus (n+1)}&\tto H_G^\bullet(X)\\
        e_i&\mapstto v_1\cdot\ldots\cdot v_{i-1}\cdot (\hbar - v_{i+1})\cdot\ldots\cdot (\hbar - v_{n+1}),
    \end{align*}
    where 
    \begin{align*}
        v_i = \begin{cases}
            u_i,& \varepsilon_i>0\\
            \hbar - u_i,& \varepsilon_i<0
        \end{cases},        
    \end{align*}
    for $i=1,\ldots,n$ and where $v_{n+1} = u_{n+1}$.  
    
    Furthermore, there exists some $\sigma\in S_{n+1}$ for which  
    \begin{align*}
        [L^G] = \sigma[(-1)^{1+i+j}]\sigma^{-1}.
    \end{align*}
\end{lemma}

\begin{proof}
    Combine Theorems \ref{T2.3.1} and \ref{T2.3.2} together with Lemma \ref{L2.3.8}.  
\end{proof}

\begin{lemma}\label{L2.3.10}
    Let $X$ be a hypertoric variety given by the full-rank $n\times(n+1)$--matrix $[a] = \begin{pmatrix}a_1 & \cdots & a_n & a_{n+1}\end{pmatrix}$ for which $\varepsilon_1 a_1+\ldots+\varepsilon_n a_n + a_{n+1} = 0$ for $\varepsilon_i\in \{\pm 1\}$.

    Then the stable basis map is given by 
    \begin{align*}
        \mathrm{Stab}_{([a],\tau)}\colon H_G^\bullet(\mathrm{pt})^{\oplus (n+1)}&\tto H_G^\bullet(X)\\
        e_i&\mapstto v_1\cdot\ldots\cdot v_{i-1}\cdot (\hbar - v_{i+1})\cdot\ldots\cdot (\hbar - v_{n+1}),
    \end{align*}
    where 
    \begin{align*}
        v_i = \begin{cases}
            u_i,& \varepsilon_i>0\\
            \hbar - u_i,& \varepsilon_i<0
        \end{cases},        
    \end{align*}
    for $i=1,\ldots,n$ and where $v_{n+1} = u_{n+1}$.  
    
    Furthermore, there exists some $\sigma\in S_{n+1}$ for which  
    \begin{align*}
        [L^G] = \sigma[(-1)^{1+i+j}]\sigma^{-1}.
    \end{align*}
\end{lemma}

\begin{proof}
    Write out 
    \begin{align*}
        [a] = \begin{pmatrix}\vline & & \vline\\
            a_1&\ldots& a_n\\
            \vline & & \vline 
        \end{pmatrix}
        \begin{pmatrix}
            1 &  &  & -\varepsilon_1 \\
            & \ddots &  & \vdots \\
            &  &  1 & -\varepsilon_n 
        \end{pmatrix},
    \end{align*}
    and combine Lemmas \ref{L2.3.8} and Lemma \ref{L2.3.9}.     
\end{proof}

From Lemma \ref{L2.3.2}, we can now reduce the problem of calculating the Steinberg operators $[L^G]$ to computing $[L^G]$ in the foregoing examples.  We now turn to the problem of demonstrating the linear independence of the Steinberg operators.

\begin{theorem}\label{T2.3.3}
    Let $X$ be a smooth hypertoric variety specified by a matrix $[a] = \begin{pmatrix}a_1&\cdots&a_n\end{pmatrix}\colon \ft_\bR^n\tto \ft_\bR^d$, and with an action on $X$ by the group $G = T^d\times\bC^\sfx$.  Then the Steinberg operators $\{L_S\mid S\mathrm{\;circuit}\}$ forms a linearly independent set.   
\end{theorem} 

We characterize the action of $L_S$ on $\mathrm{Stab}_{([a],\;\tau)}e_{P}$ for each fixed point $P\subset[n]$ of $X$.        

\begin{lemma}\label{L2.3.11}
    For $X$ a smooth hypertoric variety of Theorem \ref{T2.3.3}, and for $P\in X^G$ a fixed point, we have $L_S(\mathrm{Stab}_{([a],\;\tau)}e_P) \neq 0$ iff $\abs{S\setminus P} = 1$.  
\end{lemma}

\begin{proof}
    This follows immediately from Lemma \ref{L2.3.2} and Lemma 5.1 of~\cite{MBS13}.  
\end{proof}

\begin{lemma}\label{L2.3.12}
    The following map is a bijection:
    \begin{align*}
        \left\{S\subset [n+1]\;\;\biggr\rvert\;\; S\mathrm{ circuit},\;\abs{S\setminus P} = 1\right\}&\longrightarrow [n+1]\setminus P\\
        S&\mapstto S\setminus P.
    \end{align*}
\end{lemma}

\begin{proof}
    The fact that the map is injective follows by Lemma \ref{L2.3.1}.  Surjectivity follows because the set $\{a_i\mid i\in P\}$ forms a basis for $\ft_\bR^d$, so for any $j\in [n+1]\setminus P$, we can write $a_j = \sum_{i\in P'\subset P}\varepsilon_i a_i$, where $\varepsilon_i = \pm 1$ by the unimodularity of $[a]$.  The circuit is then defined as $S_j:= P'\sqcup\{j\}$.  
\end{proof}

\begin{lemma}\label{L2.3.13}
    Suppose we have a $d\times n$--matrix $[a]\colon \ft_\bR^n\tto \ft_\bR^d$ and a partition $P = P''\sqcup P'\subset [n+1]$, so that $[a]$ can be written as  
    \begin{align*}
        [a] = \begin{pmatrix}
            [a]_{P''} \\\\
            \hline\\
            [a]_{P'}
        \end{pmatrix},
    \end{align*}
    where $[a]_{P''}$ and $[a]_{P'}$ are $\abs{P''}\times n$-- and $\abs{P'}\times n$--matrices, respectively.  Suppose $[a_P] = \mathrm{id}_{n\times n}$ and let $\tau = (\tau_{P''},\tau_{P'})\in \ft_\bR^d = \ft_\bR^{\abs{P''}}\oplus \ft_{\bR}^{\abs{P'}}$ be a generic character (i.e. $\tau_i\neq 0$ for all $i$).  Then, 
    \begin{align*}
        \mathrm{Stab}_{([a],\;\tau)}e_{P} = 
        \left(\mathrm{Stab}_{([a]_{P''},\;\tau_{P''})}e_{P''}\right)\left(\mathrm{Stab}_{([a]_{P'},\;\tau_{P'})}e_{P'}\right).  
    \end{align*}
\end{lemma}

\begin{proof}
    We observe that because $[a_{P}] = \mathrm{id}_{d\times d}$, we have $[a_{P}]_{P''} = \mathrm{id}_{\abs{P''}\times\abs{P''}}$ and similarly, $[a_{P}]_{P'} = \mathrm{id}_{\abs{P'}\times\abs{P'}}$.  Thus, the genericity of $\tau$ is equivalent to the genericity of both $\tau_{P''}$ and $\tau_{P'}$.  Note also that the matrices $[a]_{P''}$ and $[a]_{P'}$ are simple and unimodular and hence, by Lemma \ref{L2.1.1}, define smooth hypertoric varieties.  The result then follows because both sides evaluate to 
    \begin{align*}
        \prod_{i\in P}v_i = \prod_{i\in P''}v_i\cdot \prod_{i\in P'}v_i,
    \end{align*}
    where 
    \begin{align*}
        v_i = \begin{cases}
            u_i, &\tau_i>0\\
            \hbar - u_i, &\tau_i<0
        \end{cases}.
    \end{align*}
\end{proof}

\begin{proof}[Proof of Theorem \ref{T2.3.3}]
    Suppose we have 
    \begin{align*}
        \sum_{S\mathrm{ circuit }}\lambda_S L_S = 0,
    \end{align*}
    where $\lambda_S\in\bC$.  We evaluate this sum at $\mathrm{Stab}_{([a],\;\tau)}e_{P}$.  Since $\mathrm{Stab}_{([a],\;\tau)}e_{P} = \mathrm{Stab}_{([a_{P}]^{-1}[a],\;[a_{P}]^{-1}\tau)}e_{P}$, we can assume that $[a_{P}] = \mathrm{id}_{d\times d}$.  It follows from Lemmas \ref{L2.3.11} and \ref{L2.3.12} that
    \begin{align*}
        \sum_{j\in [n+1]\setminus P}\lambda_j L_{S_j}(\mathrm{Stab}_{([a],\;\tau)}e_{P}) = 0.
    \end{align*}
    For each $j\in [n+1]\setminus P$, with corresponding circuit $S = S_j$, we partition $P = (P\cap S)\sqcup (P\setminus S)$, and we take $\tau = (\tau_{P\cap S},\tau_{P\setminus S})\in \ft_\bR^d = \ft_\bR^{\abs{P\cap S}}\oplus \ft_\bR^{\abs{P\setminus S}}$.  By Lemmas \ref{L2.3.13} and \ref{L2.3.2}, we have the following calculation:
    \begin{align*}
        L_{S}(\mathrm{Stab}_{([a],\;\tau)}e_{P})
        &= L_{S}\left(\mathrm{Stab}_{([a]_{P\cap S},\;\tau_{P\cap S})}e_{P\cap S}\right)\cdot \left(\mathrm{Stab}_{([a]_{P\setminus S},\;\tau_{P\setminus S})}e_{P\setminus S}\right)\\
        &= \mathrm{Stab}_{([a]_{P\cap S},\;\tau_{P\cap S})}L_S^G(e_{P\cap S})\cdot \left(\mathrm{Stab}_{([a]_{P\setminus S},\;\tau_{P\setminus S})}e_{P\setminus S}\right).
    \end{align*}

    From Lemma \ref{L2.3.10}, we find that 
    \begin{align*}
        L_S^G e_{P\cap S} = \sum_{i\in S}\varepsilon_i e_{S\setminus i},
    \end{align*}
    where $\varepsilon = \pm 1$.  Thus, we find that 
    \begin{align*}
        L_{S}(\mathrm{Stab}_{([a],\;\tau)}e_{P})
        &= \mathrm{Stab}_{([a]_{P\cap S},\;\tau_{P\cap S})}\left(\sum_{i\in S}\varepsilon_i e_{S\setminus i}\right)\cdot \left(\mathrm{Stab}_{([a]_{P\setminus S},\;\tau_{P\setminus S})}e_{P\setminus S}\right)\\
        &= \sum_{i\in S} \varepsilon_i \left(\mathrm{Stab}_{([a]_{P\cap S},\;\tau_{P\cap S})}e_{S\setminus i}\right)\cdot \left(\mathrm{Stab}_{([a]_{P\setminus S},\;\tau_{P\setminus S})}e_{P\setminus S}\right).
    \end{align*}
    Applying Lemma \ref{L2.3.13} again, we find that 
    \begin{align*}
        L_{S}(\mathrm{Stab}_{([a],\;\tau)}e_{P})
        &= \sum_{i\in S}\varepsilon_i\; \mathrm{Stab}_{([a],\;\tau)}\;e_{(P\sqcup{j})\setminus i}.
    \end{align*}
    Thus, in iterating over all $j$, we find that 
    \begin{align*}
        \sum_{j\in [n+1]\setminus P}\lambda_j\; L_{S_j}(\mathrm{Stab}_{([a],\;\tau)}\;e_{P}) &= 0\\
        \sum_{j\in [n+1]\setminus P}\sum_{i\in S_j} \lambda_j\varepsilon_i\;\mathrm{Stab}_{([a],\;\tau)}\;e_{(P\sqcup{j})\setminus i} &= 0\\
        \mathrm{Stab}_{([a],\;\tau)}\left(\sum_{j\in [n+1]\setminus P}\sum_{i\in S_j}\lambda_j\varepsilon_i\; e_{(P\sqcup{j})\setminus i} \right) &= 0. 
    \end{align*}
    By Lemma \ref{L2.3.5}, the stable basis map is injective.  Therefore, we have 
    \begin{align*}
        \sum_{j\in [n+1]\setminus P}\sum_{i\in S_j}\lambda_j\varepsilon_i\; e_{(P\sqcup{j})\setminus i} = 0.
    \end{align*}
    Because the set
    \begin{align*}
        \left\{e_{P\sqcup j\setminus i}\;\;\biggr\rvert\;\; j\in [n+1]\setminus P,i\in S_j\right\}\setminus \{e_{P}\}
    \end{align*}
    is linearly independent, it follows that $\lambda_j = 0$ for all $j\in [n+1]\setminus P$.  Equivalently, $\lambda_S = 0$ for all circuits $S$ for which $\abs{S\setminus P} = 1$.  Iterating over all $G$--fixed points of $X$, indexed by $P\subset[n+1]$, we find that $\lambda_S = 0$ for all $S$, as was to be shown.  
\end{proof}

\begin{example}
    Suppose $X$ is a hypertoric variety specified by the matrix 
    \begin{align*}
        [a] = \begin{pmatrix}
            1 & 0 &  0 & -1 \\
            0 & 1 & -1 & -1
        \end{pmatrix}.
    \end{align*}
    Then we have the following vectors:
    \begin{align*}
        \Phi^+ = \{\beta_{124},\beta_{23},\beta_{134}\},
    \end{align*}
    where  
    \begin{align*}
        \beta_{124} &= \begin{pmatrix} 1 & 1 & 0 & 1 \end{pmatrix}^\sfT\\
        \beta_{23} &= \begin{pmatrix} 0 & 1 & 1 & 0 \end{pmatrix}^\sfT\\
        \beta_{134} &= \begin{pmatrix} 1 & 0 & -1 & 1 \end{pmatrix}^\sfT, 
    \end{align*}
    and the fixed point vertices are given by the subsets $P_1 = \{1,2\}$, $P_2 = \{1,3\}$, $P_3 = \{1,4\}$, $P_4 = \{2,4\}$, and $P_5 = \{3,4\}$.  The generating divisors for $H_{T^2\times \bC^\sfx}^\bullet(X)$ are given by $u_1,u_2,u_3,u_4,\hbar$ and the relations are $u_1u_2u_4 = 0$, $u_2u_3 = 0$ and $u_1(\hbar - u_3)u_4 = 0$.  Choosing the cocharacter $\tau = \begin{pmatrix}2 & 1 \end{pmatrix}^\sfT$, the stable basis maps are given by 
    \begin{align*}
        \mathrm{Stab}_{([a],\;\tau)}\colon H_G^\bullet(\mathrm{pt})^{\oplus 5}&\tto H_G^\bullet(X)\\
        e_{P_1}&\mapstto u_1u_2\\
        e_{P_2}&\mapstto u_1(\hbar - u_3)\\
        e_{P_3}&\mapstto u_1(\hbar - u_4)\\
        e_{P_4}&\mapstto (\hbar - u_2)(\hbar - u_4)\\
        e_{P_5}&\mapstto u_3(\hbar - u_4).
    \end{align*}
    It follows that with respect to the ordered basis, $\{e_{P_1},e_{P_2},e_{P_3},e_{P_4},e_{P_5}\}$, we can use Lemma \ref{L2.3.2} to obtain the following description of the Steinberg operators:
    \begin{align*}
        [L_{124}^G] &= \begin{pmatrix}
            -1 & 0 &  1 & -1 & 0 \\
             0 & 0 &  0 &  0 & 0 \\
             1 & 0 & -1 &  1 & 0 \\
            -1 & 0 &  1 & -1 & 0 \\
             0 & 0 &  0 &  0 & 0
        \end{pmatrix},\\
        [L_{23}^G] &= \begin{pmatrix}
            -1 &  1 & 0 &  0 &  0 \\
             1 & -1 & 0 &  0 &  0 \\
             0 &  0 & 0 &  0 &  0 \\
             0 &  0 & 0 & -1 &  1 \\
             0 &  0 & 0 &  1 & -1
        \end{pmatrix},\\
        [L_{134}^G] &= \begin{pmatrix}
            0 &  0 &  0 & 0 &  0 \\
            0 & -1 &  1 & 0 & -1 \\
            0 &  1 & -1 & 0 &  1 \\
            0 &  0 &  0 & 0 &  0 \\
            0 & -1 &  1 & 0 & -1
        \end{pmatrix}.
    \end{align*}
    Upon inspection, one can see that these matrices are linearly independent.  
\end{example}

\begin{corollary}(Main Theorem \ref{T1})
    Let $X = T^*\mathbb{C}^n///_{0,\chi}T^k$ be a Hypertoric variety with an action of $G = T^d\times\mathbb{C}^\mathsf{x}$, where $d = n-k$ and define the following map:
    \begin{align*}
        \gamma\colon \widetilde{\mathfrak{u}_{\Phi^+}}&\longrightarrow E\\
        t_\alpha&\longmapsto \hbar L_\alpha\\
        u_i&\longmapsto u_i\cup(-)\\
        \hbar&\longmapsto \hbar\;\mathbb{I}(-),
    \end{align*}
    where $\alpha\in\Phi^+$ and $i=1,\ldots,n$.  Then $\gamma$ is well-defined on the level of Lie algebras and the restricted map
    \begin{align*}
        \gamma\colon\widetilde{\mathfrak{u}_{\Phi^+}}^1\longrightarrow E
    \end{align*}
    is injective.   
\end{corollary}
\subsection{Codomain of the map $Q$}\label{S2.4}
From the previous two sections, we have the map 
\begin{align*}
    \gamma_*\colon \mathrm{Gr}(n+1,\widetilde{\mathfrak{u}_{\Phi^+}}^1)&\longrightarrow \mathrm{Gr}(n+1,E),
\end{align*}
which is defined by taking images of subspaces in $\widetilde{\mathfrak{u}_{\Phi^+}}^1$ under $\gamma$.  
Recalling from Definition \ref{D2.2.2} the projection map $\pi\colon \widetilde{\mathfrak{u}_{\Phi^+}}^1\longrightarrow\mathfrak{u}_{\Phi^+}^1$, we have the induced map 
\begin{align*}
    \pi^*\colon \mathrm{Gr}(k,\mathfrak{u}_{\Phi^+}^1)&\longrightarrow \mathrm{Gr}(n+1,\widetilde{\mathfrak{u}_{\Phi^+}}^1),
\end{align*}
which is given by taking pullbacks.  We can now factor the map  
\begin{align*}
    Q\colon T^\mathrm{reg}&\longrightarrow \mathrm{Gr}(n+1,E)\\
    q&\longmapsto \{u\star_q(-)\mid u\in H_G^2(X)\}=\mathrm{Span}_\mathbb{C}\{u_i\star_q(-)\mid i=1,\ldots,n\}\oplus \mathbb{C}\hbar,
\end{align*} 
through the composition 
\begin{align*}
    \gamma_*\circ \pi^*\colon \mathrm{Gr}(k,\mathfrak{u}_{\Phi^+}^1)&\longrightarrow \mathrm{Gr}(n+1,E).
\end{align*}

\begin{theorem}
Let $H_G^2(X) = \mathfrak{t}_\mathbb{C}^k\oplus H_G^2(\mathrm{pt})$ be as in the introduction, let $u_1,\ldots,u_k$ be generators of $\mathfrak{t}_\mathbb{C}^k$ and let $u_{k+1},\ldots,u_n,\hbar$ be the generators of $H_G^2(\mathrm{pt})$.  For each $i=1,\ldots,n$ and $q\in T^\mathrm{reg}$, define 
\begin{align*}
    H_i^\mathrm{trig}(q) = u_i + \sum_{\alpha\in\Phi^+}\frac{q^\alpha}{1-q^\alpha}\alpha_i t_\alpha
\end{align*}
and let
\begin{align*}
    Q'\colon T^\mathrm{reg}&\longrightarrow \mathrm{Gr}(k,\mathfrak{u}_{\Phi^+}^1)\\
    q&\longrightarrow \mathrm{Span}_\mathbb{C}\{H_i^\mathrm{trig}(q)\mid i=1,\ldots,k\}. 
\end{align*}

Then the map $Q$ fits into the following commutative diagram:
\begin{center}
    % https://tikzcd.yichuanshen.de/#N4Igdg9gJgpgziAXAbVABwnAlgFyxMJZABgBpiBdUkANwEMAbAVxiRABUA9AHW5xgAeOYACcYAcwC+ISaXSZc+QigBM5KrUYs2vfkOABxEZIAUYANQBGUgFEAlDLkgM2PASJrLG+s1aIQuoLCRqYA1qS8ALZ0OAAWAGYidKHATJIA+sC8AAqxWJzmkpyWDpIaMFDi8ESgiRCRSGQgOBBI1pq+bACKjrUi9Y3ULUhqIAxYYH4gUHRwsRUg1Ax0AEYwDNkKbsogIljisTiLHdr+XQDkvSB1DYijw4iWZZJAA
\begin{tikzcd}
T^\mathrm{reg} \arrow[rr, "Q"] \arrow[rrd, "Q'"', dashed] &  & {\mathrm{Gr}(n+1,E)}                               \\
                                                        &  & {\mathrm{Gr}(k,\mathfrak{u}_{\Phi^+}^1)}. \arrow[u, "\gamma_*\circ \pi^*"']
\end{tikzcd}
\end{center}
\end{theorem}

\begin{proof}
    Set up the following decomposition: $H_G^2(X)=\mathfrak{t}_\mathbb{C}^k\oplus H_G^2(\mathrm{pt})$, so we can choose the generators $u_1,\ldots,u_n,\hbar$ of $H_G^2(X)$ so that $u_1\ldots,u_k$ generate $\mathfrak{t}_\mathbb{C}^k$ and $u_{k+1},\ldots,u_{n},\hbar$ generate $H_G^2(\mathrm{pt})$.  Then, we have 
    \begin{align*}
        Q'\colon T^\mathrm{reg}&\longrightarrow \mathrm{Gr}(k,\mathfrak{u}_{\Phi^+}^1)\\
    q&\longrightarrow \mathrm{Span}_\mathbb{C}\left\{u_i + \sum_{\alpha\in\Phi^+}\frac{q^\alpha}{1-q^\alpha}\alpha_i t_\alpha\;\biggr\rvert\; i=1,\ldots,k\right\}\\
    \end{align*}
    We then have 
    \begin{align*}
    \pi^*\colon \mathrm{Gr}(k,\mathfrak{u}_{\Phi^+}^1)&\longrightarrow \mathrm{Gr}(n+1,\widetilde{\mathfrak{u}_{\Phi^+}}^1)\\
    \mathrm{Span}_\mathbb{C}\left\{u_i + \sum_{\alpha\in\Phi^+}\frac{q^\alpha}{1-q^\alpha}\alpha_i t_\alpha\;\biggr\rvert\; i=1,\ldots,k\right\}&\longmapsto \mathrm{Span}_\mathbb{C}\left\{\;u_i + \sum_{\alpha\in\Phi^+}\frac{q^\alpha}{1-q^\alpha}\alpha_i t_\alpha\biggr\rvert\;i=1,\dots,n\right\}\oplus \bC\hbar.
    \end{align*}
    and 
    \begin{align*}
        \gamma_*\colon \mathrm{Gr}(n+1,\widetilde{\mathfrak{u}_{\Phi^+}}^1)&\longrightarrow \mathrm{Gr}(n+1,E)\\
        \mathrm{Span}_\mathbb{C}\left\{\;u_i + \sum_{\alpha\in\Phi^+}\frac{q^\alpha}{1-q^\alpha}\alpha_i t_\alpha\biggr\rvert\;i=1,\dots,n\right\}\oplus \bC\hbar&\longmapsto \mathrm{Span}_\bC\left\{u_i\star_q(-)\mid i=1,\ldots,n\right\}\oplus\mathbb{C}\hbar 
    \end{align*}
    Therefore, we find that $Q = \gamma_*\circ\pi^*\circ Q'$.  
\end{proof}

\section{The De Concini--Gaiffi Compactification}\label{S3}

\subsection{The Toric Variety}\label{S3.1}
In this section, we will introduce a toric variety $X_\Sigma$, containing $T^\mathrm{reg}$ as an open dense subset, which will be the first step in constructing the compactified parameter space for the quantum cohomology.

We recall from Definition \ref{D2.1.4} the subset $\Phi = \Phi^+\sqcup\Phi^-\subset(\mathfrak{t}_\mathbb{Z}^k)^*\simeq X^*(T^k)=:\Lambda$ of circuit vectors, where $\iota\colon\mathfrak{t}_\mathbb{Z}^k\hookrightarrow\mathfrak{t}_\mathbb{Z}^n$ is the inclusion.  By Lemma \ref{L2.1.2}, we identify $\Phi^+$ with the set of all circuits of the defining hyperplane arrangement $\mathcal{A}$, and recall that $\Phi^- = -\Phi^+$.      

\begin{definition}\label{D3.1.1}
    If $\alpha\in\Phi$, define the subvariety
    \begin{align*}
        H_\alpha := \left\{t\in T^k\mid \alpha(t)-1=0\right\}, 
    \end{align*}
    and let $\cH_{T^k} := \{H_\alpha\mid \alpha\in{\Phi}\}$ be the corresponding arrangement.  
\end{definition}

\begin{remark}
    We observe that because the $k$ coordinate entries of each $\alpha\in\Phi$ are in $\{0,\pm 1\}$, the GCD of all entries is $1$, so that the $H_\alpha$ are all connected subtori of $T^k$.  Thus, $\cH_{T^k}$ is an arrangement of subtori.  
\end{remark}

We now turn our attention to the parameter space for the quantum cohomology, given by $T^{\mathrm{reg}} = H^2(X,\mathbb{C}^\mathsf{x})\setminus\bigcup_{\alpha\in\Phi}H_\alpha = T^k\setminus\bigcup_{\alpha\in\Phi}H_\alpha$.  The arrangement $\cH_{T^k} = \{H_\alpha\mid \alpha\in\Phi\}$ linearizes to the hyperplane arrangement $\cH_{T^k}^\mathrm{lin} = \{h_\alpha\mid \alpha\in\Phi\}$ in $\mathfrak{t}_\mathbb{R}^k$, where 
\begin{align*}
    h_\alpha = \{v\in\mathfrak{t}_\mathbb{R}^k\mid \langle \alpha, v\rangle = 0\}.
\end{align*}

\begin{definition}
    For each $\alpha\in(\mathfrak{t}_\mathbb{R}^k)^*$, let $h_\alpha^+ = \{v\in\mathfrak{t}_\mathbb{R}^k\mid \langle \alpha, v\rangle\geq 0\}$ denote the corresponding positive half-space. 
    
    Given the central hyperplane arrangement $\cH_{T^k}^{\mathrm{lin}}$, the complement $\mathfrak{t}_{\mathbb{R}}^k\setminus\bigcup_{\alpha\in\Phi}h_\alpha$ is a disjoint union of open cones.  Each open cone $\sigma$ is of the form 
    \begin{align*}
        \sigma = \mathrm{Int}(h_{\alpha_1}^{+}\cap\ldots\cap h_{\alpha_l}^{+}),
    \end{align*}
    for some choice of $\alpha_1,\ldots,\alpha_l\in\Phi$, where $\mathrm{Int}(S) = S\setminus\partial S$ denotes the interior of a topological space $S$.  Let $\Sigma$ denote the collection of all such cones and let $\Sigma(k)$ denote the collection of open $k$--dimensional cones.  Then this fan defines a normal toric variety
    \begin{align*}
        X_\Sigma := \bigcup_{\sigma\in\Sigma(k)}U_\sigma,\qquad\mathrm{ where }\qquad U_\sigma = \bigcup_{\sigma\in\Sigma(k)}\mathrm{Spec\;}\mathbb{C}[\sigma^\vee\cap X^*(T^k)].
    \end{align*}
    and where $\sigma^\vee = \mathrm{Cone}(\alpha_1,\ldots,\alpha_k) = \mathbb{Z}_{\geq 0}\cdot\alpha_1 + \ldots + \mathbb{Z}_{\geq 0}\cdot \alpha_k$.  
\end{definition}

\begin{remark}
    An explicit description of the chart $U_\sigma$, corresponding to $\sigma$ is given by   
    \begin{align*}
        U_\sigma \simeq \mathrm{Spec\;}\mathbb{C}[q^{\alpha_1},\ldots,q^{\alpha_l}].
    \end{align*}
\end{remark}

We recall the following basic facts about toric varieties.

\begin{definition}
    Let $\Sigma$ be a fan in $\bR^k$.    
    \begin{enumerate}
        \item We say that $\Sigma$ is {\it complete} if its cones cover $\bR^k$.
        \item We say that $\Sigma$ is {\it polytopal} if it is the normal fan of some polytope.  
        \item We say that each $\sigma\in\Sigma$ is {\it simplicial} if its minimal generators of $\sigma$ form an $\mathbb{Q}$--linearly independent set.  We say that $\Sigma$ is {\it simplicial} if all its cones are simplicial.  
        \item We say that each $\sigma\in\Sigma$ is {\it regular} if its minimal generators form part of a $\mathbb{Z}$--basis for $\mathfrak{t}_\mathbb{R}^k$.  We say that $\Sigma$ is {\it regular} if all its cones are regular. 
    \end{enumerate}
\end{definition}

\begin{lemma}\label{L3.1.1}
    Let $\Sigma$ be a fan and $X_\Sigma$ the corresponding toric variety.  
    \begin{enumerate}  
        \item $\Sigma$ is complete iff $X_\Sigma$ is proper.
        \item $\Sigma$ is polytopal iff $X_\Sigma$ is projective.
        \item $\Sigma$ is simplicial iff $X_\Sigma$ is $\mathbb{Q}$--factorial.
        \item $\Sigma$ is regular iff $X_\Sigma$ is smooth.
    \end{enumerate}
\end{lemma}

\begin{remark}
    Polytopal fans are complete and regular fans are simplicial.  Correspondingly, projective varieties are proper and smooth toric varieties are $\mathbb{Q}$--factorial.
\end{remark}

The following lemma is standard.

\begin{definition}
    Given a finite set $\Phi\subset(\mathbb{R}^k)^*$ and a corresponding hyperplane arrangement 
    \begin{align*}
        \mathcal{H} = \{\langle \alpha\rangle^\perp\subset \mathbb{R}^k\mid \alpha\in\Phi\},
    \end{align*}
    we define the \textit{zonotope of $\mathcal{H}$} to be the following set
    \begin{align*}
        \sum_{\alpha\in\Phi}\;[-1,1]\cdot\alpha = \left\{\sum_{\alpha\in\Phi}c_\alpha\cdot \alpha\in (\mathbb{R}^k)^*\;\biggr\rvert\;c_\alpha\in[-1,1] \right\}.
    \end{align*}
\end{definition}

The following lemma is standard.

\begin{lemma}[Section 7.3 of~\cite{Z98}]
    The fan $\Sigma$ defined by the hyperplane arrangement $\cH_{T^k}^{\mathrm{lin}} = \{h_\alpha\mid\alpha\in\Phi\}$ is polytopal whose corresponding polytope is the zonotope of $\cH_{T^k}^\mathrm{lin}$.  Thus, $X_\Sigma$ is projective.
\end{lemma}

\begin{conjecture}
    The fan $\Sigma$ defined by the hyperplane arrangement $\cH_{T^k}^{\mathrm{lin}} = \{h_\alpha\mid\alpha\in\Phi\}$ is regular, so $X_\Sigma$ is smooth.
\end{conjecture}

\begin{remark}
    We will assume this conjecture to be true in what follows.  That is, we will restrict our attention to hypertoric varieties $X$ whose corresponding (linearized) discriminental arrangement $\cH_{T^k}^{\mathrm{lin}}$ cuts out a regular fan in $\mathfrak{t}_\mathbb{R}^k$.  Note that this conjecture is satisfied for the symplectic dual hypertoric varieties $X = T^*\mathbb{P}^n$ and $X = \widetilde{\mathbb{C}^2/\mathbb{Z}_{n+1}}$, whose fan $\Sigma$ is induced from the type $A_n$ root arrangement.   
\end{remark}

\begin{remark}
    It is not known by the author how to show that this fan is regular.  In order to do so, one could invoke the classification of crystallographic arrangements given by Cuntz~\cite{C11, C21}.  
\end{remark}

\begin{remark}
    In this case where $\Sigma$ is regular and polytopal, it follows that for each $\sigma\in\Sigma$, the corresponding open set $U_\sigma\subset X_\Sigma$ is given by $U_\sigma=\mathbb{A}_\mathbb{C}^k$.
\end{remark}

\begin{remark}
    By Section 8 of~\cite{DCG18}, we know that if $\Sigma$ is a regular fan, then $\Sigma$ will not need to be subdivided further in order to serve as a valid input for the toric variety $X_\Sigma$ given by the construction by De Concini and Gaiffi in~\cite{DCG18}. 
\end{remark}

\begin{remark}
    Given a cone $\sigma\in \Sigma(k)$ and given an element $\alpha\in\Phi$, we find that either $\mathrm{Supp}(\sigma)\subset h_{\alpha}^+$ or $\mathrm{Supp}(\sigma)\subset h_{-\alpha}^-$.  In writing out $\sigma = h_{\alpha_1}^+\cap\ldots\cap h_{\alpha_k}^+$, it follows that either $\alpha = \sum_{i=1}^k a_i\alpha_i$, where $a_i\geq 0$, or $-\alpha = \sum_{i=1}^k a_i\alpha_i$, where $a_i\leq 0$.  Thus, in corresponding each $\alpha\in \Phi$ to coordinates $q^\alpha$, each subvariety $H_\alpha = \{q^\alpha=1\}\subset T^k$ is given in the local coordinate chart $U_\sigma = \mathrm{Spec\;}\mathbb{C}[q_1,\ldots,q_k]$ as $q_1^{a_1}\cdot\ldots\cdot q_k^{a_k} = 1$ for $a_1,\ldots,a_k\in\mathbb{N}$, where $\pm\alpha = a_1\alpha_1+\ldots+a_k\alpha_k$.  
\end{remark}
\subsection{Compactifying Toric Arrangements}\label{S3.2}
In this section, we will discuss compactifications of toric arrangements, following the construction given by~\cite{M11}.  
Fix a torus $T^k$, and fix a finite set $\Phi^+\subset \Lambda$ of circuit vectors and for each $\alpha\in\Phi^+$, take subtori $H_\alpha\subset T^k$ and $\cH_{\Phi^+}$ to be as in Definition \ref{D3.1.1}.  For convenience, we will use the following notation.  If $T^k$ is written in terms of coordinates, as $T^k = \mathrm{Spec\;}\mathbb{C}[q_1^{\pm},\ldots,q_k^\pm]$, then for each $\alpha = (\alpha_1,\ldots,\alpha_k)\in\Phi^+$, we will write $q^\alpha:= q_1^{\alpha_1}\cdot\ldots\cdot q_k^{\alpha_k}$.  In particular, we can write 
\begin{align*}
    H_\alpha = \mathrm{Spec\;}\mathbb{C}[q_1^\pm,\ldots,q_k^\pm]/\langle q^\alpha-1\rangle.
\end{align*}
In this section, we describe a partial compactification of the parameter space 
\begin{align*}
    T^\mathrm{reg} = T^k\setminus\bigcup_{\alpha\in\Phi^+}H_\alpha,
\end{align*}
in terms of open charts, following~\cite{M11}.  The construction will involve taking the subvarieties $H_\alpha$ and blowing up along the connected components of their intersections.  The charts will then have a description in terms of maximal nested sets, given in Definition \ref{D3.2.6}.

Given a collection of subtori $H_{\alpha_1},\ldots,H_{\alpha_l}\in \cH_{\Phi^+}$, the intersection $H_{\alpha_1}\cap\ldots\cap H_{\alpha_l}$ need not be connected in general.  The example below is illustrative.  

\begin{example}
    Consider the hypertoric variety of Example \ref{E1} defined by the map 
    \begin{align*}
        [a]=\begin{pmatrix}
            1 & 0 & 0 &  0 & -1 &  1 \\
            0 & 1 & 0 &  1 &  0 & -1 \\
            0 & 0 & 1 & -1 &  1 &  0
        \end{pmatrix}\colon \mathfrak{t}_\mathbb{R}^6\longrightarrow\mathfrak{t}_\mathbb{R}^3,
    \end{align*}
    whose corresponding circuit vectors are given by $\Phi^+ = \{\varepsilon_1,\varepsilon_2,\varepsilon_3,\varepsilon_1+\varepsilon_2,\varepsilon_1+\varepsilon_3,\varepsilon_2+\varepsilon_3,\varepsilon_1+\varepsilon_2+\varepsilon_3\}$.  By simultaneously solving the equations
    \begin{align*}
        q_1q_2 &= 1\\
        q_1q_3 &= 1\\
        q_2q_3 &= 1,
    \end{align*}
    one finds that $H_{\varepsilon_1+\varepsilon_2}\cap H_{\varepsilon_1+\varepsilon_3}\cap H_{\varepsilon_2+\varepsilon_3} = \{(1,1,1),(-1,-1,-1)\}$.
\end{example}

Returning to the general case, we make the following definition.

\begin{definition}\label{D3.2.1}
    Define a {\it layer} to be any connected component of the intersection $H_{\alpha_1}\cap\ldots\cap H_{\alpha_l}$, for some $l\in\{0,\ldots,k\}$ and some $\alpha_1,\ldots,\alpha_l\in\Phi^+$.  Let $\cC(\Phi^+)$ be the set of all possible layers.  
    Equip $(\cC(\Phi^+),\subset)$ with the structure of a poset with the ordering relation the inclusion of subsets.    
\end{definition}

\begin{remark}
    By convention, we will take $\bigcap_{\alpha\in\emptyset\subset A}H_{\alpha} = T^k$ to be a layer.  
\end{remark}

\begin{remark}
    We can describe a layer more explicitly as follows.  Start with the characters $\alpha_1,\ldots,\alpha_m\in\Phi^+$ for some $m>0$, and consider the intersection $H_{\alpha_1}\cap\ldots\cap H_{\alpha_m}$ which is given as the kernel of the following sequence:
    \begin{align*}
        1\longrightarrow H_{\alpha_1}\cap\ldots\cap H_{\alpha_m}\longrightarrow T^k\xrightarrow[]{\;\;[\alpha]\;\;}T^m
    \end{align*}
    where $[\alpha]$ has the following form:
    \begin{align*}
        [\alpha] = \begin{pmatrix}
            \alpha_1^{(1)}&\cdots & \alpha_1^{(k)}\\
            \vdots & \ddots & \vdots\\
            \alpha_m^{(1)} & \cdots & \alpha_m^{(k)}
        \end{pmatrix},
    \end{align*}
    with entries in $\mathbb{Z}$.  By changing the coordinates on $T^k$ and $T^m$ one can find invertible matrices $[K]\in \mathrm{GL}_k(\mathbb{Z})$ and $[M]\in \mathrm{GL}_m(\mathbb{Z})$ so that $[\alpha]$ can be written in Smith normal form~\cite{MITOCW11}:
    \begin{align*}
        [M][\alpha][K] 
        = \begin{pmatrix}
            \begin{array}{ccc|ccc}
            d_1 & & 0 & & &   \\
            & \ddots & & & 0 & \\
            0 & & d_l & & &  \\
            \hline
            &0 & & &0
            \end{array}
        \end{pmatrix}
    \end{align*}
    for integers $d_1,\ldots,d_l>0$ with $d_1\mid \ldots\mid d_l$.  Thus, we have 
    \begin{align*}
        H_{\alpha_1}\cap\ldots\cap H_{\alpha_m} 
        &\simeq \mathrm{Spec\;}\frac{\mathbb{C}[q_1^{\pm},\ldots,q_k^\pm]}{\langle q_1^{d_1}-1,\ldots,q_l^{d_l}-1\rangle}\\
        &\simeq \mathrm{Spec\;}\frac{\mathbb{C}[q_1^{\pm}]}{\langle q_1^{d_1}-1\rangle}\times\ldots\times \mathrm{Spec\;}\frac{\mathbb{C}[q_l^{\pm}]}{\langle q_l^{d_l}-1\rangle}\times \mathrm{Spec\;}\mathbb{C}[q_{l+1}^\pm,\ldots,q_k^{\pm}].
    \end{align*}
    Thus, each connected component $C\subset H_{\alpha_1}\cap\ldots\cap H_{\alpha_l}$ is of the form $C = \{(\zeta_{d_1}^{i_1},\ldots,\zeta_{d_l}^{i_l})\}\times T^{k-l}$, where $\zeta_{d_j}$ is the $d_j$--root of unity and $0\leq i_j< d_j$.  In particular, if we take the point $p = ((\zeta_{d_1}^{i_1},\ldots,\zeta_{d_l}^{i_l}),(1,\ldots,1))\in T^k$, then $p^{-1}C = (H_{\alpha_1}\cap\ldots\cap H_{\alpha_l})_1$ is a layer containing $1$ and $C = p\;C_1$ is a coset of $C_1$.   
\end{remark}    

\begin{definition}\label{D3.2.2}
    If $\Delta\subset\Lambda$ is a sublattice, then we define the {\it completion of the sublattice} to be 
    \begin{align*}
        \overline{\Delta}:= \langle \Delta\rangle_{\bC}\cap \Lambda.
    \end{align*}
    If $\Phi^+$ is a finite set, and $A\subset \Phi^+$ is a subset, then we say that $A = \bigsqcup_i A_i$ is a {\it decomposition of }$A$ if we have 
    \begin{align*}
        \overline{\langle A\rangle_\bZ} = \bigoplus_i\overline{\langle A_i\rangle_\bZ}
    \end{align*}
    for $A_i\neq \emptyset$, where we are taking completions of the sublattices $\langle A_i\rangle_\bZ, \langle A\rangle_\bZ\subset \Lambda$.  We say that $A$ is {\it irreducible} if $A$ does not admit such a decomposition.  In this decomposition, if each of the $A_i$ are themselves irreducible, then they are termed {\it irreducible factors} of $A$.    
\end{definition}

\begin{definition}\label{D3.2.3}
    If $A\subset \Phi^+$ is a subset of the finite set $\Phi^+$, then we define the {\it completion} of $A$ as 
    \begin{align*}
      \overline{A} := \langle A\rangle_\mathbb{C}\cap \Phi^+,
    \end{align*}
    and we say that $A$ is {\it complete} if $\overline{A} = A$.      
\end{definition}

\begin{example}
    For the finite set $\{(1,1),(1,-1)\}\subset \mathbb{Z}^2$, then we have 
    \begin{align*}
        \bC^2 &= \langle \;\{(1,1),(1,-1)\}\;\rangle_\bC\\
        &= \langle\;(1,1)\;\rangle_\bC\oplus \langle\;(1,-1)\;\rangle_\bC
        = \bC\cdot (1,1)\oplus \bC\cdot (1,-1),
    \end{align*}
    while on the other hand we have 
    \begin{align*}
        \bZ^2 &= \overline{\langle\; \{(1,1),(1,-1)\}\;\rangle_\bZ}\\
        &\supsetneq \overline{\langle\; (1,1)\;\rangle_\bZ}\oplus \overline{\langle\;(1,-1)\;\rangle_\bZ} = \bZ\cdot (1,1)\oplus \bZ\cdot (1,-1).  
    \end{align*}
    Therefore, $\{(1,1),(1,-1)\}$ is $\bC$--reducible, but not ($\mathbb{Z}$--)reducible.
\end{example}

\begin{example}
    For the finite set $\{(1,1,0),(1,0,1),(0,1,1)\}\subset \bZ^3$ we have 
    \begin{align*}
        \bC^3 &= \langle\;\{(1,1,0),(1,0,1),(0,1,1)\}\;\rangle_\bC\\
        &= \langle\;(1,1,0)\;\rangle_\bC\oplus \langle\;(1,0,1)\;\rangle_\bC\oplus \langle\;(0,1,1)\;\rangle_\bC = \bC\cdot (1,1,0)\oplus \bC\cdot (1,0,1)\oplus \bC\cdot (0,1,1),
    \end{align*}
    while on the other hand we have, 
    \begin{align*}
        \bZ^3 &= \overline{\langle\;\{(1,1,0),(1,0,1),(0,1,1)\}\;\rangle_\bZ}\\
        &\supsetneq \overline{\langle\;(1,1,0)\; \rangle_\bZ}\oplus\overline{\langle\;(1,0,1)\;\rangle_\bZ}\oplus \overline{\langle\;(0,1,1)\; \rangle_\bZ} = \bZ\cdot (1,1,0)\oplus \bZ\cdot (1,0,1)\oplus \bZ\cdot (0,1,1).
    \end{align*}
    Therefore, $\{(1,1,0),(1,0,1),(0,1,1)\}$ is $\bC$--reducible, but not ($\bZ$--)reducible.
\end{example}

We now turn to characterize which hyperplanes can intersect to form a given layer $C\in\cC(\Phi^+)$. 

\begin{lemma}\label{L3.2.1}
    The following map is an inclusion-reversing bijection: 
    \begin{align*}
        \left\{A\subset \Phi^+\mid \overline{A} = A\right\}&\longrightarrow\left\{C\in\cC(\Phi^+)\mid 1\in C\right\}\\
        A&\longmapsto C(A):= \left(\bigcap_{\alpha\in A}H_\alpha\right)_{1},
    \end{align*}
    where $\left(\bigcap_{\alpha\in A}H_\alpha\right)_{1}$ is the connected component of $\bigcap_{\alpha\in A}H_\alpha$ which contains $1\in T^k$.  
    
    The inverse is given by 
    \begin{align*}
        \left\{C\in\cC(\Phi^+)\mid 1\in C\right\}&\longrightarrow\left\{A\subset \Phi^+\mid \overline{A} = A\right\}\\
        C&\longmapsto \Phi^+_C := \{\alpha\in\Phi^+\mid H_\alpha\supset C\}.
    \end{align*}
\end{lemma}

\begin{remark}
    It can be verified directly that if $A = \emptyset$, then $C(\emptyset) = T^k$ and if $A = \Phi^+$, then $C(\Phi^+) = \{1\}$.  Moreover, $\Phi^+_{T^k} = \emptyset$ and $\Phi^+_{\{1\}} = \Phi^+$.   
\end{remark}

\begin{proof}
    We first observe that the above maps are well-defined.  Let $C\in\cC(\Phi^+)$ be a layer containing $1\in C$.  We first show that $\overline{\Phi^+_C} = \langle \Phi^+_C\rangle_{\mathbb{C}}\cap \Phi^+ = \Phi^+_C$.  Suppose $\alpha\in \overline{\Phi^+_C}$.  Since $\langle \Phi^+_C\rangle_{\mathbb{C}}\cap \Phi^+ = \langle \Phi^+_C\rangle_{\mathbb{Q}}\cap \Phi^+$, it follows that for some $m,m_1,\ldots,m_{l}\in\mathbb{Z}$, we can write $m\alpha = m_1\alpha_1 + \ldots + m_{l}\alpha_{l}$, where $\alpha_j\in\Phi^+_C$, $j=1,\ldots,l$.  Thus, we find that $H_{a\alpha}\supset C$.  But $H_{a\alpha} = \bigsqcup_{i=1}^m\left\{t\in T^k\mid \alpha(t)-\zeta_m^i = 0\right\}$.  Since $1\in C$ by assumption and since $C$ is connected, it follows that $C\subset \left\{t\in T^k\mid \alpha(t) - 1 = 0\right\} = H_\alpha$, proving that $\alpha\in\Phi^+_C$ and hence, $\overline{\Phi^+_C} = \Phi^+_C$.

    On the other hand, if $A\subset \Phi^+$ is a complete subset, then we know that $1\in \bigcap_{\alpha\in\Phi^+}H_\alpha\subset \bigcap_{\alpha\in A}H_\alpha$, so $\left(\bigcap_{\alpha\in A}H_\alpha\right)_{1}$ is such a layer which contains $1\in T^k$.  

    We now show that the above two maps are inverses to each other.  We observe that by the definition of a layer, Definition \ref{D3.2.1}, the map $A\longmapsto C(A)$ is surjective.  If we can show that $\Phi^+_{C(A)} = A$, then it will also follow that $C(\Phi^+_C) = C(\Phi^+_{C(A)}) = C(A) = C$, and we would be done.  
    
    To prove the claim, let $C\in\cC(\Phi^+)$ be a layer containing $1$ and write it as $C = C(A)$.  choose a linearly independent set $\{\alpha_1,\ldots,\alpha_l\}\subset A$ so that we can write $C = H_{\alpha_1}\cap\ldots\cap H_{\alpha_l}$, and $A = \overline{\{\alpha_1,\ldots,\alpha_l\}}$.  Then $C$ is a subtorus which fits into the following split exact sequence of abelian groups:
    \begin{align*}
        1\longrightarrow C \longrightarrow T^k\xrightarrow[]{\begin{pmatrix} - & \alpha_1 & - \\
        &\vdots &\\
        - & \alpha_l & - \end{pmatrix}} T^l\longrightarrow 1.
    \end{align*}

    The induced sequence of character groups is given by 
    \begin{align*}
        0\longrightarrow \Lambda_C\longrightarrow \Lambda\longrightarrow X^*(C)\longrightarrow 0,
    \end{align*}
    and by definition of $\Phi^+_C$ and $\Lambda_C$, we have $\Phi^+_C = \Lambda_C\cap \Phi^+ = \overline{\{\alpha_1,\ldots,\alpha_l\}} = A$ as the corresponding complete subset.  Thus, $\Phi^+_{C(A)} = \Phi^+_C = A$, proving the claim.  
\end{proof}

We now generalize the situation to the case to layers that don't contain $1\in T^k$.  We will first adapt Definition \ref{D3.2.3} to the more general setting.    

\begin{definition}\label{D3.2.4}
    Define $\cC_0(\Phi^+)$ to be the set of zero-dimensional layers.  In particular, $\cC_0(\Phi^+)\subset T^k$.
    
    For any zero-dimensional layer $p\in\cC_0(\Phi^+)$, define 
    \begin{align*}
        \Phi^+_p := \Phi^+_{\{p\}} = \{\alpha\in \Phi^+\mid p\in H_\alpha\}.
    \end{align*} 

    We define the {\it completion in $\Phi^+_p$} of the finite set $A$ as $\overline{(A)}_p = \langle A\rangle_\mathbb{C}\cap \Phi^+_p = \overline{A}\cap \Phi^+_p$, where $\overline{A}$ is the completion as in Definition \ref{D3.2.3}.  We say that $A$ is {\it complete in $\Phi^+_p$} if $\overline{(A)}_p = A$.  
\end{definition}

\begin{lemma}\label{L3.2.2}
    The following map is an inclusion-reversing bijection: 
    \begin{align*}
        \left\{A\subset \Phi^+_p\mid \overline{(A)}_p= A\right\}&\longrightarrow \left\{C\in\cC(\Phi^+)\mid p\in C\right\}\\
        A&\longmapsto C_p(A):= \left(\bigcap_{\alpha\in A}H_\alpha\right)_p.
    \end{align*}
    The inverse is given by 
    \begin{align*}
        \left\{C\in\cC(\Phi^+)\mid p\in C\right\}&\longrightarrow \left\{A\subset \Phi^+_p\mid \overline{(A)}_p= A\right\}\\
        C&\longmapsto \Phi_{C,p}^+ := \{\alpha\in\Phi_p^+\mid H_\alpha\supset C\}
    \end{align*}
\end{lemma}

\begin{proof}
    Let $\cC(\Phi^+_p)\subset \cC(\Phi^+)$ be the sub-poset of connected components of intersections $H_{\alpha_1}\cap \ldots\cap H_{\alpha_l}$ for $\alpha_1,\ldots,\alpha_l\in \Phi^+_p$.  Then we can replace the set $\Phi^+$ with $\Phi^+_p$ in Lemma \ref{L3.2.1} to obtain the following bijection:
    \begin{align*}
        \left\{A\subset \Phi^+_p\mid \overline{(A)}_p= A\right\}&\longrightarrow \left\{C\in\cC(\Phi^+_p)\mid 1\in C\right\}\\
        A&\longmapsto C(A).
    \end{align*}
    From the remark immediately following Definition \ref{D3.2.1}, we have the following bijection:
    \begin{align*}
        \left\{C\in\cC(\Phi^+_p)\mid 1\in C\right\}&\longrightarrow \left\{C\in\cC(\Phi^+_p)\mid p\in C\right\}\\
        C_1&\longmapsto p\;C_1.
    \end{align*}
    Combining these two bijections with the following equality of sets 
    \begin{align*}
        \left\{C\in\cC(\Phi^+_p)\mid p\in C\right\} = \left\{C\in\cC(\Phi^+)\mid p\in C\right\}
    \end{align*}
    shows that $C_p(-)$ is a bijection, proving the first statement.

    As for the second statement, start with a layer $C\in\cC(\Phi^+)$ for which $p\in C$.  Replacing $\Phi^+$ for $\Phi_p^+$ in the second statement of Lemma \ref{L3.2.1} for the inverse bijection, it follows that the map 
    \begin{align*}
        \left\{C\in\cC(\Phi^+_p)\mid 1\in C\right\}&\longrightarrow \left\{A\subset \Phi^+_p\mid \overline{(A)}_p= A\right\}\\
        C&\longmapsto \Phi_{C,p}^+
    \end{align*}
    is the inverse to $A\mapsto C_p(A)$.  We next precompose this inverse bijection with the following map
    \begin{align*}
        \left\{C\in\cC(\Phi^+_p)\mid p\in C\right\}&\longrightarrow\left\{C\in\cC(\Phi^+_p)\mid 1\in C\right\}\\
        C_p&\longmapsto p^{-1}C_p,
    \end{align*}
    by observing that for any $C_p\in \cC(\Phi^+_p)$ containing $p$, for any $t\in C_p$ and for any character $\alpha\in \Phi_p^+ \subset X^*(T^k)$, we have $\alpha(p^{-1}t) = \alpha(p)^{-1}\alpha(t) = \alpha(t)$, for all $t\in C_p$ by definition of $\Phi_p^+$.  Thus, for all $C\in\cC(\Phi^+)$ containing $p$ the inverse bijection is given by  
    \begin{align*}
        C&\longmapsto \Phi_{p^{-1}C,p}^+ = \{\alpha\in\Phi_p^+\mid H_\alpha\supset p^{-1}C\} = \{\alpha\in\Phi_p^+\mid H_\alpha\supset C\} = \Phi_{C,p}^+,
    \end{align*}      
    completing the proof.  
\end{proof}

\begin{remark}
    For any layer $C\in\cC(\Phi_p^+)$ containing $p$, we also have the equality of sets 
    \begin{align*}
        \Phi_{C,p}^+ = \{\alpha\in\Phi_p^+\mid H_\alpha\supset C\} = \{\alpha\in \Phi^+\mid H_\alpha\supset C\} = \Phi_C^+,
    \end{align*}
\end{remark}

\begin{definition}\label{D3.2.5}
    For any $p\in\cC_0(\Phi^+)$, define
    \begin{align*}
        \cC_p(\Phi^+) = \{C\in\cC(\Phi^+)\mid p\in C\}.
    \end{align*}  
    Then Lemma \ref{L3.2.2} may be understood as a bijection between complete subsets in $\Phi^+_p$ and layers in $\cC_p(\Phi^+)$.
    
    For any $C\in\cC_p(\Phi^+)$, we say that $C\subset \bigcap_i C_i$ is a {\it decomposition in $\cC_p(\Phi^+)$} of $C$ if we have $C = C(A)$ and $C_i = C_p(A_i)$ and $C_i\neq C$ for all $i$; and if $A = \bigsqcup_i A_i$ is a decomposition of $A$, in the sense of Definition \ref{D3.2.2}.  We say that $C$ is {\it irreducible in $\cC_p(\Phi^+)$} if it admits no such decomposition.  For the decomposition $C\subset \bigcap_i C_i$ in $\cC_p(\Phi^+)$, if $C_i$ are themselves irreducible in $\cC_p(\Phi^+)$, then we say that the $C_i$ are {\it irreducible factors in $\cC_p(\Phi^+)$} of $C$.  Let $\cI_p(\Phi^+)$ be the set of irreducible layers in $\Phi_p^+$ for each $p$ and let $\cI(\Phi^+) = \bigcup_{p\in\cC_0(\Phi^+)}\cI_p(\Phi^+)$.        
\end{definition}

\begin{remark}
    By comparing Definition \ref{D3.2.4} and Definition \ref{D3.2.5}, it follows that the bijection of Lemma \ref{L3.2.2} takes decompositions into complete factors in $\Phi^+_p$ to decompositions in $\cC_p(\Phi^+)$ and hence, this bijection takes complete irreducible subsets in $\Phi^+_p$ to irreducible subsets in $\cC_p(\Phi^+)$. 

    Moreover, we find that if $p,q\in \cC_0(\Phi^+)$ are two basepoints, and $C\in \cC_p(\Phi^+)\cap \cC_q(\Phi^+)$, then if $C = \bigcap_i C_i$ is a decomposition in $\Phi_p^+$ then we know that $p,q\in C\subset C_i$ for each $i$ and hence, $C = \bigcap_i C_i$ is a decomposition in $\cC_q(\Phi^+)$ as well.  Thus, if $C\in\cI_p(\Phi^+)$ and $q\in C$, then $C\in\cI_q(\Phi^+)$ as well. 
\end{remark}

\begin{lemma}
    Let $p\in \cC(\Phi^+)$.  Then $\Phi^+_p\subset \Phi^+$ is complete in the sense of Definition \ref{D3.2.3} iff $p = 1\in T^k$. 
\end{lemma}

\begin{proof}
    We first observe that $\Phi^+_1 = \Phi^+$, because $1\in H_\alpha$ for every $\alpha\in\Phi^+$, by definition.  Thus, $\Phi^+_p$ is complete for $p = 1$.  
    
    To prove the converse, suppose $p\neq 1$.  Then we show that $\Phi^+_p$ is not complete.  We do so by showing that $\overline{\Phi^+_p}\setminus\Phi^+_p\neq \emptyset$.    Recall from the remark following Lemma \ref{L2.1.2} that the set $\Phi^+$ of circuit vectors contains a standard basis $\varepsilon_1,\ldots,\varepsilon_k\in X^*(T^k)$.  By the remark following Definition \ref{D3.2.1}, we can write $p = (\zeta_{d_1}^{i_1},\ldots,\zeta_{d_k}^{i_k})$, for $d_1,\ldots,d_k\in\mathbb{Z}$ and for $0\leq i_j<d_j$, $j=1,\ldots,k$.  Thus, there exists some $j_0$ for which $\zeta_{d_{j_0}}^{i_{j_0}}\neq 1$.  Thus, for this particular $j_0$, we find that $\varepsilon_{j_0}(p)\neq 1$.  Therefore, $\varepsilon_p\in \Phi^+\setminus\Phi^+_p$.  We claim that $\overline{\Phi^+_p} = \Phi^+$.  Then this claim implies that $\overline{\Phi^+_p}\setminus\Phi^+_p = \Phi^+\setminus\Phi^+_p \neq \emptyset$, completing the proof.  
    
    To prove the claim we observe that for each $j=1,\ldots,k$ we have $d_j\varepsilon_j\in \Phi^+_p$ because  and hence, $\varepsilon_j = d_j^{-1}\cdot d_j\varepsilon_j\in \langle \Phi^+_p\rangle_{\mathbb{Q}}\cap \Phi^+ = \overline{\Phi^+_p}$.  So $\overline{\Phi^+_p}\subset \Phi^+$ is a sublattice that contains $\varepsilon_1,\ldots,\varepsilon_k$ and therefore, $\overline{\Phi^+_p} = \Phi^+$.          
\end{proof}

\begin{example}
    In the example given at the beginning of this section, we find that $\Phi^+_{(1,1,1)} = \Phi^+$, which is complete in the sense of Definition \ref{D3.2.3} while $\Phi^+_{(-1,-1,-1)} = \{\varepsilon_1+\varepsilon_2,\varepsilon_1+\varepsilon_3,\varepsilon_2+\varepsilon_3\}$ is not.
\end{example}

We now turn to the question of blowing up a given toric layer $C\in\cC(\Phi^+_C)$ inside $T^k$.

\begin{lemma}\label{L3.2.3}
    Let $C\in\cC(\Phi^+)$ be a layer in $T^k$, and let $\mathfrak{a}_C = \langle q^\alpha - a\mid \alpha\in\overline{\Lambda_C},\; a = \alpha\rvert_C\rangle\subset \mathbb{C}[T^k]$ be the ideal of polynomial functions on $T^k$ which vanish on $C$.  Then $C$ is a locally complete intersection, in the sense that $\mathfrak{a}_C$ is generated by a regular sequence. 
    That is, one can write $\mathfrak{a}_C = \langle f_1,\ldots,f_l\rangle$ for generators $f_1,\ldots,f_l\in \mathfrak{a}_C$, $l\leq k$ satisfying the property that 
    \begin{enumerate}
        \item $\mathbb{C}[T^k]/\langle f_1,\ldots,f_l\rangle\neq 0$
        \item For each $i=1,\ldots,l$, we have $f_i$ is a nonzerodivisor in $\mathbb{C}[T^k]/\langle f_1,\ldots,f_{i-1}\rangle$.  
    \end{enumerate}
\end{lemma}

\begin{proof}
    Let $\alpha_1,\ldots,\alpha_l\in\overline{\Lambda_C}$ be a $\mathbb{Z}$--basis for $\overline{\Lambda_C}$ and let $a_1,\ldots,a_l\in\mathbb{C}$ be the constant values assumed by the $\alpha_1',\ldots,\alpha_l'$ on $C$ (so the $a_j$ are roots of unity).  Then we can write $\mathfrak{a}_C = \langle q^{\alpha_j}-a_j\mid j=1,\ldots,l\rangle$.  In following the remark immediately after Definition \ref{D3.2.1}, we form the full rank matrix from the components of the $\alpha_j$ as  
    \begin{align*}
        [\alpha] = \begin{pmatrix}
            (\alpha_1)^{(1)}&\cdots & (\alpha_1)^{(k)}\\
            \vdots & \ddots &\vdots\\
            (\alpha_l)^{(1)} &\cdots &(\alpha_l)^{(k)}
        \end{pmatrix}.
    \end{align*}
    By repeatedly changing the generators of the ideal $\mathfrak{a}_C$ and by judiciously changing coordinates on $T^k$, we can bring the matrix $[\alpha]$ to Smith normal form.  Thus,  $\mathfrak{a}_C = \langle q_j^{d_j} - a_j\mid j=1,\ldots,l\rangle\subset \mathbb{C}[q_1^{\pm},\ldots,q_k^{\pm}]$ for some $d_1,\ldots,d_l\in\mathbb{Z}$, with $a_1,\ldots,a_l\in\mathbb{C}$.  But in factoring $q_j^{d_j}-a_j = \prod_{j'=1}^{d_j}\left(q-\zeta_{d_j}^{j'}\right)$ into roots of unity, we find that the connectedness of $C$, together with the fact that $q_j^{d_j}-a_j$ are generators of $\mathfrak{a}_C$, imply that one such factor $q_j-\zeta_{d_j}^{j'}\in\mathfrak{a}_C$.  Repeating this argument for each $j=1,\ldots,l$, it follows that the generators of $\mathfrak{a}_C$ must be written as $\mathfrak{a}_C = \langle q_1-\zeta_{d_1}^{i_1},\ldots,q_l-\zeta_{d_l}^{i_l} \rangle$.  By inspection, the sequence $\left\{q_1-\zeta_{d_1}^{i_1},\ldots,q_l-\zeta_{d_l}^{i_l}\right\}$ satisfies properties (a) and (b) in the statement of the Lemma and hence, it is a regular sequence.     
\end{proof}

\begin{remark}
    In particular, since $C$ is closed, this proof implies that $C = V(q_1-\zeta_{d_1}^{i_1},\ldots,q_l-\zeta_{d_l}^{i_l}) = V(\mathfrak{a}_C) \subset T^k$.
\end{remark}

The following lemma is standard (see, e.g. Theorem 22.3.8 of~\cite{V25}).  

\begin{lemma}
    If $A$ is a commutative ring and $\mathfrak{a}\subset A$ is generated by a regular sequence, then there is a natural isomorphism
    \begin{align*}
        \mathrm{Sym}^n_A(\mathfrak{a}/\mathfrak{a}^2)\simeq \mathfrak{a}^n/\mathfrak{a}^{n+1}.
    \end{align*}
\end{lemma}

\begin{definition}
    Let $C\in\cC(\Phi^+)$ be any layer in $T^k$ and let $\mathfrak{a}_C$ be the corresponding ideal.  By Lemma \ref{L3.2.3}, we can choose generators $\mathfrak{a}_C = \langle q^{\alpha_j}-a_j\mid j=1,\ldots,l\rangle\subset \mathbb{C}[T^k]$, so that $C = V(\mathfrak{a}_C)$.  
    %Let $C\in\cC(\Phi^+)$ be any layer in $T^k$,  take a $\mathbb{Z}$--basis $\{\alpha_1,\ldots,\alpha_l\}\subset\overline{\Lambda_C}$ of the sublattice $\overline{\Lambda_C}$ and define the corresponding ideal $\mathfrak{a}_C = \langle q^{\alpha_i} - a_i\mid i=1,\ldots,l\rangle\subset \mathbb{C}[q_1^{\pm},\ldots,q_k^{\pm}]$, so that $C = V(\mathfrak{a}_C)$.  
    The blow-up of $T^k$ along $V(\mathfrak{a}_C)$ is given by 
    \begin{align*}
        \mathrm{Bl}_{C}T^k 
        &= \mathrm{Proj}_{\mathbb{C}[T^k]}\left(\mathbb{C}[T^k]\oplus \mathfrak{a}_C\oplus \mathfrak{a}_C^2\oplus \ldots\right),
    \end{align*}
    with a natural projection map $\mathrm{Bl}_{C}T^k\longrightarrow T^k$.  The exceptional fibre can be computed using Lemma \ref{L3.2.3}:
    \begin{align*}
        \mathrm{E}_{C}T^k
        &= \mathrm{Proj}\left(\frac{\mathbb{C}[T^k]}{\mathfrak{a}_C}\oplus \frac{\mathfrak{a}_C}{\mathfrak{a}_C^2}\oplus \frac{\mathfrak{a}_C^2}{\mathfrak{a}_C^3}\oplus\ldots\right)\\
        &= \mathrm{Proj}\left(\mathrm{Sym}^\bullet_{\frac{\mathbb{C}[T^k]}{\mathfrak{a}_C}}\frac{\mathfrak{a}_C}{\mathfrak{a}_C^2}\right) \\
        &= \mathbb{P}\left(\frac{\mathfrak{a}_C}{\mathfrak{a}_C^2}\right)^*
    \end{align*}
    with a natural map $E_CT^k\longrightarrow C$.
    
    In defining $N_T(C) = (\mathfrak{a}_C/\mathfrak{a}_C^2)^*$, the exceptional fibre $E_C T^k = \mathbb{P}(N_{T^k}C)$, forms a trivial bundle over $C$.  
    
    Define $\mathbb{P}_C$ to be the fibre of $E_CT^k$ over $C$.  If we write $\mathfrak{a}_C = \langle q^{\alpha_j}-a_j\mid j=1,\ldots,l\rangle$, then we have $\mathbb{P}_C = \mathbb{P}^l$. 
\end{definition}

From this description, we have the following lemma:

\begin{lemma}
    The blow-up $\mathrm{Bl}_{C}T^k$ is given as the closure of the image of the following embedding
    \begin{align*}
       \iota\times \varphi_C\;\colon\;T^k\setminus C &\longrightarrow T^k\times \mathbb{P}_C\\
       t&\longmapsto (t,[\alpha_1(t)-a_1,\ldots, \alpha_l(t)-a_l]),
    \end{align*}
    where $\{\alpha_1,\ldots,\alpha_l\}\in\overline{\Lambda_C}$ is an integral basis, and $a_1,\ldots,a_l\in \mathbb{C}$ are the constant values assumed by $\alpha_i$ on $C$.      
\end{lemma} 

\begin{definition}
    The collection of maps $\{\varphi_C\mid C\in\cI(\Phi^+)\}$, together with the embedding $\iota\colon (T^k)^{\mathrm{reg}}\hookrightarrow T^k$ defines the map 
    \begin{align*}
        \psi = \iota\times \prod_{C\in\cI(\Phi^+)}\varphi_C\;\colon\; (T^k)^{\mathrm{reg}}\hookrightarrow T^k\times\prod_{C\in\cI(\Phi^+)}\mathbb{P}_C.
    \end{align*}
    Define $Z_{\Phi^+} = \overline{\psi((T^k)^{\mathrm{reg}})}$.
\end{definition} 

\begin{remark}
    We note that the map $\varphi_C$ is independent of the basis $\{\alpha_1,\ldots,\alpha_l\}$ chosen and hence, $Z_\Phi^+$ is also basis--independent. 
\end{remark}

\begin{lemma}
    The space $Z_{\Phi^+}$ is obtained by successively blowing up the strata in $\cI(\Phi^+)$ in order of increasing dimension.   
\end{lemma}

\begin{proof}
    See Theorem 1.3 of~\cite{L09} and Remark 2.4 of~\cite{M11}.   
\end{proof}

We now define the open sets that are used to form an open cover of $Z_\Phi^+$.  Throughout, the remainder of this paper, we take $p\in\cC_0(\Phi^+)$.    

\begin{definition}
    Define a {\it flag} $\mathcal{F}^*$ to be a nested sequence of complete subsets $\emptyset \subsetneq A_1\subsetneq\ldots\subsetneq A_s = \Phi^+_p$ for some $p$.  We say that $\mathcal{F}^*$ is a {\it maximal flag} if we have $s = k$.  By completeness, we observe that $\mathrm{dim}\langle A_i\rangle_{\mathbb{C}}<\mathrm{dim}\langle A_{i+1}\rangle_{\mathbb{C}}$ so the maximal flags are precisely those flags which are not contained in a larger flag.  Define a {\it flag of layers} $\mathcal{F}$ as a nested sequence of layers $T^k \supsetneq C_1\supsetneq\ldots\supsetneq C_s = \{p\}$ for some $p$.  This flag $\mathcal{F}$ is said to be a {\it maximal flag of layers} if $s = k$.
\end{definition}    

\begin{definition}\label{D3.2.6}
    For any $p$, define a {\it (maximal) nested set} $\mathscr{S}_p^*$ as a collection of complete subsets in $\mathcal{P}(\Phi^+_p)$ satisfying the following property: there exists a (maximal) flag $\mathcal{F}^* = (\emptyset \subsetneq A_1\subsetneq \ldots\subsetneq A_{s-1}\subsetneq A_s = \Phi^+_p)$ for which $\mathscr{S}_p^*$ is the set of all irreducible factors in $\Phi^+_p$ of the $A_i$, $i=1,\ldots,s$.  Let $\mathcal{M}_p^*$ be the set of all maximal nested sets corresponding to flags in $\Phi^+_p$.  
    
    For any $p$, define a {\it (maximal) nested set of layers} $\mathscr{S}_p\subset \cI_p(\Phi^+)$ as a collection of layers in $\cC_p(\Phi^+)$ satisfying the following property: there exists a (maximal) flag of layers $\mathcal{F} = (T^k\supsetneq C_1\supsetneq \ldots\supsetneq C_s = \{p\})$ for which $\mathscr{S}_p$ is the set of all irreducible factors in $\cC_p(\Phi^+)$ of the $C_i$, $i=1,\ldots,s$.  Let $\mathcal{M}_p$ be the set of all maximal nested sets of layers corresponding to flags of layers terminating at $\Phi^+_p$.  
    
    Define $\mathcal{M} = \bigsqcup_{p\in\cC_0(\Phi^+)}\mathcal{M}_p$.
\end{definition}

\begin{remark}
    It is the set $\mathcal{M}$ that will be used to index the coordinate charts on the partial compactification of $T^\mathrm{reg}$.   
\end{remark}

\begin{remark}
    Definition \ref{D3.2.6} is equivalent to saying that the nested sets $\mathscr{S}_p^*$ are precisely the sets for which any mutually incomparable elements $B_1,\ldots,B_r\in\mathscr{S}_p^*$ in $\Phi^+_p$ have a complete union $B:= B_1\cup\ldots\cup B_r$, which presents $B$ as a decomposition into irreducible factors $B_1,\ldots,B_r$ in $\Phi^+_p$.  Moreover, $\mathscr{S}_p^*$ is maximal iff it is not contained in a larger nested set.
\end{remark}

\begin{lemma}
    For each $p\in\cC_0(\Phi^+)$, the bijection of Lemma \ref{L3.2.2} induces the following bijection:
    \begin{align*}
        \mathcal{M}_p^*&\longrightarrow \mathcal{M}_p\\
        \mathscr{S}_p^* = \{A_1,\ldots,A_k\}&\longmapsto \{C_p(A_1),\ldots,C_p(A_k)\} = \mathscr{S}_p
    \end{align*}
\end{lemma}

\begin{proof}
    This follows directly from the definitions in Lemma \ref{L3.2.2}.     
\end{proof}

\begin{example}
    Recall the hypertoric variety of Example \ref{E1}, where 
    \begin{align*}
        \Phi^+ = \{\varepsilon_1,\varepsilon_2,\varepsilon_3,\varepsilon_1+\varepsilon_2, \varepsilon_1+\varepsilon_3,\varepsilon_2+\varepsilon_3,\varepsilon_1+\varepsilon_2+\varepsilon_3\}\subset (\mathfrak{t}_\mathbb{Z}^k)^*,
    \end{align*}
    whose points are $1=(1,1,1)$ and $-1=(-1,-1,-1)$ and where 
    \begin{align*}
        \Phi^+_1 &= \Phi^+\\
        \Phi^+_{-1} &= \{\varepsilon_1+\varepsilon_2,\varepsilon_1+\varepsilon_3,\varepsilon_2+\varepsilon_3\}.
    \end{align*}
    Then the complete subsets of $\Phi_1^+$ and $\Phi_{-1}^+$ fit into the following poset diagram:
    \begin{center}
        % https://tikzcd.yichuanshen.de/#N4Igdg9gJgpgziAXAbVABwnAlgFyxMJZAJgBoBmAXVJADcBDAGwFcYkQAdDgBQAssA+gEYAegGoQAX1LpMufIRQAGUsWp0mrdl2BcGAJxhpsjAgKyk99Q8aymwAgFaWOBoybNYxVmx4eOuSSkZEAxsPAIiMjUaBhY2RE4OXVdrdztPb1TfDP8XN1t7cyyCvwEAa0Dg2XCFIgAWVXU4rUSdH3Siiw7Cs0cStN6HcvzBsq8esv7J3Iqq6Rr5SOVSIWbNBKSU0tmseZCwpcUSVfX47WSZroGcooCOIIXQuQjjxrXYjYvtsd2bzr6-yGcwe1WetWWJyUZ1aSRgAFs0DgAJ5wGA4MGHV5EABsFBhmy4fEEwAAtEJJOJMS86ig8TENOc2pdsgCHBNWcDnFdMjzhvtFti6adPkytnzihL7o8DjTIXjoaLYVwEUjUeipOoYFAAObwIigABm+gg8KQKhAOAgSDIIF4MHoUHYkDAbCextN5poVqQQho9sdzoIbpCHrNiAtPsQ5H9DqdiRdIaNJvDfst1sQjTtcaDrpANDg-ENGMQFNDKZt3ozAFZY4GE8H8yBC1hi773RXo1WkHjs-XwI2O57M93EAB2Ovxgd5ofh2vppATvtTxNgsM90dLgMrwfl4dpqPz7e5pMgdddheII85hszvfhmOXrPH2+n88HjPPm-TtgFosl8lZ0rS95xbNtEHJScTzXTsAA5RwATig18YOHJDLyEC0Xx-VDU0jDMhDTbDV0kShJCAA
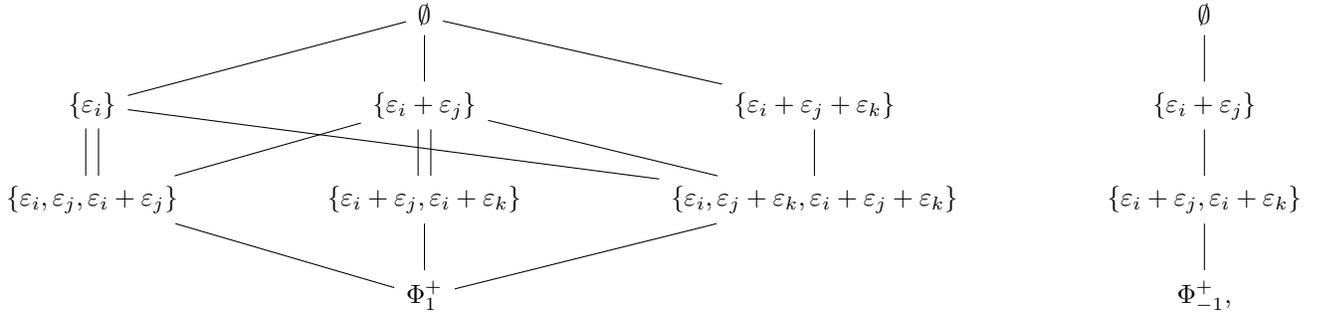
\begin{tikzcd}
                                                                                                                                                  &  & \emptyset                                                                                                                    &  &                                                                                                                                                          &  & \emptyset                                                                        \\
\{\varepsilon_i\} \arrow[rru, no head]                                                                                                            &  & \{\varepsilon_i+\varepsilon_j\} \arrow[u, no head]                                                                           &  & \{\varepsilon_i+\varepsilon_j+\varepsilon_k\} \arrow[llu, no head]                                                                                       &  & \{\varepsilon_i+\varepsilon_j\} \arrow[u, no head]                               \\
{\{\varepsilon_i,\varepsilon_j,\varepsilon_i+\varepsilon_j\}} \arrow[u, no head, shift right] \arrow[rru, no head] \arrow[u, no head, shift left] &  & {\{\varepsilon_i+\varepsilon_j,\varepsilon_i+\varepsilon_k\}} \arrow[u, no head, shift right] \arrow[u, no head, shift left] &  & {\{\varepsilon_i,\varepsilon_j+\varepsilon_k,\varepsilon_i+\varepsilon_j+\varepsilon_k\}} \arrow[u, no head] \arrow[llu, no head] \arrow[llllu, no head] &  & {\{\varepsilon_i+\varepsilon_j,\varepsilon_i+\varepsilon_k\}} \arrow[u, no head] \\
                                                                                                                                                  &  & \Phi_1^+ \arrow[u, no head] \arrow[llu, no head] \arrow[rru, no head]                                                        &  &                                                                                                                                                          &  & \Phi_{-1}^+,
                                                                                                                                                  \arrow[u, no head]                                                  
\end{tikzcd}
    \end{center}
    where $\{i,j,k\}= \{1,2,3\}$, and where $i,j,k$ are fixed for each of the sets in the above lattice.  This poset of complete sets corresponds bijectively to the following layers:
    \begin{center}
        % https://tikzcd.yichuanshen.de/#N4Igdg9gJgpgziAXAbVABwnAlgFyxMJZABgBoAmAXVJADcBDAGwFcYkQAdD4ARwH0sAXn4ArQQEYuAXxBTS6TLnyEU5UgGZqdJq3ZdgACnGlj4gJTTZ8kBmx4CRNVRoMWbRJ278sogASDfb34Aa39fSQ4ZOQU7ZSIAFgotV10PfW9hPhEQiUtom0V7FRITZJ13T14BXMirGKUHVVKXcr0vAVEaqOtbBuLE8TK3NqqfLJyI7vqix1JiIdSQABUAPWC6gtjG5AA2DQWK-QMAWmNT0lOLWvzemZQ9521htPaxsSC+YK6N27j75qei3SHSy3xuhT+uzmB3Yq3WUi0MCgAHN4ERQAAzABOEAAtkhjCAcBAkGQQAALGD0KDsSBgNj5bF4gk0YlINQUqk0jx0hnWJn4xCEtmIdQ0SnU2kEPmYnGCskixKcyU86UgGhwclYDE4JCnRly9mskmIACs4q5Uvp6pAmu1usQ+v5htFxqQe2V3PAaoNzMQSpFAHYLSrvdbfYLzUSTcHPVaZSABe63YhYxKvbyNknEAqTUq7TqWXHVeHnX6OSKowWHYT0-Gsy7c0go3WSwns2Lo82QxmfWXBZ3FT36xGkAAOFMATmHbYbfunXaFZNbYfbLvETaFtcts4RUiAA
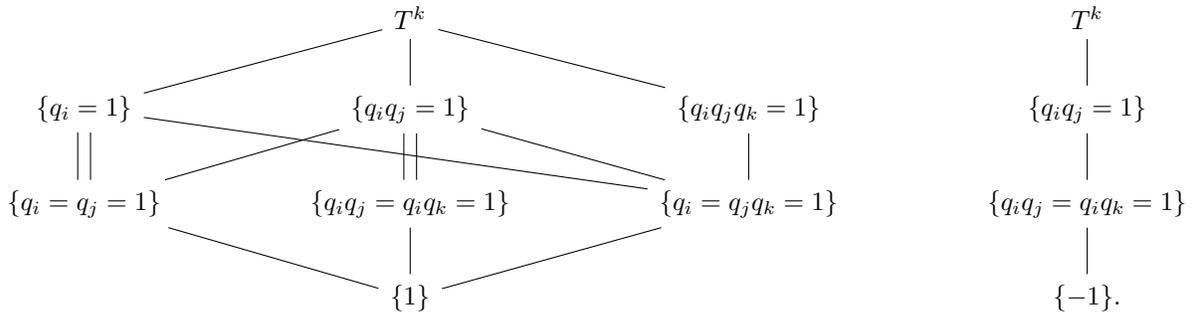
\begin{tikzcd}
                                                                                                  &  & T^k                                                                                    &  &                                                                                 &  & T^k                                    \\
\{q_i=1\} \arrow[rru, no head]                                                                    &  & \{q_iq_j=1\} \arrow[u, no head]                                                        &  & \{q_iq_jq_k=1\} \arrow[llu, no head]                                            &  & \{q_iq_j=1\} \arrow[u, no head]        \\
\{q_i=q_j=1\} \arrow[u, no head, shift left] \arrow[u, no head, shift right] \arrow[rru, no head] &  & \{q_iq_j = q_iq_k = 1\} \arrow[u, no head, shift left] \arrow[u, no head, shift right] &  & \{q_i=q_jq_k=1\} \arrow[u, no head] \arrow[llu, no head] \arrow[llllu, no head] &  & \{q_iq_j=q_iq_k=1\} \arrow[u, no head] \\
                                                                                                  &  & {\{1\}} \arrow[llu, no head] \arrow[u, no head] \arrow[rru, no head]             &  &                                                                                 &  & {\{-1\}.} \arrow[u, no head]   
\end{tikzcd}
    \end{center}

    The only complete sets in $\Phi_1^+$, and in $\Phi_{-1}^+$ that are reducible are $\{\varepsilon_i+\varepsilon_j\}\sqcup \{\varepsilon_i+\varepsilon_k\}$.  

    An example of nested set in $\Phi_1^+$ and one in $\Phi_{-1}^+$ are given by the following posets:
    \begin{center}
        % https://tikzcd.yichuanshen.de/#N4Igdg9gJgpgziAXAbVABwnAlgFyxMJZARgBoAGAXVJADcBDAGwFcYkQAdD4LhgJxhpsjAgH1iXAL4hJpdJlz5CKMsWp0mrdlx4d+g4WLK96AoVhFhRAJlImzhq8QDU9gxbHWpMuSAzY8AiIAZlI1GgYWNkRObjdzS3FXPVN3RK8OaVl5AKUiAFYw9UitGJ14xxtk-QSxYO9svwVA5RJSa2LNaNiABQALLHEAPWcfHMUglAAWds6o7Q5+weAAWmJJEbGm3MnkW3CNeZiAbi3-CdbyIoiuhYBbehw+gGMmYABlSWGAKgACAF5fmdmnkUAA2a6HUogAG-LgPJ6vRgfL6rdZDb4AOhk6hgUAA5vAiKAAGZ8CB3JBkEA4CBIK5Q6JgZiMRg0Rj0ABGMEYPRBkxAjBgJJwIBofRg9Cg7EgYDYjTJFKQMxpdMQ1JKTJZbMFXJ5fJ2ykFwtF4sl0pisvlvkVlMQhVVSFsjKQzNZ7L1vP5RqFIrFIAlUplBGtpPJdodtKQoRdiDdOo53K9hvYvtNAfNwblOMkQA
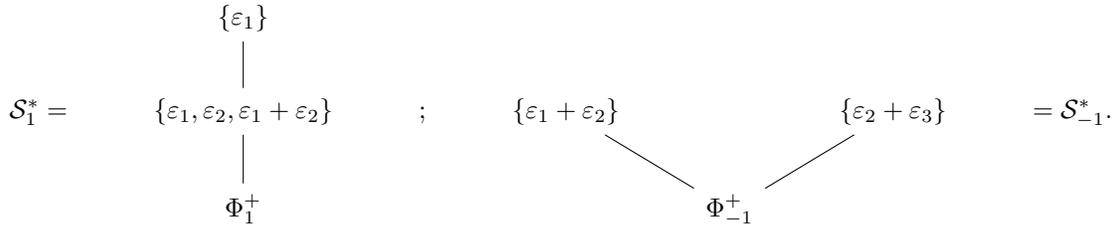
\begin{tikzcd}
                   & \{\varepsilon_1\}                                                                &   &                                 &                                                     &                                 &                        \\
\mathscr{S}_1^* =  & {\{\varepsilon_1,\varepsilon_2,\varepsilon_1+\varepsilon_2\}} \arrow[u, no head] & ; & \{\varepsilon_1+\varepsilon_2\} &                                                     & \{\varepsilon_2+\varepsilon_3\} &  = \mathscr{S}_{-1}^*. \\
                   & \Phi_1^+ \arrow[u, no head]                                                      &   &                                 & \Phi_{-1}^+ \arrow[lu, no head] \arrow[ru, no head] &                                 &                       
\end{tikzcd}
    \end{center}
    These correspond bijectively to the following nested sets of layers $\mathscr{S}_1$ and $\mathscr{S}_{-1}$ in $\cC_1(\Phi^+)$ and $\cC_{-1}(\Phi^-)$, respectively:
    \begin{center}
        % https://tikzcd.yichuanshen.de/#N4Igdg9gJgpgziAXAbVABwnAlgFyxMJZARgBoAmAXVJADcBDAGwFcYkQAdD4YrgXxB9S6TLnyEUZYtTpNW7LsACOAfWIBeVeXW8OAoSOx4CRMgAYZDFm0Sduqjbv3CQGI+KIAWCpbk27wAAUALRkoaShAJT8gi5uYiYoAMyk0jRW8raKDqpJOjEGrqLGEsgArKm+1gr2KuS5+XqxhgmlZpXpfjUAtvQ4ABYAxkzAAMp8agAE6pPNRe6JyABsHbLVtjNcvQPDjGMTwKF8AHRz8SVE5KsZ-gDcgjIwUADm8ESgAGYAThDdSO0gHAQJBkED9GD0KDsSBgNiFb6-EE0IFIK5giFQ2wwuEuBF-RApQHAxDedGQ6EEHGfH74wkoxAVMmY8CUh58IA
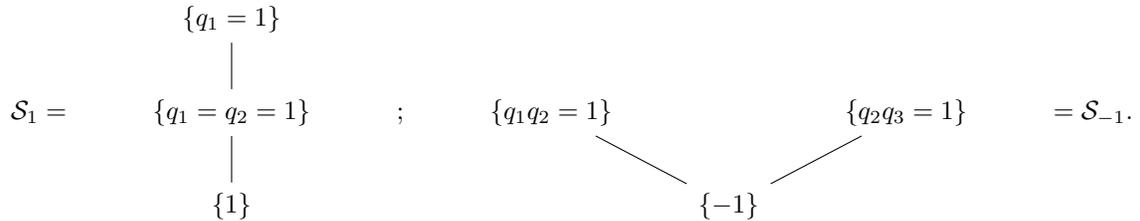
\begin{tikzcd}
                 & \{q_1=1\}                        &   &              &                                                          &              &                     \\
\mathscr{S}_1 =  & \{q_1=q_2=1\} \arrow[u, no head] & ; & \{q_1q_2=1\} &                                                          & \{q_2q_3=1\} & = \mathscr{S}_{-1}. \\
                 & \{1\} \arrow[u, no head]         &   &              & {\{-1\}} \arrow[lu, no head] \arrow[ru, no head] &              &                    
\end{tikzcd}
    \end{center}

    The other nested sets of layers are given by the following:
    \begin{align*}
        \mathscr{S}_{-1}^* = &\left\{\{\varepsilon_i+\varepsilon_j\},\{\varepsilon_j+\varepsilon_k\},\Phi_{-1}^+\right\},\\
        \mathscr{S}_{1}^* = &\left\{\{\varepsilon_i+\varepsilon_j\},\{\varepsilon_j+\varepsilon_k\},\Phi_1^+\right\},\;\left\{\{\varepsilon_i\},\{\varepsilon_i,\varepsilon_j,\varepsilon_i+\varepsilon_j\},\Phi_1^+\right\},\;
        \left\{\{\varepsilon_i+\varepsilon_j\},\{\varepsilon_i,\varepsilon_j,\varepsilon_i+\varepsilon_j\},\Phi_1^+\right\},\\
        &\left\{\{\varepsilon_i\},\{\varepsilon_i,\varepsilon_j+\varepsilon_k,\varepsilon_i+\varepsilon_j+\varepsilon_k\},\Phi_1^+\right\}, \left\{\{\varepsilon_j+\varepsilon_k\},\{\varepsilon_i,\varepsilon_j+\varepsilon_k,\varepsilon_i+\varepsilon_j+\varepsilon_k\},\Phi_1^+\right\},\\
        &\left\{\{\varepsilon_i+\varepsilon_j+\varepsilon_k\},\{\varepsilon_i,\varepsilon_j+\varepsilon_k,\varepsilon_i+\varepsilon_j+\varepsilon_k\},\Phi_1^+\right\},
    \end{align*}
    where \{i,j,k\} = \{1,2,3\}.  Notice that in the second example below Definition \ref{D3.2.3}, the sets $\Phi_1^+$ and $\Phi_{-1}^+$ are irreducible sets.   
\end{example}

We will now define an open cover of $Z_{\Phi^+}$ whose open sets are indexed by the maximal nested sets of layers centered at any $p\in\cC_0(\Phi^+)$.  

\begin{definition}
    Given a (maximal) nested set of layers $\mathscr{S}$, we say that $\mathcal{B}^\mathscr{S}$ is an {\it adapted basis to $\mathscr{S}$} if it is an integral basis for $\Lambda$ so that for every layer $C\in\mathscr{S}$, the intersection $\mathcal{B}^\mathscr{S}\cap \overline{\Lambda_C}$ is an integral basis for $\overline{\Lambda_C}$.
\end{definition}

\begin{remark}
    For any maximal nested set of layers $\mathscr{S}$, such an adapted basis to $\mathscr{S}$ always exists.  (See Lemma 3.9 of~\cite{M11}).    
\end{remark}

\begin{example}
    In the above example, an example of an adapted basis to $\mathscr{S}_1$ is $\{\varepsilon_1,\varepsilon_1+\varepsilon_2,\varepsilon_3\}$ and an example of an adapted basis to $\mathscr{S}_{-1}$ is $\{\varepsilon_1+\varepsilon_2,\varepsilon_2,\varepsilon_2+\varepsilon_3\}$.  Notice that in the latter case, not all the vectors are in $\Phi_{-1}^+$.  Indeed, since $\Phi_{-1}^+$ is not complete, no elements in $\Phi_{-1}^+$ could form a basis for $\Lambda\simeq \mathbb{Z}^3$.  
\end{example}

\begin{lemma}\label{L3.2.4}(Lemma 3.8(ii) of~\cite{M11})
    Given $C\in\cI(\Phi^+)$, which is not minimal, there is a unique $C'\in\mathscr{S}$ which is maximal among all elements $D\in\mathscr{S}$ such that $D\subset C$.  Let $\overline{C}$ denote this element, and call it the {\it core} of $C$.  
    
    If $C\in\mathscr{S}$, then there is a unique element which is the maximum among all elements $D\in \mathscr{S}$ for which $D\subsetneq C$.  We denote this unique element $s(C)$, and call it the {\it successor} of $C$ in $\mathscr{S}$.   
\end{lemma}

\begin{proof}
    Suppose $C = C(A)\in\cI$ is not minimal and $C' = C(A'),C'' = C(A'')\in\mathscr{S}$ for $A',A''\in\mathscr{S}^*$ are distinct elements with $C',C''\subset C$.  Then we know that $A',A''\supset A$, are distinct subsets which satisfy $A'\supset A$ and $A''\supset A$.  Thus, $A'\cap A''\supset A$ and $A'\cup A''$ is not a decomposition.  By Definition \ref{D3.2.6}, we know that $A$ and $A''$ must be comparable.     
\end{proof}

\begin{example}
    For the above example, in $\mathscr{S}_1$, the successor to $\{q_1=1\}$ is $\{q_1=q_2=1\}$, and the successor to $\{q_1=q_2=1\}$ is $\{1\}$.  In $\mathscr{S}_{-1}^*$, the successors to $\{q_1q_2=1\}$ and $\{q_2q_3=1\}$ are both $\{-1\}$.  
\end{example}

\begin{definition}
    Given a maximal nested set of layers $\mathscr{S}$, define a map 
    \begin{align*} 
        p_\mathscr{S}\colon \Lambda &\longrightarrow \mathscr{S}\\
        \alpha&\longmapsto C_{\alpha},
    \end{align*}
    where $C_{\alpha}$ is the largest layer in $\mathscr{S}$ such that $H_\alpha\supset C_\alpha$.  
\end{definition}

\begin{lemma} (Lemma 3.10 of~\cite{M11})
    The map $p_{\mathscr{S}}$, when restricted to the adapted basis $\mathcal{B}^\mathscr{S}$, is a bijection.  
\end{lemma}

\begin{definition}
    Denote the inverse bijection
\begin{align*}
    \mathscr{S}&\longrightarrow\mathcal{B}\\
    C&\longrightarrow \alpha_C,
\end{align*}
    so that for $\mathscr{S} = \{C_1,\ldots,C_k\}$, we have $\mathcal{B}^\mathscr{S} = \{\alpha_{C_1},\ldots,\alpha_{C_k}\}$.  
\end{definition}

\begin{example}
    In the above example for $\mathscr{S}_1$, we can define the map 
    \begin{align*}
        p_{\mathscr{S}_1}\colon \mathbb{Z}^3&\longrightarrow\mathscr{S}_1\\
        \varepsilon_1&\longmapsto \{q_1=1\}\\
        \varepsilon_2&\longmapsto \{q_1=q_2=1\}\\
        \varepsilon_3&\longmapsto \{1\}
    \end{align*}
    and, for example, we have $p_{\mathscr{S}_1}(\varepsilon_1+\varepsilon_2) = \{q_1=q_2=1\}$.  

    In the above example, for $\mathscr{S}_{-1}$, we can define the map 
    \begin{align*}
        p_{\mathscr{S}_{-1}}\colon\mathbb{Z}^3&\longrightarrow\mathscr{S}_{-1}\\
        \varepsilon_1+\varepsilon_2&\longmapsto \{q_1q_2=1\}\\
        \varepsilon_2&\longmapsto \{-1\}\\
        \varepsilon_2+\varepsilon_3&\longmapsto \{q_2q_3=1\}
    \end{align*}
    and, for example, we have $p_{\mathscr{S}_{-1}}(\varepsilon_1) = \{-1\}$.

    By inspection, we see that $p_{\mathscr{S}_1}$ and $p_{\mathscr{S}_{-1}}$ restrict to bijections on adapted bases.  
\end{example}

\begin{definition}
    Let $p\in\cC_0(\Phi^+)$ and fix a maximal nested set of layers $\mathscr{S} = \{C_1,\ldots,C_k\}$ in $\cC_p(\Phi^+)$ with $\mathcal{B}^\mathscr{S} = \{\alpha_{C_1},\ldots,\alpha_{C_k}\}$ an adapted basis of $\Lambda\simeq \mathbb{Z}^k$ and let $a_i = \alpha_{C_i}(p)$.    

    Consider the following map:
    \begin{align*}
        \overline{f^\mathscr{S}}\colon \mathbb{C}^k&\longrightarrow \mathbb{C}^k\\
        \left(z_{C_i}\mid i=1,\ldots,k\right)&\longmapsto \left(a_i+\prod_{C_j\subset C_i}z_{C_j}\;\biggr\rvert\; i=1,\ldots,k\right),
    \end{align*}
    let 
    \begin{align*}
        \mathcal{U}_\mathscr{S} = \overline{f^\mathscr{S}}^{-1}\left(T^k\setminus \bigcup_{p\not\in C}C\right) =  \overline{f^\mathscr{S}}^{-1}\left(\;(\mathbb{C}^k\setminus\{0\})\setminus\bigcup_{p\not\in C}C\;\right),
    \end{align*}
    and define 
    \begin{align}
        f^\mathscr{S} = \overline{f^\mathscr{S}}\rvert_{\mathcal{U}_\mathscr{S}}\colon \mathcal{U}_\mathscr{S}&\longrightarrow T^k.
    \end{align}
    Using the adapted basis, $\{\alpha_{C_k}\}$, we have the isomorphism
    \begin{align*}
        T^k&\xrightarrow[]{\;\;\sim\;\;} T^k\\
        t&\longmapsto (\alpha_{C_1}(t),\ldots,\alpha_{C_k}(t)),  
    \end{align*}
    so for a given $(z_{C_1},\ldots,z_{C_k})\in\mathcal{U}_\mathscr{S}$, there is a unique $t\in T^k$ for which the following equality in $T^k$ holds:
    \begin{align*}
        \alpha_{C_i}(t) - a_i = \prod_{\substack{D\subset C_i\\D\in \mathscr{S}}}z_D.
    \end{align*} 
\end{definition}

\begin{lemma}\label{L3.2.5}            
    Fix an $\alpha\in\Phi_p^+$ and let $C = p_\mathscr{S}(\alpha)\in\mathscr{S}$ be the unique maximal layer in $\mathscr{S}$ on which $\alpha\rvert_C = a$ is a constant.  Then there exists a regular function $p((z_i)_{i=1}^k)$ on $\mathbb{C}^k$, with $p(0) \neq 0$, so that $\alpha(t) - a$ can be written in terms of the $(z_{C_i})_{i=1}^k$ as follows: 
    \begin{align*}
        \alpha(t) - a = p_\alpha((z_{C_i})_{i=1}^k)\prod_{\substack{D\subset C\\D\in\mathscr{S}}}z_D,
    \end{align*}
    as an equality in $T^k$.  
\end{lemma}

\begin{example}
    In the above example for $\mathscr{S}_1$, take $C_i = p_{\mathscr{S}_1}(\varepsilon_i)$ for $i=1,2,3$.  Then we have
    \begin{align*}
        \overline{f^{\mathscr{S}_1}}(z_{C_1},z_{C_2},z_{C_3}) = (1+z_{C_1}z_{C_2}z_{C_3}, 1+z_{C_2}z_{C_3}, 1+z_{C_3}) = (\varepsilon_1(t),\varepsilon_2(t),\varepsilon_3(t))\in T^k.
    \end{align*}
    Our domain is given by 
    \begin{align*}
        \mathcal{U}_{\mathscr{S}_1} = \overline{f^{\mathscr{S}_1}}^{-1}(\mathbb{C}^3\setminus\{0\})
    \end{align*}
    and $f^{\mathscr{S}_1} = \overline{f^{\mathscr{S}_1}}\rvert_{\mathcal{U}_{\mathscr{S}_1}}$.  
    We recall that $p_{\mathscr{S}_1}(\varepsilon_1+\varepsilon_2) = C_2$, and that $(\varepsilon_1+\varepsilon_2)\rvert_{C_2} = 1$, so using multiplicative notation, we have $(\varepsilon_1+\varepsilon_2)(t) = \varepsilon_1(t)\varepsilon_2(t)$, and we obtain the following:
    \begin{align*}
        \varepsilon_1(t)\varepsilon_2(t) - 1 
        &= (1+z_{C_1}z_{C_2}z_{C_3})(1+z_{C_2}z_{C_3}) - 1\\
        &= (1 + z_{C_1} + z_{C_1}z_{C_2}z_{C_3})z_{C_2}z_{C_3},
    \end{align*}
    so that in the statement of the above lemma,
    \begin{align*}
        p_{\varepsilon_1+\varepsilon_2}(z_{C_1},z_{C_2},z_{C_3}) = 1 + z_{C_1} + z_{C_1}z_{C_2}z_{C_3}.
    \end{align*}

    Now consider the above example for $\mathscr{S}_{-1}$.  Take $C_1 = p_{\mathscr{S}_{-1}}(\varepsilon_1+\varepsilon_2)$, $C_2 = p_{\mathscr{S}_{-1}}(\varepsilon_2)$, and $C_3 = p_{\mathscr{S}_{-1}}(\varepsilon_2+\varepsilon_3)$ for $i=1,2,3$.  Then we have
    \begin{align*}
        \overline{f^{\mathscr{S}_{-1}}}(z_{C_1},z_{C_2},z_{C_3}) = (1+z_{C_1}z_{C_2}, -1+z_{C_2}, 1+z_{C_2}z_{C_3}) = (\;(\varepsilon_1+\varepsilon_2)(t),\varepsilon_2(t),(\varepsilon_2+\varepsilon_3)(t)\;)\in T^k.
    \end{align*}
    Here, our domain is given by 
    \begin{align*}
        \mathcal{U}_{\mathscr{S}_{-1}} = \overline{f^{\mathscr{S}_{-1}}}^{-1}\left(\left(\mathbb{C}^3\setminus\{0\}\right)\setminus \left(\bigcup_{i=1}^3\{q_i=1\}\cup\{q_1q_2q_3=1\}\right)\right)
    \end{align*}
    and $f^{\mathscr{S}_{-1}} = \overline{f^{\mathscr{S}_{-1}}}\rvert_{\mathcal{U}_{\mathscr{S}_{-1}}}$.  
    We recall that $p_{\mathscr{S}_{-1}}(\varepsilon_1) = \{-1\}$, and that $\varepsilon_1\rvert_{\{-1\}} = -1$.  Thus, in using multiplicative notation, we have
    \begin{align*}
        \varepsilon_1(t)+1 
        &= (\varepsilon_1(t)\varepsilon_2(t))\varepsilon_2(t)^{-1} + 1\\
        &= \varepsilon_2(t)^{-1}(\varepsilon_1(t)\varepsilon_2(t)+\varepsilon_2(t))\\
        &= \varepsilon_2(t)^{-1}(1+z_{C_1}z_{C_2} + -1 + z_{C_2})\\
        &= \varepsilon_2(t)^{-1}(1+z_{C_1})z_{C_2}.
    \end{align*}
    Since the function
    \begin{align*}
        p_{\varepsilon_1}(z_{C_1},z_{C_2},z_{C_3}) 
        &= \varepsilon_2(t)^{-1}(1+z_{C_1})\\
        &= \frac{1+z_{C_1}}{-1+z_{C_2}},
    \end{align*}
    is a regular function, whenever $\varepsilon_2(t)\neq 0$, it satisfies the statement of the above lemma.  
\end{example}

\begin{proof}
    Suppose $\alpha = \alpha_C\in\Phi^+$ is arbitrary.  Then for all $t\in T^k$, we can write $\alpha$ using the basis $\mathcal{B}^\mathscr{S}\cap \Phi^+_C$ using multiplicative notation as follows
    \begin{align*}
        \alpha_C(t) = \alpha_C^{m_C}(t)\cdot\prod_{\substack{D\supset C\\D\in\mathscr{S}}}\alpha_D^{m_D}(t),
    \end{align*}
    for some integers $m_C,m_D\in\mathbb{Z}$.  In setting $\alpha_C\rvert_C = a_C$ and $\alpha_D\rvert_D = a_D$ to be the constants on each layer, we have the following calculation:  
    \begin{align*}
        \alpha_C(t) - a_C 
        &= \alpha_C^{m_C}(t)\cdot\prod_{\substack{D\supsetneq C\\D\in\mathscr{S}}}\alpha_D^{m_D}(t) - a_C\\ 
        &= \left(\alpha_C^{m_C}(t) - a_C^{m_C}\right)\prod_{\substack{D\supsetneq C\\D\in\mathscr{S}}}\alpha_D^{m_D}(t) + a_C^{m_C}\prod_{\substack{D\supsetneq C\\D\in\mathscr{S}}}\alpha_D^{m_D}(t)-a_C.
    \end{align*}
    Factoring $\alpha_C^{m_C}(t) - a_C^{m_C}$ into roots of unity and repeating this process for each $D$ for which $D\supset C$, we find that 
    \begin{align*}
        \alpha_C(t) - a_C = \beta_C((z_{C_i})_{i=1}^k)\;(\alpha_C(t) - a_C) + \sum_{\substack{D\supsetneq C\\D\in\mathscr{S}}}\beta_D((z_{C_i})_{i=1}^k)\;(\alpha_D(t)-a_D),
    \end{align*}
    for some functions $\beta_C((z_{C_i})_{i=1}^k),\beta_D((z_{C_i})_{i=1}^k)$, which do not vanish on their respective layers $C,D$ when $z_{C_i} = 0$.  Repeating this process for each of the $\alpha_{C_i}$, we find that 
    \begin{align*}
        \alpha_C(t) - 1 &= \beta_C((z_{C_i})_{i=1}^k)\prod_{\substack{E\subset C\\E\in\mathscr{S}}}z_E + \sum_{\substack{D\supsetneq C\\D\in\mathscr{S}}} \beta_D((z_{C_i})_{i=1}^k) \prod_{\substack{E\subset D\\E\in\mathscr{S}}}z_E\\
        &= \left(\beta_C((z_{C_i})_{i=1}^k) + \sum_{\substack{D\supsetneq C\\D\in\mathscr{S}}}\beta_D((z_{C_i})_{i=1}^k) \prod_{\substack{C\subsetneq E\subset D\\E\in\mathscr{S}}}z_E\right)\prod_{\substack{E\subset C\\E\in\mathscr{S}}}z_E\\
        &= p_{\alpha_C}((z_{C_i})_{i=1}^k)\prod_{\substack{E\subset C\\E\in\mathscr{S}}}z_E.
    \end{align*}
    Since $p_{\alpha_C}(0,\ldots,0)\neq 0$, by definition of the $\beta_i$, we are done.    
\end{proof}

\begin{definition}
    Given a maximal nested set $\mathscr{S} = \{C_1,\ldots,C_k\}$, define the following open sets:  
    \begin{align*}
        \mathcal{V}_{\mathscr{S}} &= \left\{(z_{C_i})_{i=1}^k\in \mathcal{U}_{\mathscr{S}}\;\rvert\;p_\alpha((z_{C_i})_{i=1}^k)\neq 0\; \forall \alpha\in \Phi^+\right\}\\
        \mathcal{V}_{\mathscr{S}}^0 &= \left\{(z_{C_i})_{i=1}^k\in \mathcal{V}_\mathscr{S}\;\rvert\;z_{C_i}\neq 0\;\forall i=1,\ldots,k\right\}.
    \end{align*}      
    so $\mathcal{V}_\mathscr{S}^0\subset \mathcal{V}_\mathscr{S}\subset \mathcal{U}_\mathscr{S}\subset \mathbb{C}^k\simeq \mathbb{C}^\mathscr{S}$.
\end{definition}

\begin{definition}
    Let $C\in\cI(\Phi^+)$ be any irreducible layer and let $\mathscr{S} = \{C_1,\ldots,C_k\}$ be a maximal nested set of layers in $\cC_p(\Phi)$ as above.  Let $\Phi^+_C = \{\alpha_1,\ldots,\alpha_r\}$, let $a_i = \alpha_i(p)$, and let $C_{\alpha_i} = p_\mathscr{S}(\alpha_i)$ and take
    \begin{align*}
        \varphi_C\colon T^k\setminus C&\longrightarrow \mathbb{P}_C\\
        t&\longmapsto [\alpha_1(t)-a_1\colon\ldots\colon\alpha_r(t)-a_r].
    \end{align*}
    Then for each $i=1,\ldots,r$, we define the map 
    \begin{align*}
        \psi_{C}^{\mathscr{S},0}\colon \mathcal{V}_\mathscr{S}^0&\longrightarrow \mathbb{P}_C\\
        (z_{C_i})_{i=1}^k&\longmapsto [\;p_{\alpha_i}((z_{C_i})_{i=1}^k)\;\Pi_{C_j\subset C_{\alpha_i}}z_{C_j}\mid i=1,\ldots,r\;], 
    \end{align*}
    so that $\alpha_i(t)-a_i = p_{\alpha_i}((z_{C_i})_{i=1}^k)\prod_{C_j\subset C_i} z_{C_j}$, according to Lemma \ref{L3.2.5}.
\end{definition}

\begin{lemma}(Lemma 4.1 of~\cite{M11})
    The map $\psi_C^{\mathscr{S},0}$ factors uniquely through the inclusion $\mathcal{V}_{\mathscr{S}}^0\hookrightarrow\mathcal{V}_{\mathscr{S}}$.  
\end{lemma}

\begin{example}
    Consider the above example, with maximal nested set $\mathscr{S}_1$ and with the layer $C = \{q_1q_2 = 1,q_3 = 1\}$ corresponding to the complete set $\Phi_C = \{\varepsilon_1+\varepsilon_2,\varepsilon_3,\varepsilon_1+\varepsilon_2+\varepsilon_3\}$.  Choose the adapted basis $\{\varepsilon_1,\varepsilon_2,\varepsilon_3\}$ and pick generators $\{\varepsilon_1+\varepsilon_2,\varepsilon_3\}$ of $\Phi_C$.  Then we have 
    \begin{align*}
        \psi_C^{\mathscr{S}_1,0}\colon \mathcal{V}_{\mathscr{S}_1}^0&\longrightarrow \mathbb{P}_C\simeq \mathbb{P}^1\\
        (z_{C_1},z_{C_2},z_{C_3})&\longmapsto [\varepsilon_1(t)\varepsilon_2(t)-1\;\colon\;\varepsilon_3(t)-1] = [(1 + z_{C_1} + z_{C_1}z_{C_2}z_{C_3})z_{C_2}z_{C_3}\;\colon\;z_{C_3}],
    \end{align*}
    and the extension is given by 
    \begin{align*}
        \psi_C^{\mathscr{S}_1}\colon \mathcal{V}_{\mathscr{S}_1}&\longrightarrow \mathbb{P}_C\\
        (z_{C_1},z_{C_2},z_{C_3})&\longmapsto [(1 + z_{C_1} + z_{C_1}z_{C_2}z_{C_3})z_{C_2}\;\colon\;1],
    \end{align*}
    
    Now consider the above example, with $\mathscr{S}_{-1}$ and with the layer $C = \{(1,1,1)\}$ corresponding to the complete set $\Phi_1^+$.  Choose the adapted basis $\{\varepsilon_1+\varepsilon_2,\varepsilon_2,\varepsilon_2+\varepsilon_3\}$ and choose generators $\{\varepsilon_1+\varepsilon_2,\varepsilon_2,\varepsilon_2+\varepsilon_3\}$ of $\Phi^+$.  Then we have 
    \begin{align*}
        \psi_C^{\mathscr{S}_{-1},0}\colon \mathcal{V}_{\mathscr{S}_{-1}}^0&\longrightarrow \mathbb{P}_C \simeq \mathbb{P}^2\\
        (z_{C_1},z_{C_2},z_{C_3})&\longmapsto [\varepsilon_1(t)\varepsilon_2(t)-1\;\colon\; \varepsilon_2(t)-1\;\colon\;\varepsilon_2(t)\varepsilon_3(t)-1] = [z_{C_1}z_{C_2}\;\colon\;-2+z_{C_2} \;\colon\; z_{C_2}z_{C_3}],
    \end{align*}
    and it follows that this map extends to a map 
    \begin{align*}
        \psi_C^{\mathscr{S}_{-1}}\colon \mathcal{V}_{\mathscr{S}_{-1}}\longrightarrow \mathbb{P}_C.
    \end{align*}    
\end{example}

\begin{definition}
    The basis $\mathcal{B}^\mathscr{S} = \{\alpha_{C_1},\ldots,\alpha_{C_s}\}$ forms a coordinate system for $T^k$ and we can then define the map 
    \begin{align*} 
        \psi_T^\mathscr{S} \colon \mathcal{V}_{\mathscr{S}}&\longrightarrow T^k\\
        (z_{C_i})_{i=1}^k&\longrightarrow (\alpha_{C_i}(t)\mid i=1,\ldots,k),
    \end{align*}
    where $\alpha_{C_i}(t) = a_i + \prod_{C_j\subset C_i}z_{C_j}$ and where $a_i = \alpha_{C_i}(p)$ for all $i=1,\ldots,s$.  Thus, for each $C\in\cI(\Phi^+)$, define 
    \begin{align*}
        \psi^\mathscr{S} = \psi_T^\mathscr{S}\times \prod_{C\in\cI(\Phi^+)} \psi_C^{\mathscr{S}}\colon \mathcal{V}_\mathscr{S}\longrightarrow T^k\times \prod_{C\in\cI(\Phi^+)}\mathbb{P}_C.
    \end{align*}
\end{definition}

\begin{lemma}(Lemma 4.2 of~\cite{M11}
    The map $\psi^{\mathscr{S}}$ is an embedding to a smooth open set. 
\end{lemma}

\begin{example}
    Consider the maximal nested set $\mathcal{V}_{\mathscr{S}_1}$ of the above Example.  For the layer $C_3 = \{1\}$, we note that $C_3$ is minimal and thus, we can set
    \begin{align*}
        z_{C_3} = \varepsilon_3(t) - 1.
    \end{align*}
    For the layer $C_2 = \{q_1=q_2=1\}$, we have $s(C_2) = C_3$ and we can take generators $\{\varepsilon_1,\varepsilon_2,\varepsilon_3\}$ of $\Phi_{C_3}$ and so that the corresponding map is 
    \begin{align*}
        \psi^{\mathscr{S}_1}_{C_3}\colon \mathcal{V}_{\mathscr{S}_1}&\longrightarrow \mathbb{P}_{C_3}\\
        (z_{C_1},z_{C_2},z_{C_3})&\longmapsto [\varepsilon_1(t)-1\;\colon\;\varepsilon_2(t)-1\;\colon\;\varepsilon_3(t)-1] = [z_{C_1}z_{C_2}\;\colon\; z_{C_2}\;\colon\; 1],
    \end{align*}
    where $t\in T^k\setminus C_3$, and we find that 
    \begin{align*}
        z_{C_2} = \frac{\varepsilon_2(t)-1}{\varepsilon_3(t)-1}.
    \end{align*}
    For the layer $C_1$, we have $s(C_1) = C_2$ and we can take generators $\{\varepsilon_1,\varepsilon_2\}$ of $\Phi_{C_2}$.  We then have 
    \begin{align*}
        \psi^{\mathscr{S}_1}_{C_2}\colon \mathcal{V}_{\mathscr{S}_1}&\longrightarrow \mathbb{P}_{C_2}\\
        (z_{C_1},z_{C_2},z_{C_3})&\longmapsto [\varepsilon_1(t)-1\;\colon\;\varepsilon_2(t)-1] = [z_{C_1}\;\colon\;1],
    \end{align*}
    where $t\in T^k\setminus C_2$, and we find that 
    \begin{align*}
        z_{C_1} = \frac{\varepsilon_1(t)-1}{\varepsilon_2(t)-1}.
    \end{align*}
    Thus, we can recover the coordinates $(z_{C_1},z_{C_2},z_{C_3})$ from the image $\psi^{\mathscr{S}_1}(\mathcal{V}_{\mathscr{S}_1})$.  
    
    Now consider the maximal nested set $\mathscr{S}_{-1}$.  For the layers $C_1 = \{q_1q_2 = 1\}$ and $C_3 = \{q_2q_3 = 1\}$, we have $s(C_1) = s(C_3) = C_2$, and we can choose the adapted basis $\{\varepsilon_1+\varepsilon_2,\varepsilon_2,\varepsilon_2+\varepsilon_3\}$ as generators of $\Phi_{C_2}$.  Thus, we have 
    \begin{align*}
        \psi^{\mathscr{S}_{-1}}_{C_2}\colon\mathcal{V}_{\mathscr{S}_1}&\longrightarrow \mathbb{P}_{C_2}\\
        (z_{C_1},z_{C_2},z_{C_3})&\longmapsto [\varepsilon_1(t)\varepsilon_2(t)-1\;\colon\; \varepsilon_2(t)+1\;\colon\; \varepsilon_2(t)\varepsilon_3(t)-1] = [z_{C_1}\;\colon\; 1\;\colon\; z_{C_3}],
    \end{align*}
    where $t\in T^k\setminus C_2$.  Thus, we can read off the coordinates 
    \begin{align*}
        z_{C_1} &= \frac{\varepsilon_1(t)\varepsilon_2(t)-1}{\varepsilon_2(t)+1},\\
        z_{C_3} &= \frac{\varepsilon_2(t)\varepsilon_3(t)-1}{\varepsilon_2(t)+1}.
    \end{align*}
    Moreover, since $C_2$ is minimal, we also have 
    \begin{align*}
        z_{C_2} = \varepsilon_2(t) + 1.
    \end{align*}
    Thus, we can recover the coordinates $(z_{C_1},z_{C_2},z_{C_3})$ from the image $\psi^{\mathscr{S}_{-1}}(\mathcal{V}_{\mathscr{S}_{-1}})$.
\end{example}

From now on, we make the identification $\mathcal{V}_\mathscr{S}:= \psi^\mathscr{S}(\mathcal{V}_\mathscr{S})$. 

\begin{theorem}\label{T3.2.1} (Theorem 4.5 of~\cite{M11})
    We have $\psi^\mathscr{S}(\mathcal{V}_\mathscr{S}^0) \subset \psi^\mathscr{S}(T^{\mathrm{reg}})\subset T^k\times\prod_{C\in\cI(\Phi^+)}\mathbb{P}_C$.  Moreover, taking the limits as $z_C\rightarrow 0$, we have $f^\mathscr{S}(\mathcal{V}_\mathscr{S})\subset Z_{\Phi^+}$. 
    In defining
    \begin{align*}
        Y_{\Phi^+} := \bigcup_{\mathscr{S}\in\mathcal{M}}\mathcal{V}_\mathscr{S},
    \end{align*}
    we have
    \begin{align*}
        Z_{\Phi^+} = Y_{\Phi^+}.
    \end{align*}
\end{theorem}
\subsection{Extending to a Compactification of the Torus $T^\mathrm{reg}$}\label{S3.3}

In this section, we prove the following theorem.
\begin{theorem}\label{T3.3.2}
    The map 
    \begin{align*}
        Q'\colon T^{\mathrm{reg}}&\longrightarrow \mathrm{Gr}(k,\mathfrak{u}_{\Phi^+}^1)\\
        q&\longmapsto \mathrm{Span}_\mathbb{C}\{H_i^\mathrm{trig}(q)\mid i=1,\ldots,k\},
    \end{align*}
    where 
    \begin{align*}
        H_i^\mathrm{trig}(q) = u_i + \sum_{\alpha\in\Phi^+}\frac{q^\alpha}{1-q^\alpha}\alpha_i t_\alpha,
    \end{align*}
    admits an extension to the map 
    \begin{align*}
        Q'_{\Phi^+}\colon Y_{\Phi^+}\longrightarrow \mathrm{Gr}(k,\mathfrak{u}_{\Phi^+}^1).       
    \end{align*}
\end{theorem}

\begin{proof}
    Write the open cover $Y_{\Phi^+} = \bigcup_{\mathscr{S}\in\mathcal{M}} U_\mathscr{S}$, and fix a maximal nested set $\mathscr{S} = \{C_1,\ldots,C_k\}$, with an adapted basis $\mathcal{B}^\mathscr{S} = \{\beta_1,\ldots,\beta_k\} := \{\beta_{C_1},\ldots,\beta_{C_k}\}$ and with coordinates $(z_1,\ldots,z_k)= 
 (z_{C_1},\ldots,z_{C_k})\in U_\mathscr{S}$.      
    Let $B = \begin{pmatrix}\beta_1\cdots\beta_k\end{pmatrix}^\mathsf{T}$ be the matrix corresponding to the basis $\mathcal{B}^\mathscr{S}$ and take 
    \begin{align*}
        (H_B)_i^\mathrm{trig}(q) 
        &= [B^{-1}]\cdot H_i^\mathrm{trig}(q)\\
        &= [B^{-1}]u_i + \sum_{\alpha\in{\Phi^+}}\frac{1}{q^\alpha-1}\alpha([B^{-1}]u_i)t_\alpha\\
        &= [B^{-1}]u_i + \sum_{\alpha\in\Phi}\frac{1}{q^\alpha-1}\alpha_i t_\alpha,
    \end{align*}
    where we have written $\alpha = (\alpha_i)\cdot [B] = \sum_{j=1}^k \alpha_j\beta_{C_j}$.  Write
    \begin{align*}
        Q'(q) = \mathrm{Span}\left\{(H_B)_i^\mathrm{trig}(q)\mid i=1,\ldots,k\right\}
    \end{align*}
    for all $q\in T^{\mathrm{reg}}$, and take $q_i:= q^{\beta_{C_i}}$, and $q_i - 1 = \prod_{C_j\subset C_i}z_j$, and let $\alpha\in{\Phi^+}$ be a given element such that $\alpha_{i_0}\neq 0$ for some fixed $i_0\in \{1,\ldots,k\}$.  Then we have
    \begin{align*}
        q^\alpha - 1 
        &= q_1^{\alpha_1}\cdot\ldots\cdot q_k^{\alpha_k}-1\\
        &=\left(1+\prod_{C_j\subset C_1}z_j\right)^{\alpha_1}\ldots\left(1+\prod_{C_j\subset C_k}z_j\right)^{\alpha_k}-1\\
        &= \left(\alpha_1\left(\prod_{C_j\subset C_1}z_j\right)+\ldots + \alpha_k\left(\prod_{C_j\subset C_k}z_j\right)\right)+\ldots+\left(\prod_{l=1}^k \left(\prod_{C_j\subset C_l}z_j\right)^{\alpha_l}\right).
    \end{align*}
    Let $i_{\mathrm{min}}\in \{i\in \{1,\ldots,k\}\mid \alpha_i\neq 0\}$ be the index such that $C_{i_{\mathrm{min}}}$ is minimal among the layers in the set $\{C_i\mid \alpha_i\neq 0\}\subset \mathscr{S}$.  Using the identity 
    \begin{align*}
        \prod_{C_j\subset C_{l}}z_j = \left(\prod_{C_{i_{\mathrm{min}}}\subset C_j\subset C_l}z_j\right)\cdot\left(\prod_{C_j\subset C_{i_{\mathrm{min}}}}z_j\right),
    \end{align*}
    along with $C_{i_\mathrm{min}} = C_\alpha$, we find that 
    \begin{align*}
        q^\alpha - 1
        &= p_\alpha((z_{C_i})_{i=1}^k)\cdot \prod_{C_j\subset C_{\alpha}}z_j,
    \end{align*}
    according to Lemma \ref{L3.2.5}, where $p_\alpha((z_{C_i})_{i=1}^k)\neq 0$ because $\alpha_{i_0}\neq 0$.  Thus, $C_\alpha\subset \bigcap_{i\in[k], \alpha_i\neq 0}C_i\subset C_{i_0}$ and by Lemma \ref{L3.2.4}, we know that the set $\{C_i\in\mathscr{S}\mid C_i\subset C_{i_0}\}$ is a linearly ordered set, so we have the following factorization:
    \begin{align*}
        \prod_{C_j\subset C_{i_0}}z_j &= \left(\prod_{C_\alpha\subset C_j\subset C_{i_0}}z_j\right)\left(\prod_{C_j\subset C_\alpha}z_j\right)\\
        \frac{1}{\left(\prod_{C_j\subset C_\alpha}z_j\right)}&= \frac{\left(\prod_{C_\alpha\subset C_j\subset C_{i_0}}z_j\right)}{\left(\prod_{C_j\subset C_{i_0}}z_j\right)},
    \end{align*}
    so 
    \begin{align*}
        \frac{1}{q^\alpha-1} 
        &= \frac{1}{p_\alpha((z_{C_i})_{i=1}^k)\cdot \prod_{C_j\subset C_{\alpha}}z_j}\\
        &= \frac{1}{p_\alpha((z_{C_i})_{i=1}^k)}\cdot\frac{\left(\prod_{C_\alpha\subset C_j\subset C_{i_0}}z_j\right)}{\left(\prod_{C_j\subset C_{i_0}}z_j\right)},
    \end{align*}
    which is well-defined on $\mathcal{V}_\mathscr{S}^0\subset \mathcal{V}_\mathscr{S}$. 
    Summing over the set $\{\alpha\in{\Phi^+}\mid \alpha_{i_0}\neq 0\}$, we get the following:
    \begin{align*}
        (H_B)_{i_0}^\mathrm{trig} &= [B^{-1}]H_{i_0}^\mathrm{trig} + \sum_{\alpha\in{\Phi^+}}\frac{1}{q^\alpha-1}\alpha_{i_0} t_\alpha\\
        &= [B^{-1}]H_{i_0}^\mathrm{trig} +  \sum_{\substack{\alpha\in{\Phi^+}\\\alpha_{i_0}\neq 0}}\left(\frac{1}{p_\alpha((z_{C_i})_{i=1}^k)}\cdot\frac{\left(\prod_{C_\alpha\subset C_j\subset C_{i_0}}z_j\right)}{\left(\prod_{C_j\subset C_{i_0}}z_j\right)}\right)\alpha_{i_0} t_\alpha.
    \end{align*}
    We extend this formula to $\mathcal{V}_\mathscr{S}$ by multiplying by $\prod_{C_j\subset C_{i_0}}z_j$:  
    \begin{align*}
        \left(\prod_{C_j\subset C_{i_0}}z_j\right) (H_B)_{i_0}^\mathrm{trig} &= \left(\prod_{C_j\subset C_{i_0}}z_j\right)[B^{-1}]H_{i_0}^\mathrm{trig} + \sum_{\substack{\alpha\in{\Phi^+}\\\alpha_{i_0}\neq 0}}\frac{\left(\prod_{C_\alpha\subset C_j\subset C_{i_0}}z_j\right)}{p_\alpha((z_{C_{i'}})_{i'=1}^k)}\alpha_{i_0} t_\alpha
    \end{align*}
    \begin{align*}
        \left(\prod_{C_j\subset C_{i_0}}z_j\right) (H_B)_{i_0}^\mathrm{trig} &=  t_{\beta_{i_0}} + \left(\prod_{C_j\subset C_{i_0}}z_j\right)[B^{-1}]H_{i_0}^\mathrm{trig}+\sum_{\substack{\alpha\in{\Phi^+}\\\alpha_{i_0}\neq 0\\\alpha\neq \beta_{i_0}}}\frac{\left(\prod_{C_\alpha\subset C_j\subset C_{i_0}}z_j\right)}{p_\alpha((z_{C_{i'}})_{i'=1}^k)}\alpha_{i_0} t_\alpha.
    \end{align*}    
    for all $i_0\in\{1,\ldots,k\}$.  Since the $\{t_{\beta_i}\mid i=1,\ldots,k\}$ are linearly independent, it follows that the map 
    \begin{align*}
        (Q')_{\Phi^+}^\mathscr{S}\colon \mathcal{V}_\mathscr{S}&\longrightarrow \mathrm{Gr}(k,\mathfrak{u}_{\Phi^+}^1)\\
        (z_i)_{i=1}^k&\longmapsto \mathrm{Span}\left\{\left(\prod_{C_j\subset C_i}z_j\right)\cdot v_i''\;\biggr\rvert\;C_i\in\mathscr{S}\right\}
    \end{align*}
    extends to the locus where $z_j = 0$, $j=1,\ldots,k$ and thus forms a well-defined map.  Moreover, the $Q_{\Phi^+}^\mathscr{S}$ glue together to give a well-defined map 
    \begin{align*}
        Q'_{\Phi^+}\colon Y_{\Phi^+} = \bigcup_{S\in\mathcal{M}}\mathcal{V}_\mathscr{S}\longrightarrow \mathrm{Gr}(k,\mathfrak{u}_{\Phi^+}^1),
    \end{align*}
    as needed.
\end{proof}

\begin{definition}\label{D3.3.1}
    In a manner analogous to above, let us consider a torus $T^k$, with a hyperplane arrangement $\{H_\alpha\subset T^k\mid \alpha\in{\Phi^+}\}$ for a finite set ${\Phi^+}$, along with a poset of irreducible layers $\cI(\Phi^+)$.  Then, for a positive integer $l\in\mathbb{Z}$, define the map
    \begin{align*}
        \psi\times \mathrm{id}_{\mathbb{A}^l}\;\colon\;((T^k)^\mathrm{reg}\times \mathbb{A}^l)&\longrightarrow ((T^k)\times \mathbb{A}^l)\times \prod_{C\in\cI(\Phi^+)}\mathbb{P}_C\\ 
        (t,x)&\longmapsto \left(\;(t,x)\times (\;[\alpha(t)-1\mid \alpha\in{\Phi^+}_C]\;)_{C\in\cI(\Phi^+)}\;\right).
    \end{align*}
    Then analogously to Theorem \ref{T3.2.1}, we have the following open cover decomposition:
    \begin{align*}
        \overline{(\psi\times\mathrm{id}_{\mathbb{A}^1})((T^k)^\mathrm{reg}\times \mathbb{A}^l)} = Z_{\Phi^+}\times\mathbb{A}^l = \bigcup_{\mathscr{S}\in\mathcal{M}}(\mathcal{V}_\mathscr{S}\times\mathbb{A}^l).
    \end{align*}
\end{definition}
\subsection{Extending to the Boundary Divisors of $X_\Sigma$}\label{S3.4}

In the previous two sections, we have considered wonderful models of arrangements inside the torus $T^k$.  In this section, we will now extend our previous arrangement of subvarieties to the subvarieties in the toric variety $X_\Sigma$ and consider wonderful models inside $X_\Sigma$.  We will then extend the map $Q$ to $X_\Sigma$ in Theorem \ref{T3.4.1}.  Throughout, we will follow the constructions given in~\cite{DCG18} and~\cite{L09}.  

\begin{definition} (5.1 of~\cite{L09})
    Let $X$ be a smooth variety and let $\cC_X = \left\{C_i\right\}$ be a finite collection of smooth subvarieties of $X$.  For each pair $C_i,C_j\in\cC_X$, we say that $C_i$ and $C_j$ intersects {\it cleanly} if $C_i\cap C_j$ is smooth at each point $x\in C_i\cap C_j$ and we have  
    \begin{align*}
        T_{C_i,x}\cap T_{C_j,x} = T_{C_i\cap C_j,x}.
    \end{align*}
\end{definition}

\begin{definition} (c.f. Definition 2.1, 2.12, 5.4 of~\cite{L09}, Definition 2.1, 2.4, Theorem 2.1 of~\cite{DCG18}) \label{D3.4.1}
    Let $U\hookrightarrow \mathbb{A}^k$ be a smooth subvariety and let $\cC_U = \left\{C_i\right\}$ be a finite set of smooth subvarieties in $U$.  We say that $(U,\cC_U)$ is {\it a simple arrangement of subvarieties of $U$} if it satisfies the following conditions:
    \begin{enumerate}
        \item $C_i$ and $C_j$ intersect cleanly for every pair $i,j$. 
        \item For every $i,j$, we have $C_i\cap C_j = \bigsqcup_k C_k$ or $C_i\cap C_j = \emptyset$.   
    \end{enumerate}
    
    For such $(U,\cC_U)$, define a poset structure by inclusion.  A subposet $\cG_U\subset \cC_U$ is said to be a {\it building set of $\cC_U$} if it satisfies the condition that for each $C_i\in \cC_U$, the minimal elements of $\{G\in \cG_U\mid G\supset C_i\}$ intersect transversely and their intersection is $C_i$.  The minimal elements are then said to be the {\it $\cG$--factors} of $C_i$.  A subposet $\cG_U'\subset \cC_U$, is said to be a {\it building set} if the poset $\cC_U'$, formed by taking the collection of all possible intersections of collections of subvarieties in $\cG_U'$, is a simple arrangement of subvarieties of $U$, and if $\cG_U'$ is the building set of $\cC_U'$.  In this case, we say that $\cC_U'$ is the {\it induced arrangement of $\cG_U'$}.  
    
    Given the data of $(U,\cC_U,\cG_U)$, where $\cC_U$ is an arrangement of subvarieties of $U$ and $\cG_U$ is a building set of $\cC_U$, define the {\it blow-up of $U$ along $\cC_U$ with building set $\cG_U$} to be the space, denoted $\mathrm{Bl}\left(U,\cC_U,\cG_U\right)$, formed by the following construction:  enumerate the elements of $\cG_U = \{G_1,\ldots,G_N\}$ in the order so that the sets $\cG_{U,i} = \{G_1,\ldots,G_i\}$ are building sets of some $\cC_{U,i}$, respectively, and take the iterated blowup 
    \begin{align*}
        \mathrm{Bl}_{\widetilde{G_N}}\ldots\mathrm{Bl}_{\widetilde{G_2}}\mathrm{Bl}_{G_1}U,
    \end{align*}
    where $\widetilde{G_i}$ is the dominant transform of $G_i$ in $\mathrm{Bl}_{\widetilde{G_{i-1}}}\ldots\mathrm{Bl}_{\widetilde{G_2}}\mathrm{Bl}_{G_1}U$.  
\end{definition}

\begin{definition}(c.f. Definition 5.5 of~\cite{L09}, Definition 2.5 of~\cite{DCG18})\label{D3.4.2}
    Let $X = \bigcup_\sigma U_\sigma$ be an open covering of quasi-affine varieties and let $\cC_X = \left\{C_i\right\}$ be a finite set of smooth subvarieties in $X$.  We say that $(X,\cC_X)$ is {\it an arrangement of subvarieties of $X$} if for each open set $U_\sigma$ the pair $(U_\sigma,\cC_X\rvert_{U_\sigma})$ is a simple arrangement of subvarieties, in the sense of Definition \ref{D3.4.1}, where 
    \begin{align*}
        \cC_X\rvert_{U_\sigma} = \{C_i\cap U_\sigma\mid C_i\in\cC_X\}.
    \end{align*}  
    
    For such $(X,\cC_X)$, define a poset structure on $\cC_X$ by inclusion.  A subposet $\cG_X\subset \cC_X$ is said to be a {\it building set of $\cC_X$} if for each $U_\sigma$, the restriction  $\cG_X\rvert_{U_\sigma}\subset \cC_X\rvert_{U_\sigma}$ is a building set of $\cC_X\rvert_{U_\sigma}$ in the sense of Definition \ref{D3.4.1}.  A subposet $\cG_X'\subset \cC_X$ is said to be a {\it building set} if the poset $\cC_X'$, formed by taking the collection of all possible intersections of collections of subvarieties in $\cG_X'$, is an arrangement $\cC_X'$ of subvarieties of $X$, and if $\cG_X'$ is the building set of $\cC_X'$.  In this case, we say that $\cC_X'$ is the {\it induced arrangement of $\cG_X'$}.  
    
    Given the data of $(X,\cC_X,\cG_X)$, where $\cC_X$ is an arrangement of subvarieties of $X$ and $\cG_X\subset \cC_X$ is a building set of $\cC_X$, define the {\it blow-up of $X$ along $\cC_X$ with building set $\cG_X$} to be the space, denoted $\mathrm{Bl}\left(X,\cC_X,\cG_X\right)$, formed by the following construction:  enumerate the elements of $\cG_X = \{G_1,\ldots,G_N\}$ in the order so that the sets $\cG_{X,i} = \{G_1,\ldots,G_i\}$ are building sets for some $\cC_{X,i}$, respectively, and take the iterated blowup 
    \begin{align*}
        \mathrm{Bl}_{\widetilde{G_N}}\ldots\mathrm{Bl}_{\widetilde{G_2}}\mathrm{Bl}_{G_1}X,
    \end{align*}
    where $\widetilde{G_i}$ is the dominant transform of $G_i$ in $\mathrm{Bl}_{\widetilde{G_{i-1}}}\ldots\mathrm{Bl}_{\widetilde{G_2}}\mathrm{Bl}_{G_1}X$.  
\end{definition}

We construct open charts on the space defined by  De Concini and Gaiffi in~\cite{DCG18}, which is formed by blowing up a toric arrangement inside a given toric variety.  Let $\Phi = \Phi_+\sqcup\Phi_-\subset \mathbb{R}^k$ be a finite set and consider a fan $\Sigma$ in $\mathbb{R}^k$ which is cut out by the hyperplanes 
\begin{align*}
    H_\alpha = \{q\in T^k\mid  q^\alpha - 1 = 0\}
\end{align*}
for $\alpha\in\Phi$.  
We invoke the assumption in Section \ref{S3.1} that $\Sigma$ is regular, so that all extremal vectors of each cone form a basis of the ambient space $T^k$.  Then we have an embedding of 
\begin{align*}
    T^k\hookrightarrow X_\Sigma = \bigcup_{\sigma\in\Sigma(k)}U_\sigma,
\end{align*}
where each $U_\sigma$ is an open affine space.  Let $\cH_{X_\Sigma} = \{\overline{H_\alpha}\subset X_\Sigma\mid \alpha\in\Phi\}$.  Let $\cD_{X_\Sigma} = \{D_i\}$, where we have taken the boundary divisors $X_\Sigma\setminus T^k = \bigcup_i D_i$, which are irreducible components of the boundary of $X_\Sigma$.  

On each coordinate chart $U_\sigma$, the torus $T^k$ embeds into each open set as follows 
\begin{align*}
    T^k = \mathrm{ Spec }\mathbb{C}[q_1^\pm,\ldots,q_k^\pm]\hookrightarrow \mathrm{ Spec }\mathbb{C}[q_1,\ldots,q_k] =: U_\sigma \simeq \mathbb{A}^k,
\end{align*}
for some choice of coordinates $q_1,\ldots,q_k$.  The boundary divisors which intersect $U_\sigma$ are given by $\{D_{i_1},\ldots,D_{i_k}\}$, so
\begin{align*}
    U_\sigma\setminus T^k = \bigcup_{j=1}^k D_{i_j}^\sigma,    
\end{align*}
where each $D_{i_j}^\sigma$ are given by 
\begin{align*}
    D_{i_j}^\sigma = D_{i_j}\cap U_\sigma = \{q\in U_\sigma\mid q_{i_j} = 0\}.
\end{align*}
Let 
\begin{align*}
    \cD_{U_\sigma} = \{D_{i_j}^\sigma\mid j=1,\ldots,k\}.
\end{align*}
Taking the hypersurfaces
\begin{align*}
    \overline{H_\lambda}^\sigma = \overline{H_\lambda}\cap U_\sigma := \{q\in U_\sigma\mid q^\lambda = 1\},
\end{align*}
we take the hyperplane arrangement
\begin{align*}
    \cH_{U_\sigma}:= \left\{\overline{H_\lambda}^\sigma\mid \lambda\in\Phi\right\}. 
\end{align*}
and obtain the families $\cH_{X_\Sigma}\cup\cD_{X_\Sigma}$ and $\cH_{U_\sigma}\cup \cD_{U_\sigma}$ of codimension-$1$ subvarieties of $X_\Sigma$ and $U_\sigma$, respectively, for each $\sigma\in\Sigma(k)$.

Define $\cC_{X_\Sigma}$ to be the set of connected components of the intersections of the subvarieties in $\cH_{X_\Sigma}\cup\cD_{X_\Sigma}$ and for each $\sigma\in\Sigma(k)$, define $\cC_{U_\sigma}$ to be the set of connected components of the intersections of the subvarieties in $\cH_{U_\sigma}\cup\cD_{U_\sigma}$.  Define $\cG_{X_\Sigma}:= \cI_{X_\Sigma}$ to be minimal building set of $\cC_{X_\Sigma}$ consisting of the closures in $X_\Sigma$ of the irreducible layers in $T^k$, in the sense of Definition \ref{D3.2.5}.  Similarly for each $\sigma\in\Sigma(k)$, define $\cG_{U_\sigma}:= \cI_{U_\sigma}$ to be the minimal building set of $\cC_{U_\sigma}$ consisting of the closures in $U_\sigma$ of the irreducible layers in $T^k$, in the sense of Definition \ref{D3.2.5}.     

\begin{lemma} (Lemma 7.1 of~\cite{DCG18})

    The families $(X_\Sigma,\cC_{X_\Sigma})$ and $(U_\sigma,\cC_{U_\sigma})$, $\sigma\in\Sigma(k)$, each form arrangements of subvarieties in $X_\Sigma$ and $U_\sigma$, $\sigma\in\Sigma(k)$, respectively, according to Definition \ref{D3.4.2}.
\end{lemma} 

\begin{definition}\label{D3.4.3}
    Given $T^k$ and $\Phi$, as above, we define the {\it De Concini-Gaiffi compactification} of $T^\mathrm{reg}$ to be $\widetilde{X_\Sigma}:=\mathrm{Bl}(X_\Sigma,\cC_{X_\Sigma},\cI_{X_\Sigma})$, as defined in Definition \ref{D3.4.2}. 
\end{definition}

\begin{remark}\label{R1}
    Given $\sigma\in\Sigma(k)$ with corresponding open set $U_\sigma = \mathrm{Spec}\mathbb{C}[\sigma^\vee\cap X^*(T)]$, the extremal vectors $\{\beta_i\mid i=1,\ldots,k\}$ defining $\sigma$ turn out to be the normal vectors $\lambda\in\Phi\subset X^*(T)$ to the hyperplanes forming $\Sigma$.  Since $\Sigma$ is assumed to be a regular fan, then according to Section 8 of~\cite{DCG18}, this fan satisfies Property (E), according to the notation of Section 8 and thus, for every $\lambda\in \Phi$, we can write $\lambda = \sum_{i=1}^k \lambda_i \beta_i$, where $\lambda_i\geq 0$.  Thus, the hyperplanes $\overline{H_\lambda}^\sigma\in\cH_{U_\sigma}$ are of the form $q^\lambda = q_1^{\lambda_1}\cdot\ldots\cdot q_k^{\lambda_k} = 1$, where $\lambda_i\geq 0$ for all $i$.  
\end{remark}

\begin{definition}
    Let $S\in\mathcal{P}([k]) = \mathcal{P}(\{1,\ldots,k\})$ and let $\sigma\in\Sigma(k)$.  Define
    \begin{align*}
        (\cH_{U_\sigma})_S = \left\{\overline{H_\lambda}^\sigma\in\cH_{U_\sigma}\mid \overline{H_\lambda}^\sigma\cap\bigcap_{i\in [k]\setminus S}D_i^\sigma\neq \emptyset\right\};
    \end{align*}
    define $(\cC_{U_\sigma})_S$ to be the poset formed by the irreducible components of intersections of subvarieties from $(\cH_{U_\sigma})_S$; and define $(\cI_{U_\sigma})_S$ to be the poset formed by the minimal building set of $(\cC_{U_\sigma})_S$.   
\end{definition}

\begin{remark}
    The equations for $\overline{H_\lambda}^\sigma$ are given by $q^\lambda = q_1^{\lambda_1}\cdot\ldots\cdot q_k^{\lambda_k} = 1$ in $U_\sigma$.  Thus, 
    \begin{align*}
        (\cH_{U_\sigma})_S 
        &= \left\{\lambda\in\Phi\mid \lambda_i=0\;\;\forall i\in [k]\setminus S \right\}\\
        &= \left\{\lambda\in\Phi\mid \mathrm{Supp}(\lambda)\subset S\right\},
    \end{align*}
    where $\mathrm{Supp}(\lambda) = \{i\in[k]\mid \lambda_i\neq 0\}$.  It follows that the equations for each $\overline{H_\lambda}^\sigma\in(\cH_{U_\sigma})_S$ only depend on $i\in S$, for each $S\subset[k]$.  Defining
    \begin{align*}
        p_{S}\;\colon \; U_\sigma\setminus\bigcup_{i\in S}D_i^\sigma\simeq (T^1)^{S}\times(\mathbb{A}^1)^{[k]\setminus S}\longrightarrow (T^1)^{S},
    \end{align*}
    to the the projection map, it follows that $\overline{H_\lambda}^\sigma\in(\cH_{U_\sigma})_S$ if and only if we can write $\overline{H_\lambda}^\sigma = p_{S}^{-1}((\overline{H_\lambda}^\sigma)_{S})$ for some $(\overline{H_\lambda}^\sigma)_{S}\subset (T^1)^{S}$.  The previous definitions of irreducible components of intersections and building sets is also compatible with the projection to $(T^1)^S$.     
\end{remark}

\begin{lemma}\label{L3.4.1}
    Assume that for each $\sigma\in\Sigma(k)$, $U = U_\sigma\simeq \mathbb{A}^k$ in Definition \ref{D3.4.2}.  Define 
    \begin{align*}
        (U_\sigma)_S := \left(U_\sigma\setminus\bigcup_{i\in S}D_i^\sigma\right)\setminus\left(\bigcup_{\cH_{U_\sigma}\setminus(\cH_{U_\sigma})_S}\overline{H_\lambda}^\sigma\right).
    \end{align*}
    Then we have the following:
    \begin{align*}
        \mathrm{Bl}\left(U_\sigma,\cC_{U_\sigma},\cI_{U_\sigma}\right) 
        = 
        \bigcup_{S\in\mathcal{P}([k])}\mathrm{Bl}\left((U_\sigma)_S,\;(\cC_{U_\sigma})_S\rvert_{(U_\sigma)_S},\;(\cI_{U_\sigma})_S\rvert_{(U_\sigma)_S}\;\right)
    \end{align*}
\end{lemma}

\begin{proof}
    For each subset $S\subset [k]$, we have 
    \begin{align*}
        \left(U_\sigma\cap \bigcap_{i\in [k]\setminus S}D_i^\sigma\right)\setminus \bigcup_{i\in S}D_i^\sigma 
        \subset (U_\sigma)_S 
    \end{align*}
    forms an open neighbourhood and that 
    \begin{align*}
        U_\sigma 
        &= \coprod_{S\subset [k]}\left(\left(U_\sigma\cap\bigcap_{i\in [k]\setminus S}D_i^\sigma\right)\setminus \bigcup_{i\in S}D_i^\sigma\right)\\
        &= \bigcup_{S\subset [k]}(U_\sigma)_S.
    \end{align*}
    We observe that each $(U_\sigma)_S$ is Zariski-dense in $U_\sigma$ and that there is a Zariski-dense copy of $T^k$ in each open set $(U_\sigma)_S$.  Thus, the formation of irreducible components of intersections and the formation of building sets, in the sense of Definition \ref{D3.4.2}, are both compatible with the restrictions $T^k\subset (U_\sigma)_S\subset U_\sigma\subset X_\Sigma$ (c.f. Definition 2.5 of~\cite{DCG18}).  This completes the proof.    
\end{proof}

\begin{remark}
    We observe that if $S = \emptyset$, then $U_\sigma\setminus\bigcup_{i\in S}D_i^\sigma = \mathbb{A}^k$, and $(\cH_{U_\sigma})_S = \left\{\overline{H_\lambda}^\sigma\cap \bigcap_{i=1}^k D_i^\sigma\neq \emptyset\right\} = \emptyset$, so we have 
    \begin{align*}
        \left(\; (U_\sigma)_S,\;(\cC_{U_\sigma})_S,\;(\cI_{U_\sigma})_S\;\right) 
        = \left(\mathbb{A}^k\setminus\bigcup_{\cH_{U_\sigma}}\overline{H_\lambda}^\sigma,\;\emptyset,\;\emptyset \right).
    \end{align*}
    On the other hand, if $S = [k]$, then $U_\sigma\setminus\bigcup_{i\in S}D_i^\sigma = T^k$ and $(\cH_{U_\sigma})_S = \left\{\overline{H_\lambda}^\sigma\neq\emptyset\right\} = \cH_{U_\sigma}$, so we have  
    \begin{align*}
        \left(\;(U_\sigma)_S,\;(\cC_{U_\sigma})_S,\; (\cI_{U_\sigma})_S\;\right)
        = \left(T^k,\cC_{U_\sigma}, \cI_{U_\sigma}\right). 
    \end{align*}
\end{remark}

\begin{definition}
    Given an arrangement of subvarieties $(X,\cC_X)$, let $(\cC_X)_0$ be the $0$--dimensional intersections of subvarieties in $\cC_X$.
\end{definition}

\begin{theorem}\label{T3.4.1}
    Let $\Sigma$ be a regular fan and $X_\Sigma$ the corresponding toric variety with open set decomposition $X_\Sigma = \bigcup_{\sigma\in\Sigma}U_\sigma$.  Then we have the following chart decomposition:
    \begin{align*}
        \mathrm{Bl}(X_\Sigma,\cC_{X_\Sigma},\cI_{X_\Sigma}) 
        &= \bigcup_{\sigma\in\Sigma}\mathrm{Bl}(U_\sigma,\cC_{U_\sigma},\cI_{U_\sigma})\\
        &=
        \bigcup_{\sigma\in\Sigma}\bigcup_{S\in\mathcal{P}([k])}\mathrm{Bl}\left((U_\sigma)_S,\;(\cC_{U_\sigma})_S\rvert_{(U_\sigma)_S},\;(\cI_{U_\sigma})_S\rvert_{(U_\sigma)_S}\right)\\
        &= \bigcup_{\sigma\in\Sigma}\bigcup_{S\in\mathcal{P}([k])}\bigcup_{\mathscr{S}\in \mathcal{M}_S^\sigma}(\mathcal{W}_\mathscr{S}^\sigma)_S,
    \end{align*}
    where $\mathcal{M}_\mathscr{S}^\sigma$ are the maximal nested sets in the toric arrangement $((T^1)^S,(\cH_{U_\sigma})_\mathscr{S})$ (c.f. Definition \ref{D3.2.6}) and where $(\mathcal{W}_\mathscr{S}^\sigma)_S$ fits into the following fibre diagram:
    \begin{center}
        % https://tikzcd.yichuanshen.de/#N4Igdg9gJgpgziAXAbVABwnAlgFyxMJZABgBpiBdUkANwEMAbAVxiRGAB0OBbOnACwDGjYADUAvgH0uvAcIbAAyuPEByEONLpMufIRQBmclVqMWbGXyEiJ0nlflLxXPN3iWBAI0-AAguIA9YGQAawouOBgcbiwwJjgAAmUNLRAMbDwCIjIARhN6ZlZEEBcYAA8cYAAhBmcOBhgAMxwACgAVIOUXLDc4D35vP0DgsIiomLjE5VJ+xwAJQIisAHNeSUUuACcV-hwASnVNbQy9IiM86gLzYtKK6tquBub2zrrXd3svH38g0PCOSLRWLxJKaWYiBYBJarOjrLY7fYpY66LIoMgAJnyZiKJXqTVaHQ2HHefU+A2+wz+YyBk1B8OWuz21Im8Ue+JaXE8K0ETDQkk4ZPmiwBK14zOBcA5gohwuwML26zqEBoME2DFiMGAczsDDo3E8UDostFdHpjKRaR0mX0yCMmMu2LYhO6vX6gx+I3+gJZU3EGhMMCgy3gRFAjU2EG4SDIIBwECQOQdhSQcQYDGous8MAYAAUradik8cBbw5Ho9Q40h0Unrqn0yBM9m8ydUQ38SWI1HEInY-HEEZTMnEHWM3Qs7n862ix2y4hq72kAAWGtFEcNsdNyf6NvNGdd+eV-srlNMNOj8fNlHb6dHEClrsDw8AVmPw9P9cbE5b1-bt-vS4rPsX0HWt33PTdvzYG8KHEIA
\begin{tikzcd}
(\mathcal{W}_\mathscr{S}^\sigma)_S \arrow[rrr] \arrow[d]                                                                                                                            &  &  & {\mathcal{V}_\mathscr{S}\times\mathbb{A}^{[k]\setminus S}} \arrow[d]                          \\
{\mathrm{Bl}\left(T^{S}\times\mathbb{A}^{[k]\setminus S},\;(\cC_{U_\sigma})_S,\; (\cI_{U_\sigma})_S\right)'} \arrow[d] \arrow[rrr]                                                                  &  &  & {\mathrm{Bl}\left(T^{S}\times\mathbb{A}^{[k]\setminus S},\;(\cC_{U_\sigma})_S,\; (\cI_{U_\sigma})_S\right)} \arrow[d] \\
{\left(T^S\times\mathbb{A}^{[k]\setminus S}\right)\setminus\left(\bigcup_{\cH_{U_\sigma}\setminus(\cH_{U_\sigma})_S}\overline{H_\lambda}^\sigma\right)} \arrow[rrr] &  &  & {T^S\times\mathbb{A}^{[k]\setminus S}},                                                       
\end{tikzcd}
    \end{center} 
    with $\mathcal{V}_\mathscr{S}\subset \mathrm{Bl}(T^k,\;\cC_{T^k},\;\cI_{T^k})$ a maximal nested set, as in Theorem \ref{T3.2.1}.
\end{theorem}

\begin{proof}
    This follows from a direct application of Definition \ref{D3.4.3} and Lemma \ref{L3.4.1}.
\end{proof}

\begin{theorem}
    We have the following extension diagram:
    \begin{center}
        % https://tikzcd.yichuanshen.de/#N4Igdg9gJgpgziAXAbVABwnAlgFyxMJZABgBpiBdUkANwEMAbAVxiRABUA9YAHR5xgAPHMABOMAOYBfKSCml0mXPkIoATOSq1GLNnwHDgAcVFSAFAGsAtAEZSfALZ0cACwBmouheBMpAfT4ABRcsThsASjkFEAxsPAIiMhstemZWRBAATQCeYKw5LRgoCXgiUA8IByQyEBwIJDV5ctFK6uo6pDttNLYARSjm1sQNWvrELoYsMHSQKDo4FyKCqSA
\begin{tikzcd}
T^{\mathrm{reg}} \arrow[d] \arrow[rr, "Q'"] &  & {\mathrm{Gr}(k,\mathfrak{u}_{\Phi^+}^1)} \\
\widetilde{X_\Sigma} \arrow[rru, dashed]               &  &                                     
\end{tikzcd}
    \end{center}
\end{theorem}

\begin{proof}
    Recall the map 
    \begin{align*}
        Q'\colon T^{\mathrm{reg}}&\longrightarrow \mathrm{Gr}(k,\mathfrak{u}_{\Phi^+}^1)\\
    q&\longmapsto \mathrm{Span}_\mathbb{C}\{H_i^\mathrm{trig}(q)\mid i=1,\ldots,k\}\\&\qquad\; = \mathrm{Span}_\mathbb{C}\left\{u_i +\sum_{\alpha\in\Phi^+}\frac{q^\alpha}{1-q^\alpha}\alpha_i t_\alpha\mid i=1,\ldots,k\right\}.
\end{align*}
    As in Theorem \ref{T3.4.1}, fix a $\sigma\in\Sigma$, $S\in\mathcal{P}([k])$ and $\mathscr{S}\in \mathcal{M}_S^\sigma$, In partitioning 
    \begin{align*}
        \cH = (\cH_{U_\sigma})_S\sqcup\left(\cH_{U_\sigma}\setminus(\cH_{U_\sigma})_S\right),
    \end{align*}
    so that
    \begin{align*}
        \Phi^+ = (\Phi^+)_S\sqcup (\Phi^+)_{S}^\mathsf{c},
    \end{align*}
    where 
    \begin{align*}
        \Phi_{S} := \{\lambda\in \Phi^+\mid \mathrm{Supp}(\lambda)\subset S\},
    \end{align*}
    we can write each term as follows:
    \begin{align*}
        H_i^\mathrm{trig}(q) &= \left(u_i + \sum_{\alpha\in(\Phi^+)_S^c}\frac{q^\alpha}{1-q^\alpha}\alpha_it_\alpha\right) + \sum_{\alpha\in(\Phi^+)_S}\frac{q^\alpha}{1-q^\alpha}\alpha_it_\alpha.
    \end{align*}
    Setting 
    \begin{align*}
        u_i' := u_i + \hbar\sum_{\alpha\in(\Phi^+)_S^c}\frac{q^\alpha}{1-q^\alpha}\alpha_it_\alpha,
    \end{align*}
    we get 
    \begin{align*}
        H_i^\mathrm{trig}(q) = u'_i + \hbar\sum_{\alpha\in(\Phi^+)_S}\frac{q^\alpha}{1-q^\alpha}\alpha_it_\alpha,
    \end{align*}
    for each $i=1,\ldots,k$.  Then, we can extend $\mathrm{Span}_\mathbb{C}\{H_i^\mathrm{trig}(q)\}$ across each $(\mathcal{W}_\mathscr{S}^\sigma)_S$, using the proof of Theorem \ref{T3.3.2}.  
\end{proof}

\begin{corollary} (Main Theorem \ref{T2})
    Let $X$ be a hypertoric variety with an action by $G = T^d\times\mathbb{C}^\mathsf{x}$ and K\"{a}hler torus $T^k$, and assume that the toric variety $X_\Sigma$ formed from the discriminantal arrangement of $X$ is smooth.  Let $\widetilde{X_\Sigma}$ be the De Concini-Gaiffi compactification of $T^\mathrm{reg}$.  Then the map 
    \begin{align*}
        Q\colon T^\mathrm{reg}&\longrightarrow \mathrm{Gr}(n+1,E)\\
        q&\longmapsto \{u\star_q(-)\mid u\in H_G^2(X)\}
    \end{align*}
    admits an extension to $\widetilde{X_\Sigma}$, in the sense that the following diagram commutes:
    \begin{center}
    % https://tikzcd.yichuanshen.de/#N4Igdg9gJgpgziAXAbVABwnAlgFyxMJZABgBpiBdUkANwEMAbAVxiRABUA9AHW5xgAeOYACcYAcwC+ISaXSZc+QigBM5KrUYs2vfkOABxEZIAUYAOSkAogEoZckBmx4CRMgEYN9Zq0QheAO5YsHgMsMAAGgD6vADKWOIAtnSy9vLOSkRqntTe2n66gsJGphakgcEwoeG8yTgAFgBmInQA1sBMklHAvAAK9VicANSSkpzudpIaMFDi8ESgzRCJSGQgOBBI7rlaviAAiiDUDHQARjAMvQouyiAMMI04aSBLK4hrG0hqmj5IYEwMBjHM4XK4ZVx+e6PZ6vVbUT6IADMO1+iH+gOB50u10ykIeT2OWDAeygdDg9RmMJEyy+8M2SJR+XRQLuIOx4NuUIJdyJJLJFKgVJpDPW9O2PyZAJZJyxYMUELu+JkFEkQA
\begin{tikzcd}
T^\mathrm{reg} \arrow[rr, "Q"] \arrow[d] \arrow[rrd, "Q'" description, dashed] &  & {\mathrm{Gr}(n+1,E)}                                             \\
{\widetilde{X_\Sigma}} \arrow[rr, dashed]                 &  & {\mathrm{Gr}(k,\mathfrak{u}_{\Phi^+}^1)}. \arrow[u]
\end{tikzcd}
\end{center} 
\end{corollary}

\bibliographystyle{plain}
\bibliography{sources.bib}

\end{document}